\documentclass[11pt]{article}

\usepackage{bbold}
\usepackage{amsmath,amsthm,amssymb,amsfonts}
\usepackage{bm}
\usepackage{colonequals}
\usepackage{fullpage}
\usepackage{dsfont}
\usepackage{graphicx}
\usepackage{xcolor}
\usepackage{hyperref}
\usepackage{comment}

\renewcommand{\emptyset}{\varnothing}

\def\dunderline#1{\underline{\underline{#1}}}

\newcommand{\Erdos}{Erd\H{o}s}
\newcommand{\Renyi}{R\'{e}nyi}

\newcommand{\Aut}{\mathrm{Aut}}
\newcommand{\eAut}{\mathrm{eAut}}

\newcommand{\Frob}{\mathrm{Frob}}
\newcommand{\conn}{\mathrm{conn}}

\newcommand{\Unif}{\mathrm{Unif}}
\newcommand{\Gin}{\mathrm{Gin}}

\newcommand{\chop}{\mathrm{chop}}

\newcommand{\eps}{\epsilon}
\newcommand{\e}{\mathrm{e}}

\newcommand{\id}{\mathds{1}}

\DeclareMathOperator*{\Exp}{\mathbb{E}}
\newcommand{\Var}{\mathop{\mathrm{Var}}}
\newcommand{\tr}{\mathop{\mathrm{tr}}}
\newcommand{\Ex}{\mathop{\mathbb{E}}}
\newcommand{\Varx}{\mathop{\mathrm{Var}}}
\newcommand{\Cat}{\mathrm{Cat}}

\newcommand{\sF}{\mathcal{F}}
\newcommand{\sG}{\mathcal{G}}
\newcommand{\sN}{\mathcal{N}}
\newcommand{\sO}{\mathcal{O}}
\newcommand{\sM}{\mathcal{M}}

\newcommand{\One}{\bm{1}}
\newcommand{\EE}{\mathbb{E}}
\newcommand{\NN}{\mathbb{N}}
\newcommand{\RR}{\mathbb{R}}
\newcommand{\PP}{\mathbb{P}}
\newcommand{\QQ}{\mathbb{Q}}
\renewcommand{\SS}{\mathbb{S}}

\newcommand{\Wg}{\mathrm{Wg}}
\newcommand{\Sym}{\mathrm{Sym}}
\newcommand{\Wig}{\mathrm{Wig}}

\newcommand{\Haar}{\mathrm{Haar}}
\newcommand{\Adv}{\mathrm{Adv}}
\newcommand{\Corr}{\mathrm{Corr}}
\newcommand{\MMSE}{\mathrm{MMSE}}

\newcommand{\switch}{\mathrm{switch}}
\newcommand{\circuit}{\mathrm{circ}}
\newcommand{\cyc}{\mathrm{cyc}}

\newcommand{\cupcap}{\raisebox{-9pt}{\includegraphics[height=24pt]{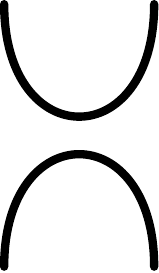}}}

\newcommand{\ccup}
{\raisebox{-3pt}{\includegraphics[height=12pt]{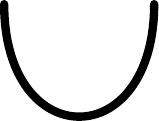}}}

\newcommand{\cupcaploop}{\raisebox{-9pt}{\includegraphics[height=24pt]{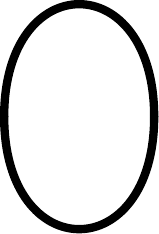}}}

\newcommand{\embed}[1]{\raisebox{-9pt}{\includegraphics[height=24pt]{#1}}}

\newcommand{\what}{\widehat}

\newtheorem{theorem}{Theorem}[section]
\newtheorem{remark}[theorem]{Remark}

\newtheorem{lemma}[theorem]{Lemma}

\newtheorem{definition}[theorem]{Definition}
\newtheorem{example}[theorem]{Example}
\newtheorem{proposition}[theorem]{Proposition}
\newtheorem{corollary}[theorem]{Corollary}

\title{Tensor cumulants for statistical inference on invariant distributions}

\usepackage{authblk}

\author[1]{Dmitriy Kunisky\thanks{Email: \texttt{dmitriy.kunisky@yale.edu}. Partially supported by ONR Award N00014-20-1-2335 and a Simons Investigator Award to Daniel Spielman.}}

\author[2]{Cristopher Moore\thanks{Email: \texttt{moore@santafe.edu}. Partially supported by the National Science Foundation through grant BIGDATA-1838251.}}

\author[3]{Alexander S.\ Wein\thanks{Email: \texttt{aswein@ucdavis.edu}. Partially supported by an Alfred P.\ Sloan Research Fellowship and NSF CAREER Award CCF-2338091.}}

\affil[1]{Department of Computer Science, Yale University}

\affil[2]{Santa Fe Institute}

\affil[3]{Department of Mathematics, UC Davis}

\date{April 29, 2024}

\begin{document}

\maketitle

\begin{abstract}
Many problems in high-dimensional statistics appear to have a \emph{statistical-computational gap}: a range of values of the signal-to-noise ratio where inference is information-theoretically possible, but (conjecturally) computationally intractable.
A canonical such problem is Tensor PCA, where we observe a tensor $Y$ consisting of a rank-one signal plus Gaussian noise.
Multiple lines of work suggest that Tensor PCA becomes computationally hard at a critical value of the signal's magnitude.
In particular, below this transition, no low-degree polynomial algorithm can detect the signal with high probability; conversely, various spectral algorithms are known to succeed above this transition.
We unify and extend this work by considering \emph{tensor networks}, orthogonally invariant polynomials where multiple copies of $Y$ are ``contracted'' to produce scalars, vectors, matrices, or other tensors.
We define a new set of objects, \emph{tensor cumulants}, which provide an explicit, near-orthogonal basis for invariant polynomials of a given degree.
This basis lets us unify and strengthen previous results on low-degree hardness, giving a combinatorial explanation of the hardness transition and of a continuum of subexponential-time algorithms that work below it, and proving tight lower bounds against low-degree polynomials for recovering rather than just detecting the signal.
It also lets us analyze a new problem of distinguishing between different tensor ensembles, such as Wigner and Wishart tensors, establishing a sharp computational threshold and giving evidence of a new statistical-computational gap in the Central Limit Theorem for random tensors.
Finally, we believe these cumulants are valuable mathematical objects in their own right: they generalize the free cumulants of free probability theory from matrices to tensors, and share many of their properties, including additivity under additive free convolution.
\end{abstract}

\thispagestyle{empty}

\clearpage

\tableofcontents

\thispagestyle{empty}

\clearpage

\setcounter{page}{1}
\pagestyle{plain}

\section{Introduction}

We will study statistical problems formulated over \emph{tensors}.
Here a tensor $T$ is a $p$-dimensional table of real numbers, indexed as $T_{i_1,i_2,\ldots,i_p}$ with $i_a \in \{1, \dots, n\} \equalscolon [n]$. We call $p$ the \emph{arity} of $T$, say that $T$ is a \emph{$p$-ary tensor}, and call $n$ the \emph{dimension} of $T$. A 1-ary tensor is an $n$-dimensional vector. A 2-ary tensor can be viewed as an $n \times n$ matrix or as an $n^2$-dimensional vector, and so on. A 0-ary tensor is a scalar, i.e., a real number.
We say that $T$ is \emph{symmetric} if $T_{i_1,\ldots,i_p} = T_{\sigma(i_1),\ldots,\sigma(i_p)}$ for any permutation $\sigma \in S_p$. In this case, each entry of $T$ is specified by a multiset in $[n]$ of size $p$. The vector space of symmetric tensors of dimension $n$ and arity $p$ is denoted $\Sym^p (\RR^n)$.

Many statistical problems of broad interest are modeled with tensor observations, representing collections of estimates of degree $p$ moments of high-dimensional random variables or $p$-way interaction data such as hypergraphs and generalizations thereof.
Often, whether in
\emph{hypothesis testing} (distinguishing two distributions over tensors) or in \emph{estimation} (recovering some parameter from a noisy tensor observation) problems, the distributions of tensors involved are reasonably assumed to be \emph{orthogonally invariant}, that is, invariant to simultaneous changes of basis of all tensor ``axes.''
\begin{definition}[Change of basis]
    \label{def:change-basis}
    For $T$ a tensor of dimension $n$ and $Q$ a (usually but not necessarily orthogonal) matrix, we define the tensor $Q \cdot T$ to have entries
    \begin{equation}
        (Q \cdot T)_{j_1,\ldots,j_p} =
\sum_{i_1,\ldots,i_p}  T_{i_1,\ldots,i_p}\prod_{t=1}^p Q_{i_t,j_t}.
    \end{equation}
    Alternatively, we may view $Q \cdot T = (Q^{\top})^{\otimes p} \,T$ where $T$ is viewed as a vector in $\RR^{n^p}$.
    This makes the action of the orthogonal group $\sO(n)$ or the general linear group $\mathrm{GL}(n)$ on $(\RR^n)^{\otimes p}$ a \emph{right} group action: $R \cdot (Q \cdot T) = (QR) \cdot T$, which preserves symmetry and thus is also an action on $\Sym^p(\RR^n)$.
\end{definition}
\noindent
For $T$ a matrix one may check that $Q \cdot T = Q^{\top} T Q$ coincides with the usual change of basis.

\begin{definition}[Invariance]
    A function $f: \Sym^p(\RR^n) \to \RR$ is \emph{(orthogonally) invariant} if $f(Q \cdot T) = f(T)$ for all $T \in \Sym^p(\RR^n)$ and all $Q \in \sO(n)$.
    The law of a random $T$ is \emph{(orthogonally) invariant} if $Q \cdot T$ has the same law for all orthogonal $Q \in \sO(n)$.
\end{definition}

\emph{Tensor networks} are a graphical notation that extends linear algebra to tensors, generalizing operations like the inner product of vectors, matrix products, and traces.
This gives a powerful language for discussing quantities that would be quite tricky to express in conventional notation (as we already see above for changes of basis).
We will discuss general tensor networks later in Section~\ref{sec:prelim:tensor-networks}, but we mention now that, aside from changes of basis, the following is an important class of computations that tensor networks can express.

\begin{definition}[Graph moments]
\label{def:moments}
    Let $T \in \Sym^p(\RR^n)$ and $G = (V, E)$ a $p$-regular graph (not necessarily simple).
    Then, we define the \emph{$G$-moment} of $T$ as
    \begin{equation}
        m_G(T) = \sum_{i \in [n]^E} \prod_{v \in V} T_{i(\partial v)} \, .
    \end{equation}
\noindent
where $i(\partial v)$ denotes the multiset of indices $i(e)$ associated with the edges $e$ incident to $v$. If $G$ is the empty graph with $V=E=\emptyset$, then $m_G(T)=1$.
\end{definition}
\noindent
When $T$ is symmetric---as it always is here---we do not need to order the edges incident to a vertex.\footnote{If we considered tensors with complex entries the edges would need to be directed, since the complex inner product has $u^* v \neq v^* u$ in general. For us, $T$'s entries will always be real, so undirected edges suffice.}

\begin{figure}
    \centering
    \includegraphics[width=1.4in]{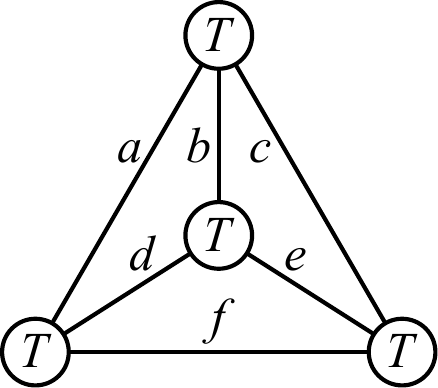}
    \caption{The graph moment $m_G(T)$ where $T$ is a symmetric 3-ary tensor and $G=K_4$. Summing over all six indices, one on each edge, contracts the graph and yields the  scalar~\eqref{eq:k4}.
    }
    \label{fig:k4}
\end{figure}
\begin{example}
Let $G=K_4$ be the complete graph on four vertices, as shown in Figure~\ref{fig:k4}. Then
\begin{equation}
    \label{eq:k4}
m_G(T) = \sum_{a,b,c,d,e,f \in [n]} T_{abc} \,T_{adf} \,T_{bde} \,T_{cef} \, ,
\end{equation}
Note that this is a homogeneous polynomial, of degree 4, of $T$'s entries.
In general, $m_G(T)$ is a homogeneous polynomial, of degree $|V(G)|$, of $T$'s entries.
\end{example}

Linear combinations of the tensor networks $m_G(T)$ are a natural class of scalar functions of a tensor.
In particular, as we will show
in Proposition~\ref{prop:moments-are-invariant}, they are themselves invariant: $m_G(Q \cdot T) = m_G(T)$ for all $Q \in \sO(n)$. Moreover, by classical invariant theory, they span \emph{all} invariant polynomials.
As we will discuss in Section~\ref{sec:prelim:tensor-networks}, these linear combinations should therefore be thought of as the correct ``space of spectral algorithms'' for computing scalar statistics from tensors, analogous to the linear spectral statistics $T \mapsto \tr(f(T))$ of a matrix $T$.

We wish to draw a connection between these algebraic facts and the substantial body of literature that has developed recently around using \emph{low-degree polynomial algorithms} for tasks like hypothesis testing and estimation.
We leave a detailed discussion to Section~\ref{sec:prelim:low-deg}, but simple instances of this approach are to perform hypothesis testing between two distributions of tensors by computing and thresholding a low-degree polynomial of the observed tensor~\cite{HS-bayesian,sos-detect,hopkins-thesis}, or estimation by computing an estimator that consists of a vector of low-degree polynomials of the observation~\cite{SW-2020-LowDegreeEstimation}.
We view the degree of a polynomial as a proxy for the runtime of the associated algorithm (as, assuming we naively evaluate polynomials without taking advantage of any special structure, a degree $D$ polynomial in $n$ variables takes time $O(n^D)$ to evaluate), and so the degree of polynomial required for testing or estimation is taken as a measure of the problem's complexity.

The goal of this paper is to explore the consequences of the following general observation (formalized for low-degree algorithms in our Proposition~\ref{prop:adv-invariant}): for hypothesis testing tasks over invariant tensor distributions, the best low-degree polynomials are invariant themselves, and therefore are linear combinations of graph moments $m_G(T)$. Similarly, the best low-degree estimators are linear combinations of graph moments with a single ``open'' edge, which evaluate to vectors rather than scalars (Proposition~\ref{prop:corr-equivariant}).
In addition to giving us a set of tools for reasoning graphically about low-degree algorithms, this dramatically restricts the space of such algorithms for invariant problems, and therefore can also reduce the analytical difficulties of proving computational lower bounds.
Indeed, the dimension of the space of low-degree invariant polynomials, corresponding for degree $d$ to the number of $p$-regular multigraphs $G$ on $d$ vertices, does not depend on the tensor dimension $n$ at all!
Moreover, the space spanned by the $m_G(T)$ turns out to have a rich algebraic structure with important ramifications for the computational complexity of statistical problems.
For example, we will see that its dimension (the number of $p$-regular multigraphs of a given size) has a direct connection to the tradeoff between signal strength and subexponential runtime of algorithms for hypothesis testing and estimation.

We will develop a general theory about the spaces of invariant polynomials (and \emph{equivariant} vector-valued polynomials, to be used as estimators), extending classical invariant theory to be directly useful for the analysis of low-degree algorithms.
Surprisingly, a central role will be played by a notion of \emph{finite free cumulants} for tensors. These generalize aspects of the theory of free probability for random matrices, and yield an explicit near-orthogonal basis for the space of invariant polynomials. Using these tools, we will give new results for two examples of invariant statistical problems over tensors.

First, we will consider the well-studied problem of \emph{tensor PCA (principal component analysis)}, recovering and unifying previous results on hypothesis testing and giving a new and tight analysis of estimation.
Second, we will study a variation on the newer problem of \emph{hypothesis testing between Wigner and Wishart tensors}, adapted to our focus on invariant distributions.
Through one lens this is a relative of the task of \emph{tensor decomposition}, and through another it is a variation on the question of distinguishing \Erdos-\Renyi\ from geometric random hypergraphs (a hypergraph version of the problem treated for matrices by \cite{BDER-2016-WignerWishartDetection,BBN-2020-PhaseTransitionsLatentGeometry,LMSY-2022-TestingRandomGeometric,BB-2024-FourierRandomGeometricGraphs}). Through yet a third it is a question about a ``computational central limit theorem,'' a class of questions to which our results will apply more generally (a closely related tensor problem was studied by \cite{Mikulincer-2020-CLTWishartTensors}, and similar matrix problems by \cite{BDER-2016-WignerWishartDetection,BBH-2021-DeFinettiWishartMatrices}).

Below we first focus on these concrete applications, and then sketch some of the technical ideas before giving full details in Section~\ref{sec:cumulants}.

\subsection{Main Results}

\paragraph{Tensor PCA}
Our first subject will be a model of random tensors built as follows:
\begin{equation}
\label{eq:spiked-model}
Y = \lambda v^{\otimes p} + W \, .
\end{equation}
Here $\lambda \ge 0$ is a signal-to-noise ratio, the \emph{spike} $v$ is chosen from some prior distribution on $\RR^n$, and $W$ is a tensor of random noise.
This model is known as \emph{tensor PCA} or the \emph{spiked tensor model}~\cite{richard-montanari}.

In this model, hypothesis testing or \emph{detection} entails distinguishing the spiked model above from a null model where $Y = W$ consists only of noise.
Estimation or \emph{reconstruction} entails producing a vector $\what{v} = \what{v}(Y)$ that approximates $v$.
A reasonable request is for $\what{v}$ to be more correlated with $v$ than a random guess, i.e., if $\|v\|^2, \|\what{v}\|^2 \approx n$, then $|\langle \what{v}, v \rangle| \geq \eps n$ for some constant $\eps > 0$.

Recall that we are interested in the case where the distributions involved above are orthogonally invariant.
To adhere to this setting, we will assume that $v$ is chosen uniformly from the sphere $\SS^{n - 1}(\sqrt{n}) = \{v \in \RR^n : \|v\|^2 = n \}$ and that $W$ is a \emph{Wigner random tensor} with Gaussian entries.
To keep $Y$ symmetric, we will symmetrize $W$ as follows.
\begin{definition}
    For $p \geq 1$ and $\sigma^2 > 0$, we write $\Wig(p, n, \sigma^2)$ for the law of the symmetric tensor $W \in \Sym^p(\RR^n)$ that is given entrywise by
    \begin{equation}
    \label{eq:wig-def}
        W_{i_1, \dots, i_p} = \frac{1}{\sqrt{p!}} \sum_{\pi \in S_p} G_{i_{\pi(1)}, \dots, i_{\pi(p)}}
    \end{equation}
    where $G \in (\RR^n)^{\otimes p}$ is an asymmetric tensor whose entries are i.i.d.\ as $G_{i_1, \dots, i_p} \sim \sN(0, \sigma^2)$. We write $\Wig$ for $\Wig(p,n,1)$ when these parameters are clear.
\end{definition}
\noindent
This is the natural tensor analog of the Gaussian orthogonal ensemble (GOE), which is the matrix case $p = 2$.
As there, the law $\Wig(p, n, \sigma^2)$ is orthogonally invariant.
Up to symmetry, the entries are independent but a few have larger variance, just as in the GOE entries on the diagonal have twice the variance of the rest.
See Appendix~\ref{app:wigner} for further details.

Using the tools we develop, we give a new proof of the following result on low-degree polynomial algorithms.
We state this result informally and defer to later the precise notion of ``success'' of low-degree polynomials for distinguishing two distributions; see Definition~\ref{def:adv} and Remark~\ref{rem:adv}.
We call $\QQ = \Wig$ the law of the pure noise model, and $\PP$ the law of $Y = \lambda v^{\otimes p} + W$ with $v \sim \Unif(\SS^{n - 1}(\sqrt{n}))$ and $W \sim \Wig$ independently.

\begin{theorem}[Tensor PCA detection; informal; Theorem 3.3 of \cite{KWB-2022-LowDegreeNotes}]
    \label{thm:tensor-pca-detection}
    Let $D = D(n) \in \NN$ have $D \leq \sqrt{n / 2p^2}$.
    There are constants $a_p, b_p > 0$ such that:
    \begin{enumerate}
    \item If $\lambda \leq a_p n^{-p/4} D^{-(p - 2)/4}$, then polynomials of degree at most $D$ cannot distinguish $\PP$ from $\QQ$.
    \item If $\lambda \geq b_p n^{-p/4} D^{-(p - 2)/4}$ and $D = \omega(1)$, then polynomials of degree at most $D$ can distinguish $\PP$ from $\QQ$.
    \end{enumerate}
\end{theorem}
\noindent
When $p = 2$, there is no dependence on $D$ in the thresholds for $\lambda$, and the result suggests that, on the scale $\lambda = \Theta(n^{-1/2})$, the problem transitions from being tractable in polynomial time (which is indeed achieved by various algorithms) to requiring at least time roughly $\exp(\Omega(\sqrt{n}))$.
When $p \geq 3$, the dependence on $D$ means that there are regimes of $\lambda$ where the problem is tractable but only in \emph{subexponential time} $\exp(O(n^{\delta}))$ for various $\delta \in (0, 1)$.
In fact, these scalings for low-degree upper and lower bounds are known (see again \cite{KWB-2022-LowDegreeNotes}, which sharpened the results of~\cite{sos-detect}) to hold for $D \sim n^{\delta}$ for \emph{any} $\delta \in (0, 1)$.
It is a technical limitation of our results that they are restricted to $D \lesssim \sqrt{n}$, as it is only in this regime that the basis of invariants we propose is approximately orthogonal.

One advantage of our analysis is that it gives an intuitive explanation for the latter subexponential scaling phenomenon.
In particular, the factor of $D^{\frac{p - 2}{4}}$ is directly related to the number of non-isomorphic $p$-regular multigraphs on $D$ vertices, which (see Proposition~\ref{prop:count-multigraphs}) is of the order $D^{\frac{p - 2}{2}D}$ up to smaller multiplicative factors; the analysis turns out to indicate that this quantity must be compared with $\lambda^{2D}$, leading to the factor of $D^{\frac{p - 2}{2}D \cdot \frac{1}{2D}} = D^{\frac{p - 2}{4}}$.
Per our remarks above, this is also the dimension of the space of degree $D$ invariant polynomials of a tensor.

We also give a new result for the related reconstruction problem. Since reconstruction is a more difficult task than detection, it is to be expected that degree $D$ polynomials fail to reconstruct in the same regime of parameters that they fail to detect per Theorem~\ref{thm:tensor-pca-detection}. However, actually proving that low-degree polynomials fail to reconstruct---in a precise sense deferred to Section~\ref{sec:prelim:low-deg}---does not follow from Theorem~\ref{thm:tensor-pca-detection}, and generally such reconstruction lower bounds are more difficult to prove than their detection counterparts. The first tools for low-degree reconstruction lower bounds appeared in~\cite{SW-2020-LowDegreeEstimation}, and this approach can be applied to any additive Gaussian model, which includes tensor PCA. However, to use this machinery, one needs to bound certain recursively-defined quantities, and existing analyses tend to be loose by factors involving $D$. That is to say, we expect the existing approach of~\cite{SW-2020-LowDegreeEstimation} would not obtain the precise dependence between $\lambda$ and $D$ in Theorem~\ref{thm:tensor-pca-detection} (likely the exponent on $D$ would be wrong). By using our new machinery in conjunction with the approach of~\cite{SW-2020-LowDegreeEstimation}, we manage to prove low-degree hardness for reconstruction with the correct relation between $\lambda$ and $D$. To our knowledge, this is the first low-degree reconstruction result (in any model) that achieves the ``correct'' exponent on $D$.

\begin{theorem}[Lower bound for tensor PCA reconstruction; informal]
    \label{thm:tensor-pca-recovery}
    Let $D = D(n) \in \NN$ have $D \leq \sqrt{n / 2p^2}$.
    For all odd $p \geq 3$, there is a constant $c_p > 0$ such that if $\lambda \leq c_p n^{-p/4} D^{-(p - 2) / 4}$, then polynomials of degree at most $D$ cannot reconstruct $x$ from $Y \sim \PP$ with positive correlation.
\end{theorem}
\noindent
The restriction to odd $p$ makes estimation of the vector $v$ itself well-defined, since for even $p$, $v^{\otimes p} = (-v)^{\otimes p}$.
For even $p$ one may seek to estimate $v \otimes v$ instead, but this would introduce considerable additional technicalities into our approach.

\paragraph{Computational central limit theorems and Wigner vs.\ Wishart}

The second problem we consider is inspired by the recent work of \cite{Mikulincer-2020-CLTWishartTensors}, who considered a \emph{tensor central limit theorem}: given i.i.d.\ tensors $X_j$, how quickly does the distribution of $\frac{1}{\sqrt{r}}\sum_{j = 1}^r X_j$ converge to a Gaussian tensor as $r$ increases?
That work considered a \emph{Wishart tensor} model, where each $X_j$ is the rank-one tensor $x_j^{\otimes p}$ with $x_j$ a standard Gaussian vector, with ``diagonal'' entries of the tensor---those indexed by tuples $(i_1, \dots, i_p)$ with a repeated entry---zeroed out.\footnote{This ensures that the limit tensor has centered and independent Gaussian entries, while including diagonal entries would introduce correlations.}
The author showed that $r \gg n^{2p - 1}$ suffices for \emph{information-theoretic convergence}: past this number of summands, the Wigner and Wishart laws are close in a suitable distributional distance, and no statistic whatsoever can distinguish them with high probability.

Earlier work \cite{BDER-2016-WignerWishartDetection} considered the matrix case $p = 2$, but considered \emph{computational convergence}: for what scaling of $r$ can a computationally efficient hypothesis test distinguish the Wigner and Wishart laws?
Their answer was $r \ll n^3$, which coincides with the information-theoretic lower bound, $n^{2p - 1} = n^3$.
Here we ask: how does computational convergence behave when $p > 2$?

The distribution studied by \cite{Mikulincer-2020-CLTWishartTensors} is not invariant, and so is not amenable to our tools.
As a substitute, we propose a family of random tensors whose entries are degree $p$ homogeneous polynomials in Gaussian random variables, just like the $x_j^{\otimes p}$ above, which our methods can treat and which we believe should behave similarly to Wishart models.

\begin{definition}
    The \emph{real Ginibre ensemble} $\Gin(n, \sigma^2)$ is the law of an (asymmetric) $n \times n$ random matrix $Z$ whose entries are i.i.d.\ with law $\sN(0, \sigma^2)$.
\end{definition}
\noindent
We write $\Haar = \Haar(n)$ for the Haar probability measure over the orthogonal group $\sO(n)$, omitting the $n$ when it is clear from context.
The basic idea that will be involved below is that the distributions $\Gin(n, 1/n)$ and $\Haar(n)$ behave similarly: both are orthogonally invariant (Proposition~\ref{prop:Gin-invariant}), and the entries of both are typically $O(1/\sqrt{n})$.
Moreover, ``invariantizing'' a tensor by forming $Z \cdot T$ with $Z \sim \Gin(n, 1/n)$ makes its entries homogeneous of degree $p$ in the Gaussian entries of $Z$, giving an object resembling $x_j^{\otimes p}$ from the model of \cite{Mikulincer-2020-CLTWishartTensors}.

\begin{theorem}[Lower bound for Wigner vs.\ Wishart detection; informal]
    \label{thm:wishart-informal}
    Suppose that $\mu_n$ are probability measures on $\Sym^p(\RR^n)$ satisfying the following properties for $A \sim \mu_n$:
    \begin{enumerate}
    \item For all $i \in [n]^p$ having a repeated entry, $A_{i_1, \dots, i_p} = 0$ almost surely.
    \item There is a constant $C > 0$ such that, for all $i \in [n]^p$, $|A_{i_1, \dots, i_p}| \leq C$ almost surely.
    \item $\|A\|_F^2 = n^p$ almost surely.
    \end{enumerate}
    Let $Z_1, \dots, Z_r \sim \Gin(n, 1/n)$ be i.i.d.\ and $A_1, \dots, A_r \sim \mu_n$ be i.i.d., and write $\PP = \PP_{n, r}$ for the law of $r^{-1/2}\sum_{j = 1}^r Z_j \cdot A_j$.
    There is a constant $a_{p, C} > 0$ such that the following holds.
    Suppose that $D = D(n) \leq \sqrt{n / 2p^2}$ is given and $r = r(n)$ satisfies
    \begin{equation}
        r \geq a_{p, C} \cdot \left\{\begin{array}{ll} n^p & \text{if } p \text{ is odd}, \\ n^{3p/2} & \text{if } p \text{ is even}\end{array}\right\}.
    \end{equation}
    Then, polynomials of degree at most $D$ cannot distinguish $\PP_{n, r(n)}$ from $\Wig$.
\end{theorem}

We emphasize one simple example: consider $A$ a deterministic tensor, with
\begin{equation}
    \label{eq:A-all-ones}
    A_{i_1, \dots, i_p} = c\cdot \One\{\text{no entry is repeated in } i\},
\end{equation}
where $c = c(p, n)$ is chosen (close but not equal to 1) so that $\|A\|_F^2 = n^p$.
If $A$ were simply the all-ones tensor $1^{\otimes p}$, then $Z \cdot A$ would have the law of $x_j^{\otimes p}$ from the above Wishart model.
Thus, we believe it is reasonable to think of $A$ as an invariant surrogate for the law of $x_j^{\otimes p}$ with positions having repeated indices zeroed out.
Note also that, as in the latter model, in our model we have $\EE Z_j \cdot A_j = 0$ even when $p$ is even, so there is no need to center these terms.

The result suggests that computational central limit theorem convergence (i.e., with respect to polynomial-time algorithms) in this Wishart-like model occurs once $r \gg n^{(2 + \One\{p \text{ even}\})p / 2}$.
For the special case of the above choice of $A$, we may bolster this proposal with a matching upper bound.
\begin{theorem}[Upper bound for Wigner vs.\ Wishart detection; informal]
    \label{thm:wishart-informal-upper}
    Let $A$ be as in \eqref{eq:A-all-ones}, and $\mu_n$ be the Dirac delta mass on $A$.
    Then, in the setting of Theorem~\ref{thm:wishart-informal},
    if $r = r(n)$ satisfies
    \begin{equation}
        r  \ll \left\{\begin{array}{ll} n^p & \text{if } p \text{ is odd}, \\ n^{3p/2} & \text{if } p \text{ is even}\end{array}\right\},
    \end{equation}
    then there is a polynomial of degree $D = 3$ if $p$ is even and $D = 4$ if $p$ is odd that can distinguish $\PP_{n, r(n)}$ from $\Wig$.
\end{theorem}
\noindent
As in the case of tensor PCA, our diagrammatic approach gives a clear intuitive explanation of the structure of this threshold.
Here, the unusual-seeming dependence on the parity of $p$ arises because the parity of $p$ controls the smallest possible size of a $p$-regular multigraph on more than two vertices, which is 3 when $p$ is even but 4 when $p$ is odd.
Our calculations will demonstrate that in this setting the smallest graphs (or disjoint unions thereof) correspond to the most powerful polynomials for hypothesis testing, so the above phenomenon accounts for detection being proportionally easier (as reflected in a larger threshold for $r$) when $p$ is even.

We note also that the computational threshold $r \gg n^{(2 + \One\{p \text{ even}\})p / 2}$ we establish coincides with the information-theoretic lower bound $r \gg n^{2p - 1}$ of \cite{Mikulincer-2020-CLTWishartTensors} when $p = 2$, but is strictly lower by a polynomial factor once $p \geq 3$.
In this sense, our result gives initial evidence that testing Wigner vs.\ Wishart tensors may exhibit a statistical-computational gap not present in the matrix case (provided that the information-theoretic lower bound of \cite{Mikulincer-2020-CLTWishartTensors} can be matched by an upper bound, which to the best of our knowledge is not yet established).

\subsection{Main Proof Technique: Tensorial Finite Free Cumulants}

The following remarkable objects are at the heart of our approach to analyzing low-degree algorithms for invariant problems.

\begin{theorem}[Finite free cumulants of a tensor]
    \label{thm:cumulants}
    For any $0 \leq D \leq \sqrt{n / 2p^2}$, there is a collection of polynomials $\kappa_G(T)$ of degree at most $D$ in the indices of a tensor $T$, indexed by (non-isomorphic) $p$-regular multigraphs on at most $D$ vertices, such that the following hold:
    \begin{enumerate}
    \item (Empty Graph) $\kappa_{\emptyset}(T) = 1$.
    \item (Invariance) The $\kappa_G$ are invariant polynomials: $\kappa_G(Q \cdot T) = \kappa_G(T)$ for all $Q \in \sO(n)$.
    \item (Basis) The $\kappa_G$ are a basis for the invariant polynomials of degree at most $D$.
    \item (Approximate Orthonormality) The Gram matrix $M_{G, H} \colonequals \EE_{T \sim \Wig}[\kappa_G(T) \kappa_H(T)]$, restricted to those $G$ and $H$ none of whose connected components consist of two vertices connected by $p$ parallel edges, satisfies
    \begin{equation}
        \frac{1}{2} < \lambda_{\min}(M) \leq \lambda_{\max}(M) \leq 2.
    \end{equation}
    \item[5.] (Wigner Expectations) For any $G$ that is not empty and not a single two-vertex connected component of the kind described above,
    \begin{equation}
        \Ex_{T \sim \Wig} \kappa_G(T) = 0.
    \end{equation}
    \item[6.] (Additive Free Convolution) For any $S, T \in \Sym^p(\RR^n)$ and $G$ connected,
    \begin{equation}
        \Ex_{Q \sim \Haar} \kappa_G(S + Q \cdot T) = \kappa_G(S) + \kappa_G(T).
    \end{equation}
    \end{enumerate}
\end{theorem}

\begin{remark}
    When $D > \sqrt{n / 2p^2}$, the finite free cumulants exist and satisfy Conditions 1--3 and 5--6, but Condition 4 breaks down.
    More dramatically, once $D \gg \sqrt{n}$ the Gram matrix of the finite free cumulants no longer has a bounded condition number, so they are no longer approximately orthogonal.
    See Appendix~\ref{app:weingarten} for further details.
\end{remark}

\noindent
Of course, by basic linear algebra over the space of polynomials endowed with inner product $\langle f, g \rangle = \EE_{T \sim \Wig}[f(T)g(T)]$, there must exist polynomials satisfying Claims 1--5, and indeed if this were all we were interested in we could demand exact orthonormality in Claim 4.
What is surprising is, first, that we can identify such polynomials explicitly (analogous to, say, the Boolean Fourier basis or the Hermite basis, as opposed to arbitrary abstract collections of orthogonal polynomials), and second, that those same polynomials happen to satisfy the exact algebraic Condition 6.

We will see that these objects are useful tools for the analysis of low-degree polynomial algorithms in models having both invariance and additive structure.
Consider the case of hypothesis testing. The analysis of low-degree polynomials, given an approximate basis of this kind, boils down to evaluating expectations of $\kappa_G$ under the alternative hypothesis---the distribution of $Y = \lambda v^{\otimes p} + W$ in tensor PCA, or of $r^{-1/2} \sum_{i = 1}^r Z_i \cdot A_i$ in a computational central limit theorem.
But, Claim 6 above gives a powerful tool to handle such situations, as is especially apparent for tensor PCA: combining Claims 5 and 6, we see that
\begin{equation}
\Ex_{v, W} \kappa_G(\lambda v^{\otimes p} + W) = \Ex_{v, W, Q} \kappa_G(\lambda v^{\otimes p} + Q \cdot W) = \Ex_v \kappa_G(\lambda v^{\otimes p}) + \Ex_W \kappa_G(W) = \Ex_v \kappa_G(\lambda v^{\otimes p}),
\end{equation}
and the last expression is simple to compute in closed form (Proposition~\ref{prop:all-spike-cumulant}).
To give some intuition, the result is equal to leading order to $\EE_v m_G(\lambda v^{\otimes p})$ (generally, the $\kappa_G$ are built by a series of ``adjustments'' to the corresponding $m_G$), which, even inside the expectation, is easily computed by hand as $\lambda^{|V(G)|} \|v\|^{2|E(G)|}$ (Proposition~\ref{prop:all-spike}).

As we have mentioned, there is a deeper mathematical meaning to the $\kappa_G$: as Claim 6 hints at, the $\kappa_G$ are really the correct generalization to tensors of \emph{finite free cumulants} from the free probability theory of random matrices.
In this regard, our $\kappa_G$ are the third step in a chain of generalizations in the literature, which we review below.

First, the classical cumulants are polynomials $\kappa_t^{(1)}$ in the moments of a random variable with the property that, if $A$ and $B$ are independent scalar random variables, then the cumulant of their convolution $A + B$ is the sum of their cumulants: $\kappa(A+B)=\kappa(A)+\kappa(B)$. For instance, the first cumulant is the expectation, and the second cumulant is the variance, for which these additivity rules are familiar.

Second, the theory of \emph{free probability} considers the less straightforward question of how addition affects the spectra of \emph{matrix-valued} random variables $A$ and $B$.
To make this as tractable as the scalar case, it is not enough to assume that $A$ and $B$ are independent---the typical spectrum of $A + B$ still depends on the extent to which $A$ and $B$ commute.
Free probability restricts its attention to \emph{freely independent} pairs $A, B$ whose frames of eigenvectors are ``maximally uncorrelated'' with each other. For our purposes, we may take this to mean that we are interested in the typical spectrum of $A + Q^{\top}BQ = A + Q \cdot B$ for $Q \sim \Haar$ drawn independently of $A$ and $B$.
The \emph{finite free cumulants} $\kappa_t^{(2)}$ then satisfy this additivity property in expectation, $\EE_Q \kappa_t^{(2)}(A + Q \cdot B) = \kappa_t^{(2)}(A) + \kappa_t^{(2)}(B)$, just as in Claim 6.
They are a relatively recent innovation in free probability \cite{AP-2018-CumulantsFiniteFreeConvolution,AGVP-2023-FiniteFreeCumulants}, but a notion of their limit as $n \to \infty$, called just the \emph{free cumulants}, was crucial to the earlier origins of the theory (see, e.g., the reference \cite{mingo-speicher}).

Given that they satisfy Claim 6, it should now be clear that our $\kappa_G$ generalize this notion of finite free cumulant to tensors.
Even for matrices, it seems not to have been noticed previously that finite free cumulants lead two parallel lives: on the one hand, they are statistics of random matrices that behave well under freely independent summations; on the other, as our work shows, they yield a natural, explicit, near-orthogonal basis of invariant functions under the inner product induced by the Wigner distribution, which for matrices is the Gaussian orthogonal ensemble.
Also interesting is that there is no clear notion of limiting free cumulants for tensors, because there is also no clear notion of eigenvalues or empirical spectral distribution.
Still, \emph{finite} free cumulants are sensible and useful objects.

The astute reader will notice that Claim~6 is restricted to connected graphs---a simple issue that we address in Proposition~\ref{prop:additivity-cumulants}---and that Claims~4 and~5 make exceptions for a particular type of connected component, namely a two-vertex multigraph, which we call a \emph{Frobenius pair} (see Section~\ref{sec:frobenii}).
The associated cumulant may be viewed as a formal analog of the variance (the second classical cumulant), and like the variance is always non-negative.
To actually use the finite free cumulants as a basis, we will have to define and work carefully with a suitable re-centering that accounts for this, which will be one of our main tasks in Section~\ref{sec:cumulants}.

We also have not said anything about the application of these ideas to estimation, which requires vector-valued functions of a tensor.
This turns out to admit a parallel theory concerning \emph{equivariant functions}, ones $f: \Sym^p(\RR^n) \to \RR^n$ satisfying $f(Q \cdot T) = Q \cdot f(T)$, which are spanned by constructions like $m_G(T)$ but where $G$ has an ``open'' edge corresponding to the vector output.
These modified graph moments also are spanned by a basis of objects similar to the $\kappa_G$. Our lower bound for reconstruction is obtained by using this basis in conjunction with the existing strategy of~\cite{SW-2020-LowDegreeEstimation}.

\subsection{Related Work}
\label{sec:related}

\paragraph{Eigenvalues of tensors}
There are several reasonable different definitions of eigenvalues and eigenvalue-eigenvector pairs for tensors~\cite{Qi-2005-EigenvaluesRealSupersymmetricTensor,Qi-2007-EigenvaluesInvariantsTensors,QCC-2018-TensorEigenvalues}.
It is non-trivial to compute these eigenvalues~\cite{CDN-2014-AllRealEigenvaluesSymmetricTensors}, and, under the common notion of ``E-eigenvalues,'' typical tensors have exponentially many eigenvalues in their dimension, per the analysis of the critical points of spherical spin glass models~\cite{ABAC-2013-RandomMatricesComplexity,Subag-2017-ComplexitySecondMoment}.
Most importantly, there seems to be no natural analog for tensors of the spectral decomposition for matrices, making it seemingly impossible in general to reduce the evaluation of tensor invariants to computations with eigenvalues.
This is in contrast to the matrix setting where the invariants are just power sums of the eigenvalues, $\tr(T^\ell) = \sum_{i=1}^n \lambda_i^\ell$, and polynomials thereof.

\paragraph{Random tensor theory}
Aspects of random matrix theory that have been generalized to the tensor setting\footnote{Some works, such as \cite{AHH-2012-RandomTensors}, also use the term ``random tensor theory'' to describe their results, though they really concern random matrices built out of random tensors, say by taking sample covariances of vectorizations of these tensors.} include the derivation of the maximum eigenvalue of a Gaussian random tensor by a ``tensorial power method'' \cite{Evnin-2021-MelonicDominanceLargestEigenvalueTensor}, a definition of a resolvent, a spectral density, and Wigner's semicircle law for Gaussian random symmetric tensors~\cite{Gurau-2020-WignerSemicircleLawTensors}, universality of this generalization of Wigner's law~\cite{Gurau-2014-UniversalityRandomTensors} (with applications to spherical spin glass models~\cite{BGS-2013-UniversalityPSpinGlass}), and Harish-Chandra--Itzykson--Zuber integrals \cite{CGL-2023-TensorHCIZIntegral1,CGL-2023-TensorHCIZIntegral2}.
See also the book \cite{Gurau-2017-RandomTensors} for many details.
We will both elaborate on some of these results (see, e.g., our Theorem~\ref{thm:wigner-moments-hard} on the scaling of graph moments of Gaussian random tensors) and will show generalizations to the tensor setting of aspects of free probability theory, including the notion of free cumulants, properties of additive free convolution, and Voiculescu's theorem on the asymptotic freeness of matrices under random rotation of their eigenvectors.

Very recently, and independently of our work, another approach to free probability for tensors based on Gurau's tensor resolvent was proposed in~\cite{bonnin2024universality}.
This line of work differs from our results in focusing on a coarser collection of invariants which is, in our notation, the sum of the moments $m_G(T)$ over all connected $G$ on a given number of vertices.

\paragraph{Tensor PCA}

Starting with the work of~\cite{richard-montanari}, the tensor PCA problem is now well-studied. A variety of different polynomial-time algorithms are known to achieve the scaling $\lambda \sim n^{-p/4}$~\cite{richard-montanari,pmlr-v40-Hopkins15,HSSS,homotopy,hastings,iron-landscapes}. While information-theoretically it is possible to succeed for smaller $\lambda$ \cite{JLM-2020-StatisticalTensorPCA}, we expect that no poly-time algorithm can achieve this due to lower bounds against low-degree polynomials~\cite{sos-detect} and a reduction from a variant of the planted clique problem~\cite{BB-2020-ReducibilityStatCompGaps}. Furthermore, various algorithms are known to achieve a particular tradeoff between the SNR $\lambda$ and the (super-polynomial but sub-exponential) runtime, when $\lambda \ll n^{-p/4}$~\cite{RSS,BhattiproluGGLT16,BhattiproluGL17,wein-elalaoui-moore}.\footnote{Notably, unlike in the matrix case, once $p \geq 3$ the tensor analog of power iteration---which may be viewed as computing a tree-shaped tensor network---performs suboptimally \cite{richard-montanari,WZ-2024-PowerIterationTensorPCA}.}
The low-degree result of~\cite{KWB-2022-LowDegreeNotes} that we recalled in Theorem~\ref{thm:tensor-pca-detection} suggests that this tradeoff cannot be improved.
Most closely related to our results, there is a line of work considering the application of tensor invariants to tensor PCA, using the ``random tensor theory'' machinery mentioned above~\cite{Gurau-2020-WignerSemicircleLawTensors,Ouerfelli-Tamaazousti-Rivasseau-2022, OR-2020-TensorPCATraceInvariants, Ouerfelli-2022-TensorPCA}.

\paragraph{Algorithms using tensor networks and distinct indices}

\begin{figure}
    \centering
    \includegraphics[width=1.8in]{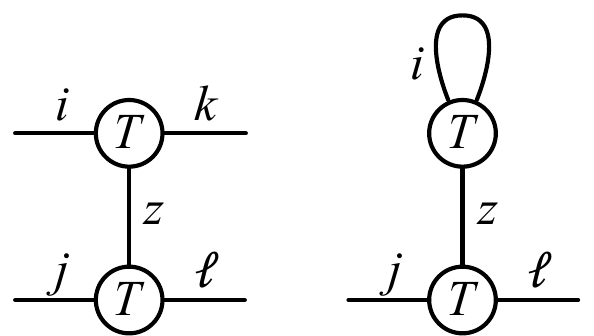}
    \caption{On the left, a tensor network defining a 4-index tensor from two copies of a 3-ary tensor; we sum over all $n$ values of the internal index $z$. It can be viewed, among other things, as an $n^2$-dimensional matrix $M_{(i,j),(k,\ell)}$. This matrix is used in a spectral algorithm for tensor PCA in~\cite{pmlr-v40-Hopkins15}. On the right, a partial trace where we also sum over $i$, giving an $n$-dimensional matrix $M_{j\ell}$ used by~\cite{HSSS} for another spectral algorithm.}
    \label{fig:H}
\end{figure}

Our notion of ``tensorial free cumulants'' arose from seeking an invariant version of the ``distinct-index'' tensor networks used in the recent work~\cite{Wein-2022-LowDegreeTensorDecomposition} on tensor decomposition.
In the matrix setting, if $A$ is the adjacency matrix of a graph, the distinct-index trace
\[
A^{(\ell)} = \sum_{\substack{i_1,i_2,\ldots,i_\ell \\ \textrm{distinct}}} A_{i_1,i_2} A_{i_2,i_3} \cdots A_{i_{\ell-1},i_\ell}
\]
counts self-avoiding walks. This was used previously by~\cite{Massoulie-2014-CommunityDetectionWeakRamanujan} in the context of community detection in the stochastic block model, and by~\cite{DHS-2020-SpikedMatrixHeavyTailed} as a variance reduction technique for estimators and test statistics in matrix PCA.
Other previous uses of tensor networks in tensor PCA include~\cite{pmlr-v40-Hopkins15}, which uses the $n^2$-dimensional matrix $M_{(i,j),(k,\ell)} = \sum_z T_{ikz} T_{j\ell z}$ shown in Figure~\ref{fig:H} on the left, and~\cite{HSSS} which uses its partial trace on the right. Algorithmic applications of more elaborate tensor networks can be found in~\cite{MW-2019-SpectralTensorNetworks,DOLTS-2022-Fast3TensorDecomposition}.
The recent work of \cite{semerjian2024matrix} uses similar ideas of optimizing over invariant and equivariant polynomials as well, but, working in the matrix setting, does not use the graphical tensor network notation.

The \emph{graph matrices} widely used to analyze the \emph{pseudocalibration} construction in lower bounds against the sum-of-squares hierarchy also feature a restriction of similar summations to distinct indices, and through this give an orthogonal basis for the polynomials invariant under a group action (in that case the action of the symmetric group).
See \cite{sos-clique,GJJPR-2020-SK,PR-2022-SOSPCA,CPRX-2023-SOSDensestSubgraph} for applications of pseudocalibration, and Remark 2.3 of \cite{AMP-2016-GraphMatrices} for discussion of the orthogonality properties of graph matrices.
Finally, the combinatorial structure of a nearly-orthogonal basis for a special class of tensor networks arising in \cite{GJJPR-2020-SK} was studied by \cite{JP-2021-BasesInnerProductPolynomials}.

\paragraph{Cumulants in low-degree lower bounds}
The work of \cite{SW-2020-LowDegreeEstimation} also treats reconstruction using low-degree polynomials, and also encounters a quantity that is referred to as a ``cumulant.'' Those quantities are simply the classical (joint) cumulants of scalar random variables. In contrast, our cumulants are firstly the \emph{free} cumulants of free probability, and secondly are polynomial functions of a matrix or tensor, rather than statistics of an entire probability distribution.

\section{Notation}

We write $\id$ for the identity matrix, $J$ for the all-ones matrix, and 1 for the all-ones vector (when it is used in linear-algebraic context).
We write $n^{\underline{b}} \colonequals n! / (n - b)!$ for the falling factorial, and use the less conventional $n^{\dunderline{b}} \colonequals n!! / (n - 2b)!!$ for the falling double factorial.

\section{Preliminaries}

\subsection{Closed Tensor Networks as Invariant Polynomials}
\label{sec:prelim:tensor-networks}

First, let us motivate the family of linear combinations of tensor networks as a natural class of algorithms.
For computations over symmetric matrices $T$, an important role is played by \emph{linear spectral statistics}, the quantities $F(T) = \sum_{i = 1}^n f(\lambda_i(T))$ for some function $f: \RR \to \RR$.

The space of all linear spectral statistics is a natural limit of the same space where $f$ are polynomials.
That space admits a more algebraic interpretation: it is the same as the space of all symmetric polynomials of the eigenvalues of $T$, which is also the space of all invariant polynomials of a matrix $T$: if $F(Q T Q^{\top}) = F(T)$ for all $Q \in \sO(n)$ for a polynomial $F$, then $F$ must be a symmetric polynomial of $(\lambda_1(T), \dots, \lambda_n(T))$.

In the matrix case, when $f(\lambda)$ is a polynomial, then we have $F(T) = \tr(f(T))$.
These are linear combinations of $\tr(T^k)$, and one may check that $\tr(T^k) = m_{C_k}(T)$, where $C_k$ is the $k$-cycle.
In summary, then, the linear spectral statistics for $f$ a polynomial---a natural class of ``spectral algorithms'' over matrices---have two coinciding interpretations: on the one hand, they are all invariant functions of a matrix $T$; on the other, they are linear combinations of the $m_G(T)$ where $T$ is a 2-regular multigraph.

These interpretations do not depend on the notion of ``eigenvalue.''
While there is no one clear definition of eigenvalues for tensors, this algebraic interpretation of linear spectral statistics does admit a tensorial generalization.
In particular, we have:
\begin{definition}[$p$-regular multigraphs]
    Write $\sG_{d, p}$ for the set of (non-isomorphic) $p$-regular multigraphs on $d$ unlabelled vertices.
\end{definition}
\begin{theorem}
    \label{thm:tensor-invariant}
    Let $R$ be the ring of invariant polynomials $f: \Sym^p(\RR^n) \to \RR$.
    Then $R$ is generated as a vector space by the polynomials $m_G(T)$ over $G \in \bigsqcup_{d \geq 0} \sG_{d, p}$.
    In particular, any invariant homogeneous polynomial $f(T)$ of degree $d$ is a linear combination $\sum_i \alpha_i \,m_{G_i}(T)$ for $G_i \in \sG_{d, p}$.
\end{theorem}
\noindent
In this sense, we propose that the space of linear combinations of $m_G(T)$ should be viewed as the ``right'' tensorial generalization of linear spectral statistics.
(The analogy between $m_G(T)$ for $T$ a tensor and $m_{C_k}(T) = \tr(T^k)$ for $T$ a matrix, the latter of which is, up to renormalization, the $k$th moment of the empirical spectral distribution of $T$, is also why we call the $m_G(T)$ ``moments.'')

Theorem~\ref{thm:tensor-invariant} is a classical result in invariant theory, namely the ``First Fundamental Theorem'' for the orthogonal group~\cite{Weyl46}; see \cite[\S 4.3]{GoodmanWallach} for a modern treatment, following~\cite[pp.\ 285--6]{Atiyah73}.\footnote{These results are usually stated for the action of $\sO(n)$ on all ``axes'' of a general tensor $T \in (\RR_n)^{\otimes p}$, but by Proposition~\ref{prop:sym-invariant-extension} they also immediately apply to invariants of symmetric tensors.}
However, these proofs may be opaque to some readers, and lack graphical panache.
We provide a graphical proof in the style of our other calculations in Appendix~\ref{app:invariants} for the reader's appreciation.

We can also consider mixed moments of multiple tensors, possibly with different arities:

\begin{definition}
\label{def:mixed-moments}
Let $\mathcal{T}$ be a set of symmetric tensors, and let $G=(V,E)$ be a multigraph associated with a function $T:V \to \mathcal{T}$ that labels each vertex $v \in V$ with a tensor $T(v) \in \mathcal{T}$ of arity $\deg v$. Then  define the mixed moment $m_G(\mathcal{T})$ as the contraction
\begin{equation}
\label{eq:mixed-moment}
m_G(\mathcal{T})
= \sum_{i \in [n]^E} \prod_{v \in V} T(v)_{i(\partial v)} \, .
\end{equation}
\end{definition}

\begin{figure}
    \centering \includegraphics[width=4in]{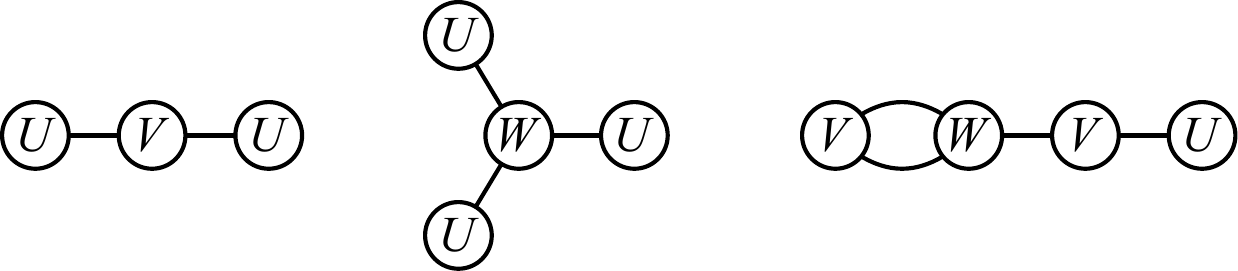}
    \caption{Three examples of mixed moments of tensors. Here $U$, $V$, and $W$ have arity 1, 2, and 3 respectively. From left to right, these yield the bilinear form $\langle U,VU \rangle$, the trilinear form $\langle W, U^{\otimes 3}\rangle$, and the quantity in~\eqref{eq:mixed-example}.}
    \label{fig:mixed-example}
\end{figure}
\begin{example}
Let $U$, $V$, and $W$ be tensors of arity $1$, $2$, and $3$ respectively. The rightmost multigraph $G$ in Figure~\ref{fig:mixed-example} yields
\begin{equation}
\label{eq:mixed-example}
m_G(U,V,W)
= \sum_{ijk\ell} V_{ij} \,W_{ijk} \,V_{k\ell} \,U_\ell \, ,
\end{equation}
which is quadratic in $V$ and multilinear in $W$ and $U$.
\end{example}
\begin{example}
    Let $U$ and $V$ be matrices of the same dimension.
    Then, the mixed connected graph moments of these matrices are the mixed traces, $\tr(U^{k_1}V^{\ell_1} \cdots U^{k_m}V^{k_m})$ for $k_i, \ell_j \geq 1$.
    These are the quantities to which the concept of \emph{freeness} in free probability theory pertains; we will see later that graph moments of tensors admit a generalization of parts of this theory.
\end{example}
\noindent
A version of Theorem~\ref{thm:tensor-invariant} also holds for this setting:
\begin{theorem}
    \label{thm:tensor-invariant-multi}
    Let $R$ be the ring of invariant polynomials $f: \Sym^{p_1}(\RR^n) \times \cdots \times \Sym^{p_m}(\RR^n)$, in the sense that $f(T_1, \dots, T_m) = f(Q \cdot T_1, \dots, Q \cdot T_m)$ for all $Q \in \sO(n)$.
    Then, $R$ is generated as a vector space by the $m_G(T_1, \dots, T_m)$: any invariant homogeneous polynomial $f(T_1, \dots, T_m)$ of degree $d_i$ in $T_i$ for each $i \in [m]$ is a linear combination of mixed moments $m_G(T_1, \dots, T_m)$ where each $T_i$ corresponds to $d_i$ vertices of $G$.
\end{theorem}
\noindent
We again give a proof in Appendix~\ref{app:invariants}, which is essentially identical to that of Theorem~\ref{thm:tensor-invariant}.
This is a non-trivial generalization even in the case $d_i = 2$ of matrices, due to \cite{Procesi-1976-InvariantTheoryMatrices}, and we give a simple and unified proof for either setting.

\subsection{Open Tensor Networks as Equivariant Polynomials}
\label{sec:prelim:open-tensor-networks}

We can also consider vector-, matrix-, or tensor-valued functions $f(T)$ whose outputs transform properly when $T$ is transformed. This will include, for instance, an optimal algorithm for reconstructing the spike in tensor PCA.

\begin{definition} Let $f$ be a function from $p$-ary tensors to $\ell$-ary tensors. We say that $f$ is \emph{equivariant} if for any $T$ and any orthogonal matrix $Q \in \sO(n)$,
\[
f(Q \cdot T) = Q \cdot f(T) \, ,
\]
and similarly if $f$ is a function from multiple tensors to $\ell$-ary tensors.
\end{definition}

\begin{example}
Let $Q$ be a matrix. If $v$ is a left eigenvector of $T$, then $Q \cdot v = vQ$ is a left eigenvector of $Q \cdot T = Q^\top T Q$. Thus the output of a spectral algorithm that returns the dominant eigenvector of a matrix is equivariant.
\end{example}

Given the machinery we have set up, it is easy to generalize Theorem~\ref{thm:tensor-invariant} to show that the equivariant polynomials are linear combinations of ``open graph'' moments where the graph has $\ell$ open edges; see~\cite[p.64]{Weyl46} and~\cite[p.286]{Atiyah73}. These are also known as partial contractions.

\begin{definition}
An \emph{$t$-open multigraph} $G$ is a triple $(V,E,E')$ where each edge $e \in E \cup E'$ is a  multiset $e \subseteq V$ with $|e|=2$ if $e \in E$ and $|e|=1$ if $e \in E^{\prime}$, and where $|E^{\prime}| = t$. We say $E$ and $E^{\prime}$ are the sets of closed and open edges respectively. The degree of a vertex $v \in V$ is the total number of times it appears in $E \cup E^{\prime}$, including self-loops in $E$ where it occurs twice. We say $G$ is $p$-regular if every vertex has degree $p$.
\end{definition}

\begin{definition}
\label{def:open-moments}
Let $T \in \Sym^p(\RR^n)$ and let $G=(V,E,E')$ be a $p$-regular open multigraph with $|E'|=\ell$ open edges. Let $i \in [n]^{E \cup E'}$ denote a labeling of the edges with indices in $[n]$. Let $i(E)$ and $i(E')$ denote the labels of the closed and open edges respectively, and as in Definition~\ref{def:moments} let $i(\partial v)$ denote the multiset of indices $i(e)$ associated with the edges $e$ incident to $v$. Then, the associated \emph{open graph moment} $m_G(T)$ is the following $\ell$-ary tensor,
\[
m_G(T)_{i(E')}
= \sum_{i(E) \in [n]^E} \prod_{v \in V} T_{i(\partial v)} \, .
\]
For a set of tensors $\mathcal{T}$, we similarly define
\[
m_G(\mathcal{T})_{i(E')}
= \sum_{i(E) \in [n]^E} \prod_{v \in V} T(v)_{i(\partial v)} \, .
\]
where for each $v \in V$, $T(v) \in \mathcal{T}$ has arity $\deg v$.
\end{definition}

\noindent In words, these partial contractions are graph moments where we sum only over the indices on closed edges, and leave the indices on open edges as unbound variables.

\begin{figure}
    \centering
    \includegraphics[width=0.5 \columnwidth]{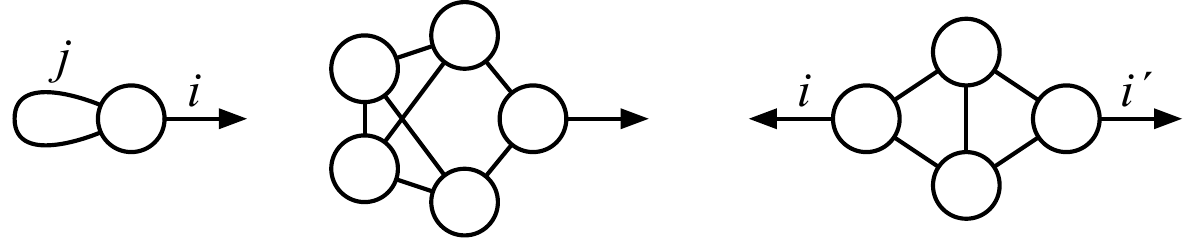}
    \caption{Three open multigraphs whose partial contractions $m_G(T)$ yield vectors or matrices, i.e., 1-ary or 2-ary tensors. The first is the vector $m_G(T)_i = \sum_j T_{ijj}$. The third is the matrix $m_G(T)_{ii'}$ given by~\eqref{eq:open-k4}.}
    \label{fig:open-graphs}
\end{figure}

\begin{example}
We show three 3-regular open multigraphs in Figure~\ref{fig:open-graphs}. The first two have one open edge each, and the third has two open edges. Thus their partial contractions $m_G(T)_{i(E')}$ result in a vector or matrix. For the first graph we have the partial trace $m_G(T)_i = \sum_j T_{ijj}$, and for the third one we have the matrix
\begin{equation}
\label{eq:open-k4}
m_G(T)_{ii'}
= \sum_{jk\ell mn} T_{ijk} \,T_{j \ell m} \,T_{i' \ell n} \,T_{kmn} \, .
\end{equation}
This is the graph moment $G_{K_4}(T)$ given by~\eqref{eq:k4} except that $i$ and $i'$ are open indices rather than being set equal and summed over. Taking the trace of this matrix closes this loop and yields $G_{K_4}(T)$.
\end{example}

We may already start to glean some of the benefits of this graphical language.
For example, it lets us understand more clearly the change of basis from Definition~\ref{def:change-basis}.

\begin{remark}
As shown in Figure~\ref{fig:conjugation}, $Q \cdot T$ consists of placing a copy of $Q$ on each ``outward'' edge of $T$. Since $Q$ is not symmetric, we place arrows on the edges to indicate which of its indices is contracted with an index of $T$; reversing this arrow swaps $Q$'s indices and thus converts it to its transpose $Q^\top$.
In the matrix case $p=2$, $Q \cdot T$ corresponds to standard matrix conjugation, since $Q^\top T Q$ can also be written $T \cdot Q^{\otimes 2}$:
\[
\left(Q^\top T Q \right)_{jj'}
= \sum_{i,i'} Q^\top_{ji} \,T_{ii'} \,Q_{i'j'}
= \sum_{ii'} T_{ii'} \,Q_{ij} \,Q_{i'j'}
= \left( T \cdot Q^{\otimes 2} \right)_{jj'} \, .
\]
\end{remark}

\begin{figure}
    \centering
    \includegraphics[width=3.3in]{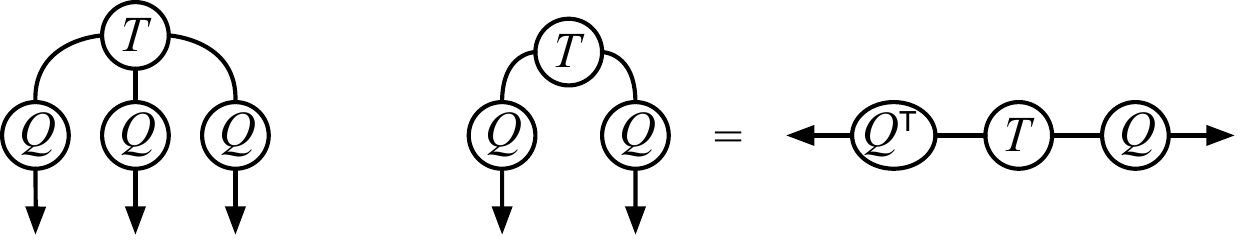}
    \caption{On the left, conjugating a 3-ary tensor $T$ by an orthogonal matrix $Q$. Each index of $T$ undergoes the same orthogonal basis change. On the right, for a 2-ary tensor, i.e., a matrix, this coincides with the usual notion of conjugation $Q^\top T Q$. Arrows indicate that $T$'s indices are contracted with the left index $i$ of $Q_{ij}$. Reversing an arrow converts $Q$ to its transpose $Q^\top$.
    }
\label{fig:conjugation}
\end{figure}
\noindent
Using this, the proof that graph moments are invariant functions becomes straightforward.\begin{proposition}
\label{prop:moments-are-invariant}
    For any $T \in \Sym^p(\RR^n)$, any multigraph $G$, and any $Q \in \sO(n)$,  $m_G(Q \cdot T) = m_G(T)$. The same holds for mixed moments $m_G(\mathcal{T})$ as in Definition~\ref{def:mixed-moments}.
\end{proposition}

\begin{proof}
Since reversing the direction of an edge changes $Q$ to $Q^\top$, placing outgoing $Q$'s on the edges of each vertex creates a pair $QQ^\top$ on each edge. Since $Q$ is orthogonal, $Q^\top = Q^{-1}$ and $QQ^\top = \id$.
\end{proof}

Finally, analogous results hold for open graph moments as did for closed ones.
First, just as closed graph moments are invariant, open graph moments are equivariant:

\begin{proposition}
\label{prop:open-moments-are-equivariant}
    For any $T \in \Sym^p(\RR^n)$, any open multigraph $G$, and any $Q \in \sO(n)$, $m_G(Q \cdot T) = Q \cdot m_G(T)$. The same holds for mixed moments $m_G(\mathcal{T})$ as in Definition~\ref{def:open-moments}.
\end{proposition}

\begin{proof}
As in Proposition~\ref{prop:moments-are-invariant}, placing an outgoing $Q$ on the edges of each vertex causes a canceling pair $QQ^\top = \id$ on each closed edge, and an outgoing $Q$ on each open edge. The latter transform $m_G(T)$ to $Q \cdot m_G(T)$.
\end{proof}

And, as for closed graph moments and invariant polynomials, the open graph moments span all equivariant polynomials.
\begin{theorem}
\label{thm:tensor-equivariant}
Let $p$ be an equivariant function from $p$-ary tensors to $\ell$-ary tensors. If $f$ is a homogeneous polynomial of degree $d$, then $f(T)$ is a linear combination $\sum_i \alpha_i m_{G_i}(T)$ where the $G_i$ are $p$-regular open multigraphs with $d$ vertices and $\ell$ open edges.

More generally, let $f(T_1,\ldots,T_m)$ be an equivariant function that yields $\ell$-ary tensors. If $p$ is a homogeneous polynomial  of degree $d_i$ in each $T_i$, then is a linear combination of mixed moments $m_G(T_1,\ldots,T_m)$ where $G$ has $\ell$ open edges and each $T_i$ corresponds to $d_i$ vertices of $G$.
\end{theorem}
\noindent
Once again, we defer a proof to Appendix~\ref{app:invariants}.

\subsection{Invariant Tensors, Brauer Space, and the Weingarten Function}
\label{sec:prelim:weingarten}

An important role in our calculations will be played by the subspace of vectors in $(\RR^n)^{\otimes \ell}$ that is invariant under the action of $\sO(n)$ on this space, i.e., the vectors $w$ such that $Q^{\otimes \ell} w = w$ for all $Q \in \sO(n)$ (not to be confused with our other discussion of invariant polynomials).
In our case we will have $\ell = pd$ for working with $G \in \sG_{d, p}$.
This is sometimes called the \emph{Brauer space} after Brauer's early work on its structure.

In particular, his classic results in the representation theory of Lie groups~\cite{brauer-37} show that this subspace is spanned by vectors of the following form. Assume $\ell$ is even; otherwise, since $-\id \in \sO(n)$, only the zero vector is invariant. Then, for each perfect matching $\mu = \{(t_1,t_1'),\ldots,(t_{\ell/2},t'_{\ell/2})\}$ of $[\ell]$,
define
\begin{equation}
\label{eq:matching-vector}
w(\mu)_{i_1,\ldots,i_\ell}
= \prod_{(t,t') \in \mu} \delta_{i_t,i_{t'}}
\end{equation}
where $\delta$ is the Kronecker delta $\delta_{i,i'}=1$ if $i=i'$ and $0$ otherwise.
In tensor network notation, we can regard these \emph{matching vectors} as $\ell$-ary tensors consisting of a set of ``cups,'' each of which ensures that two indices are identical.
Alternatively, one may view them as tensor powers of the identity matrix, though viewed as a tensor rather than a matrix, i.e.\ $(\sum_{i = 1}^n e_i \otimes e_i)^{\otimes \ell / 2}$, subject to various permutations of the vector axes.
Note in particular that the $w(\mu)$, unlike the tensors discussed earlier, are \emph{not} symmetric tensors; indeed, their symmetrizations are all the same, not depending on the matching $\mu$.

For instance, if $\ell=6$ and $\mu=\{(1,3),(2,5),(4,6)\}$ then $w(\mu) = {\raisebox{-4pt}{\includegraphics[height=12pt]{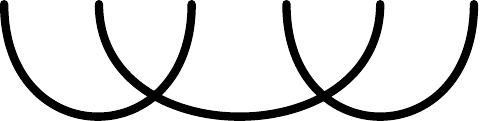}}}$. These vectors are invariant since, if we apply an orthogonal matrix $Q$ to both indices of a cup, $Q$ and $Q^\top$ meet in the middle and cancel (just as we will discuss in Proposition~\ref{prop:moments-are-invariant}).
Brauer's result cited above is that these are \emph{all} of the invariant tensors.

\begin{definition}
    For a finite set $X$, write $\sM(X)$ for the set of perfect matchings of $X$.
\end{definition}
\begin{proposition}
    The subspace of $\sO(n)$-invariant tensors in $(\RR^n)^{\otimes \ell}$ is spanned by the $w(\mu)$ over $\mu \in \sM([\ell])$.
\end{proposition}

We will also need to work with the orthogonal projection to this invariant subspace.
On the one hand, by general representation theory, this object is just the averaging operator over all rotations,
\begin{equation}
    \Pi_{\ell} \colonequals \Ex_{Q} Q^{\otimes \ell},
\end{equation}
viewed as an operator on $(\RR^n)^{\otimes \ell}$ (equivalently, as a large matrix formed as the Kronecker power).
In particular, $\Pi_{\ell}$ contains as its entries the moments of the Haar measure over $\sO(n)$.
This gives an elegant connection between the probability theory of the Haar measure and the representation theory of $\sO(n)$.

On the other hand, by elementary linear algebra, $\Pi_{\ell}$ is given by some linear combination of rank one matrices,
\begin{equation}
    \Pi_{\ell} = \sum_{\mu, \nu} \Wg_{\mu, \nu} w(\mu) \otimes w(\nu).
\end{equation}
Indeed, the matrix $\Wg$ can be taken to be the Moore--Penrose pseudoinverse of the Gram matrix of the $w(\mu)$.
Its entries as a function of $\mu, \nu$ are known as the \emph{(orthogonal) Weingarten function}.

In particular, $\Pi_{\ell}$ therefore admits a diagrammatic description, as a linear combination of operator ``switching'' one matching to another (the outer products of two matching vectors).
For instance, $\Pi_2$ is the ``cupcap''
\[
\Pi_2
= \Exp_Q Q \otimes Q = \frac{1}{n}\, \cupcap \, ,
\]
namely the outer product of $w(\{(1,2)\}) = \ccup$ with itself. The normalization $1/n$ comes from the fact that
\[
\Big| \ccup \Big|^2
= \tr \cupcap
= \cupcaploop
= n \, ,
\]
corresponding to the trace of the identity matrix equaling $n$.
$\Pi_4$ is already more complicated:
\[
\arraycolsep=1.4pt
\def\arraystretch{3}
\label{eq:pi4}
\begin{array}{rrr}
\Pi_4 =
{\displaystyle \frac{1}{(n+2)n(n-1)}}
\left[
(n+1)\; \embed{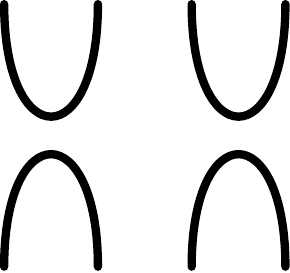}
\right. &
\;-\; \embed{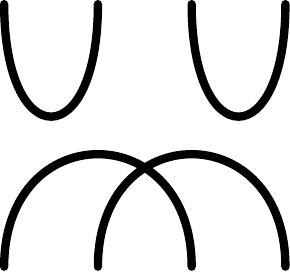} &
\;-\; \embed{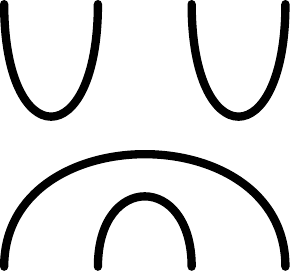}
\;\phantom{\Bigg]}\\
\;-\; \embed{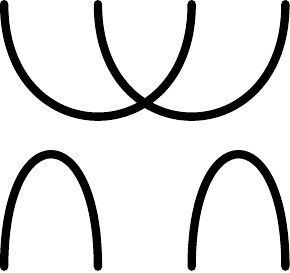} &
\;+ \;(n+1)\; \embed{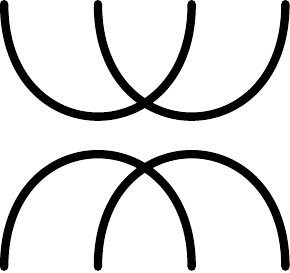} &
\;-\; \embed{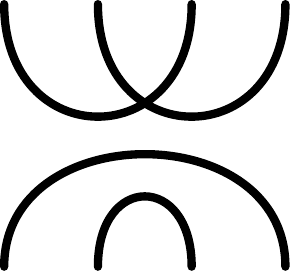}
\;\phantom{\Bigg]} \\
\;-\; \embed{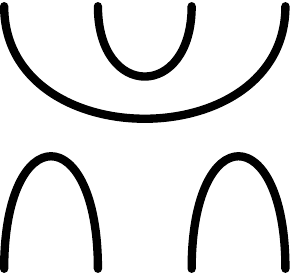} &
\;-\; \embed{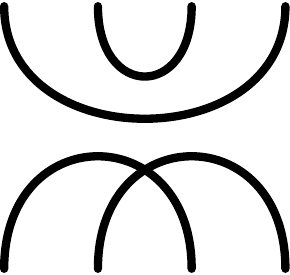} &
\left. \;+ \;(n+1)\; \embed{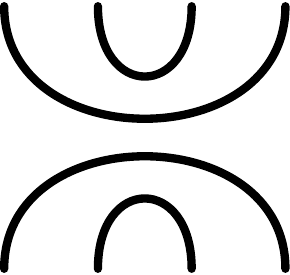} \;\; \right]
\end{array}
\]
Fortunately, there is a combinatorial description of the Weingarten function that aids in computing these coefficients, as we discuss in Appendix~\ref{app:weingarten}.

\subsection{Low-Degree Polynomial Algorithms}
\label{sec:prelim:low-deg}

The idea to consider polynomials as algorithms for statistical problems with degree as a measure of complexity first arose from~\cite{sos-clique,HS-bayesian,sos-detect} and the framework has been subsequently refined over the years, e.g., by~\cite{hopkins-thesis,KWB-2022-LowDegreeNotes}; see, for instance,~\cite{fp} for a modern account. Originally this line of work considered simple ``planted versus null'' hypothesis testing problems, but later extensions treated other styles of questions such as estimation~\cite{SW-2020-LowDegreeEstimation}, optimization~\cite{ld-opt}, and refutation~\cite{coloring-clique}, as well as more complex testing problems~\cite{count-communities}. The low-degree polynomial framework has by now found numerous applications and is a leading approach to understand computational complexity of statistical problems.

We will next specify what it means for a polynomial to solve a detection or reconstruction problem.

\begin{definition}[Low-degree advantage]
    \label{def:adv}
    \begin{equation}
        \Adv_{\leq D}(\QQ, \PP) \colonequals \sup_{\substack{f \in \RR[Y]_{\leq D} \\ \EE_{Y \sim \QQ} f(Y)^2 \neq 0}} \frac{\EE_{Y \sim \PP} f(Y)}{\sqrt{\EE_{Y \sim \QQ} f(Y)^2}}.
    \end{equation}
\end{definition}

\begin{remark}\label{rem:adv}
This notion has been standard since the work of~\cite{HS-bayesian,sos-detect,hopkins-thesis}, where its boundedness, i.e., $\Adv_{\leq D}(\QQ, \PP) = O(1)$, indicates hardness at degree $D$. The value $\Adv$ is also called the \emph{norm of the low-degree likelihood ratio}~\cite{hopkins-thesis} but the terminology $\Adv$ avoids defining likelihood ratios and low-degree projections. While boundedness of $\Adv$ does not completely rule out low-degree polynomials as a test statistic, it means that we cannot use Chebyshev's inequality to show that one succeeds. If $\Adv = \omega(1)$, this may suggest success of low-degree tests, but a number of recent examples show that this is not necessarily the case~\cite{fp,grp-test,subhypergraph,graph-matching}.

In light of this, one should really define ``success'' for low-degree tests not in terms of $\Adv$ but by the notion of ``strong separation'' first coined in~\cite{fp}. This way, ``success'' legitimately implies a high-probability distinguisher, and furthermore ``success'' can be precluded by showing $\Adv = O(1)$ (or even by bounding a conditional variant of $\Adv$). There is also a related notion of ``weak separation''; again, see~\cite{fp}.

Our Theorem~\ref{thm:tensor-pca-detection} (or rather, the formal version Theorem~\ref{thm:tensor-pca-detection-formal}) will focus on showing $\Adv$ is either $O(1)$ or $\omega(1)$, but we expect that our proof of $\Adv = \omega(1)$ could be strengthened to a proof of strong separation.
\end{remark}

For reconstruction tasks, success is naturally measured in terms of mean squared error.

\begin{definition}[Low-degree minimum mean squared error~\cite{SW-2020-LowDegreeEstimation}]
\begin{equation}
    \MMSE_{\leq D}(\PP) \colonequals \inf_{f \in \RR[Y]_{\leq D}^n} \Ex_{(x, Y) \sim \PP} \|f(Y) - x\|^2.
\end{equation}
\end{definition}

Similar to Fact~1.1 of~\cite{SW-2020-LowDegreeEstimation}, this can be equivalently formulated in terms of ``correlation'':
\begin{equation}
\MMSE_{\leq D}(\PP) = \EE\|x\|^2 - \Corr_{\leq D}(\PP)^2
\end{equation}
where $\Corr$ is defined below.

\begin{definition}[Low-degree correlation]
    \begin{equation}
        \Corr_{\leq D}(\PP) \colonequals \sup_{\substack{f \in \RR[Y]_{\leq D}^n \\ \EE_{Y \sim \PP} \|f(Y)\|^2 \neq 0}} \frac{\EE_{(x, Y) \sim \PP} \langle x, f(Y) \rangle}{\sqrt{\EE_{Y \sim \PP} \|f(Y)\|^2}}.
    \end{equation}
\end{definition}

\noindent
In tensor PCA we have $\|x\|^2 = n$, so to rule out non-trivial reconstruction we will aim to show $\Corr_{\leq D}(\PP)^2 = o(n)$.

Finally, to relate low-degree polynomial algorithms to our study of invariant polynomials, we show that both the advantage and correlation may be restricted to invariant polynomials without changing their value, provided the underlying distributions have certain invariance properties.
\begin{proposition}
    \label{prop:adv-invariant}
    Suppose that both $\QQ$ and $\PP$ are invariant distributions.
    Then,
    \begin{equation}
    \Adv_{\leq D}(\QQ, \PP) = \sup_{\substack{f \in \RR[Y]_{\leq D} \\ f \text{ invariant} \\ \EE_{Y \sim \QQ} f(Y)^2 \neq 0}} \frac{\EE_{Y \sim \PP} f(Y)}{\sqrt{\EE_{Y \sim \QQ} f(Y)^2}}.
    \end{equation}
\end{proposition}
\begin{proof}
    Consider $\bar{f}(Y) \colonequals \EE_{Q \sim \Haar(n)} f(Q \cdot Y)$.
    Then, $\bar{f}$ is invariant.
    And, we have
    \begin{align}
        \Ex_{Y \sim \PP} \bar{f}(Y) = \Ex_{Q \sim \Haar(n)} \Ex_{Y \sim \PP} f(Q \cdot Y) = \Ex_{Y \sim \PP} f(Y),
    \end{align}
    by the invariance of $\PP$.
    The same holds for $\QQ$ with $f$ replaced by $f^2$, and the result follows.
\end{proof}

\begin{remark}
    By the same operations on $f$, one may also show that, outside of the low-degree framework but under the same assumptions on $\PP$ and $\QQ$, if there is a test statistic achieving given Type~I and II error probabilities, then there is an invariant test statistic achieving the same.
\end{remark}

\begin{proposition}
    \label{prop:corr-equivariant}
Suppose that $\PP$ is a distribution over $\RR^n \times \Sym^p(\RR^n)$ such that $(x, Y) \sim \PP$ has the same law as $(Qx, Q \cdot Y)$ for any $Q \in \sO(n)$ (that is, so that $\PP$ viewed as a distribution over $x \otimes Y$ is invariant).
    Then,
    \begin{equation}
        \Corr_{\leq D}(\PP) = \sup_{\substack{f \in \RR[Y]_{\leq D}^n \\ f \text{ equivariant} \\ \EE_{Y \sim \PP} \|f(Y)\|^2 \neq 0}} \frac{\EE_{(x, Y) \sim \PP} \langle x, f(Y) \rangle}{\sqrt{\EE_{Y \sim \PP} \|f(Y)\|^2}}.
    \end{equation}
\end{proposition}
\begin{proof}
    Consider $\bar{f}(Y) \colonequals \EE_{Q \sim \Haar(n)} Q^{\top} f(Q \cdot Y)$.
    $\bar{f}$ is equivariant, since
    \begin{equation}
        \bar{f}(R \cdot Y) = \Ex_{Q \sim \Haar(n)} Q^{\top} f(QR \cdot Y) = \Ex_{Q} RQ^{\top}  f(Q \cdot Y) = R \bar{f}(Y).
    \end{equation}
    And we have
    \begin{align*}
        \Ex_{(x, Y) \sim \PP} \langle x, \bar{f}(Y) \rangle
        &= \Ex_{Q \sim \Haar(n)} \Ex_{(x, Y) \sim \PP} \langle x, Q^{\top} f(Q \cdot Y) \rangle \\
        &= \Ex_{Q} \Ex_{(x, Y) \sim \PP} \langle Qx, f(Q \cdot Y) \rangle \\
        &= \Ex_{(x, Y) \sim \PP} \langle x, f(Y) \rangle,
    \end{align*}
    and similarly for the denominator,
    \begin{align*}
        \Ex_{Y \sim \PP} \|\bar{f}(Y)\|^2
        &= \Ex_{Y \sim \PP} \| \Ex_{Q \sim \Haar(n)} Q^{\top} f(Q \cdot Y) \|^2 \\
        &\leq  \Ex_{Q \sim \Haar(n)} \Ex_{Y \sim \PP} \| Q^{\top} f(Q \cdot Y) \|^2 \\
        &= \Ex_{Q \sim \Haar(n)} \Ex_{Y \sim \PP} \| f(Q \cdot Y) \|^2 \\
        &= \Ex_{Y \sim \PP} \| f(Y) \|^2
    \end{align*}
    where we have used Jensen's inequality followed by the invariance of the marginal distribution of $\PP$ on $Y$.
    The result then follows.
    We note here that, because of the inequality in the last calculation, $\bar{f}$ can actually have \emph{strictly larger} objective function in the correlation than $f$; on the other hand, the equivariant polynomials are included in the original maximization for the correlation, so the stated result is in fact an exact formula and not just a bound for $\Corr_{\leq D}(\PP)$.
\end{proof}

\begin{remark} When $p$ is even, the reconstruction problem has an ambiguity since $v$ is only identifiable up to a sign. In this case we can define $g(Y)$ as returning $v \otimes v$, or the entire  rank-1 tensor $v^{\otimes p}$. More generally if we define the accuracy of $g$ as
$\langle g(Y), v^{\otimes p} \rangle$, the same equivariance argument applies.
\end{remark}

\section{Tensorial Finite Free Cumulants and Invariant Bases}
\label{sec:cumulants}

\subsection{Frobenius Pairs}
\label{sec:frobenii}

A special role in our techniques will be played by the following multigraph, which is called the ``melon'' in high energy physics.\footnote{A related recursive structure gives the class of ``melonic graphs'' (the melon among them) that play an important role in other diagrammatic computations; see, e.g., \cite{Evnin-2021-MelonicDominanceLargestEigenvalueTensor} or Chapter 4 of \cite{Gurau-2017-RandomTensors}.}

\begin{definition}
    The \emph{Frobenius pair} $F$ of degree $p$ is the $p$-regular multigraph on two vertices with $p$ parallel edges connecting the vertices.
\end{definition}
\noindent
We call this the Frobenius pair because the associated graph moment is the Frobenius norm:
\begin{equation}
    m_F(T) = \|T\|_F^2 = \sum_{i_1,\ldots,i_p
    \in [n]} T_{i_1,\ldots,i_p}^2.
\end{equation}
For instance, for a 3-ary tensor we have
\[ {\raisebox{-16pt}{\includegraphics[height=36pt]{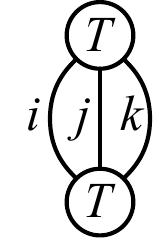}}} = \sum_{i,j,k} T_{ijk}^2 \, . \]
For 2-ary tensors $F$ is a 2-cycle, and we recover the matrix case
\[
{\raisebox{-16pt}{\includegraphics[height=36pt]{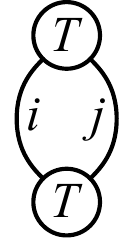}}} = \|T\|_F^2 = \tr (T^\top T) \, .
\]
For multigraphs that include multiple Frobenius pairs as connected components, we use ``Frobenii'' as a plural and urge others to do the same.

As we have mentioned in the Introduction, the issue that Frobenii raise is that they are the only connected graph moment whose terms are not centered for generic (or, say, Gaussian as in the Wigner law) tensors $T$.
This will create an obstruction to orthogonality for our initial definition of finite free cumulants, which we will need to circumvent with an additional centering operation.

\subsection{Distinct-Index Graph Moments}

In order to define tensor cumulants, we first follow recent work~\cite{Wein-2022-LowDegreeTensorDecomposition} that modifies graph moments by requiring that the indices on the edges be distinct.

\begin{definition}
\label{def:distinct-uncentered}
Given a $p$-ary symmetric tensor $T$ and a $p$-regular multigraph $G$ with $|E|=b$ edges, define
\begin{equation}
\label{eq:distinct-uncentered}
m^!_G(T) =
\sum_{\substack{i \in [n]^E \\ \textrm{$i_1,\ldots,i_b$ distinct}}}
\prod_{v \in V}
T_{i(\partial v)} \, ,
\end{equation}
where the ! connotes distinctness.
\end{definition}
\noindent
Clearly $m^!_G(T)=0$ unless $n \ge b$. As with ordinary graph moments, we define $m^!_G(T)=1$ if $G$ is empty.

\begin{remark}
If $T$ is a random tensor whose law is invariant under conjugation by orthogonal matrices, or even just conjugation by permutation matrices, then every term in the sum arising from $\Exp m^!_G(T)$ is equal.
In that case, $\Exp m^!_G(T)$ is just the number of distinct tuples of indices, i.e.,
$n^{\underline{b}} = n(n-1)(n-2)\cdots(n-b+1)$, times the expectation of any one of them.
\end{remark}

Requiring distinct indices has an important consequence: except in a Frobenius pair, vertices cannot have matching edge neighborhoods. To deal with Frobenii we also define a centered version of $m^!_G$ as follows.

\begin{definition}
\label{def:distinct-centered}
Let $T$ be a $p$-ary symmetric tensor and let $G=(V,E)$ be a $p$-regular multigraph with $|E|=b$ edges. Let $U$ be the set of  vertices forming Frobenius pairs, and for a given pair $\pi = \{u,v\} \subseteq U$ let $i(\pi) = i(\partial u) = i(\partial v)$ be the $p$ indices appearing on the edges of $\pi$. Then define
\begin{equation}
\label{eq:distinct-centered}
m^c_G(T) =
\sum_{\substack{i \in [n]^E \\ \textrm{$i_1,\ldots,i_b$ distinct}}}
\left[ \prod_{v \in V \setminus U}
T_{i(\partial v)}
\prod_{\pi \subseteq U}
\left( T_{i(\pi)}^2 - 1 \right)
\right] \, .
\end{equation}
\end{definition}

\noindent
This definition centers the Frobenii so that $m^c_G$ has zero expectation in the Wigner model for all $G$.  Recall that we abbreviate $\Wig(p,n,\sigma^2)$ as $\Wig$ when $\sigma^2=1$ and $p$ and $n$ are clear.

\begin{proposition}
\label{prop:distinct-zero}
Let $G$ be a nonempty $p$-regular multigraph, and let $W \sim \Wig$. Then
\[
\Exp_{W} m^c_G(W) = 0 \, .
\]
\end{proposition}

\begin{proof}
We will show that the summand in~\eqref{eq:distinct-centered} has zero expectation for each fixed set of indices $i \in [n]^E$. If $G$ has a connected component which is not a Frobenius pair, then the $W_{i(\partial v)}$ in the first product are all distinct and independent for different $v \in V \setminus W$. They are also independent from the $W_{i(\pi)}$, which are independent from each other for different Frobenii $p$.

Thus the expectation is a product of expectations. For any $i \in [n]^p$ we have $\Exp W_i = 0$, so the first product has zero expectation if $V \setminus U$ is nonempty. For the second products, the $p$ indices in $i(\pi)$ are distinct, and therefore $\Exp W_{i(\pi)}^2 = 1$. Thus the second product has zero expectation if $G$ has any Frobenii.
\end{proof}

Note what has happened here. As in
Proposition~\ref{prop:even-colorings}, the only nonzero contributions to $\Exp_W m_G(W)$ correspond to even colorings of $G$'s edges, where vertex neighborhoods $i(\partial v)$ are repeated. By requiring all indices to be distinct, we force the vertex neighborhoods to be distinct, except on Frobenius pairs where we center the normal distribution.

More generally, in planted models $Y=T+W$ where $W$ is Wigner noise, the expectation of $m^c_G(Y)$ is simply the distinct-index moment of the signal $T$. This addresses one of the drawbacks of ordinary graph moments, namely that their expectations in the spiked model involves mixed terms with the signal at some vertices and the noise at others.

\begin{proposition}
\label{prop:spike-distinct}
For any $p$-regular multigraph $G$ and $T \in \Sym^p(\RR^n)$ we have
\begin{equation}
\label{eq:spike-distinct}
\Ex_{W \sim \Wig} m^c_G(T + W) = m^!_G(T) \, .
\end{equation}
\end{proposition}

\begin{proof}
Fix $i \in [n]^E$. Then as in Proposition~\ref{prop:distinct-zero}, the entries $Y_{i(\partial v)}$ at distinct vertices $v$ which are not part of Frobenius pairs are independent, and have expectation $T_{i(\partial v)}$. Similarly, the terms $Y_{i(\pi)}^2-1$ on distinct Frobenius pairs are independent, and have expectation
\[
\Exp[(T+W)_{i(\pi)}^2 - 1]
= T_{i(\pi)}^2 + 2 T_{i(\pi)} \Exp W_{i(\pi)} + \Exp[W_{i(\pi)}^2] - 1
= T_{i(\pi)}^2 \, ,
\]
which for a pair $\pi=(u,v)$ is the same as $T_{i(\partial u)} T_{i(\partial v)}$. Since the expectation of a product of independent variables is the product of their expectations, we are left with $\prod_{v \in V} T_{i(\partial v)} = m^!_G(T)$.
\end{proof}

Helpfully, these centered distinct-index moments have zero covariance in the Wigner model if their graphs are nonisomorphic. That is, they are orthogonal with respect to the inner product $\langle f, g \rangle = \Exp_{T \sim \Wig} f(T) \,g(T)$.

\begin{proposition}
\label{prop:centered-covariance}
Let $G, H$ be $p$-regular multigraphs. If they are not isomorphic, then
\[
\Ex_{W \sim \Wig} m^c_G(W) m^c_H(W) = 0 \, .
\]
\end{proposition}

\begin{proof}
With the restriction that the edge indices in each graph are distinct, we have a sum of products like that in~\eqref{eq:distinct-centered}, but where $i(G)$ and $i(H)$ are each tuples of distinct indices. (This is not to be confused with $m^c_{G \sqcup H}$, in which case all $|E_G| + |E_H|$ indices would be distinct.)

Fix $i \in [n]^{E_G \sqcup E_H}$. We will show that if $i$ makes a nonzero contribution to $\EE m^c_G(W) m^c_H(W)$, then $G$ and $H$ are isomorphic.
First, if $v \in V_G$ is not part of a Frobenius pair, then $T_{i(\partial v)}$ appears exactly once in $G$, and similarly if $v \in V_H$ is not in a Frobenius pair. Thus each such vertex in $G$ must be matched with one in $H$, with corresponding indices on their edges. This establishes an isomorphism between the components of $G$ with those of $H$ that do not consist of Frobenii.

It remains to prove that $G$ and $H$ have the same number of Frobenii and are hence isomorphic. Suppose without loss of generality that $G$ has more Frobenii than $H$ does. Then for some Frobenius pair $\pi=(u,v)$ in $G$, the indices $i(\pi)$ do not appear in $H$. Since the entries of $T$ are independent, this contributes $\Exp[T^2_{i(\pi)}-1] = 0$ to the product, in which case $i$'s contribution to $\EE m^c_G(W) m^c_H(W)$ is zero.

Thus if any $i$ makes a nonzero contribution to the inner product, $G \cong H$.
\end{proof}

We can also compute the variances of the centered moments by counting a slightly nontraditional kind of graph automorphism. These automorphisms will appear again in the inner products of cumulants.

\begin{definition}
    \label{def:eAut}
    Let $G=(V,E)$ be a multigraph. Give each edge $e \in E$ an arbitrary direction, and write it as an ordered pair $(e_1,e_2)$. Then let an \emph{edge automorphism} be a one-to-one mapping $\phi:E \to E$ such that, for all $e, e' \in E$, and for each $j,j' \in \{1,2\}$, $\phi(e)_j = \phi(e')_{j'}$ if and only if $e_j = e'_{j'}$. In other words, $\phi$ preserves whether $e$ and $e'$ share their ``heads'' or their ``tails'', or whether one's head is the other's tail. Note that $\phi(e)=e$ means that $\phi(e)$ is either $(e_1,e_2)$ or $(e_2,e_1)$, where the latter reverses $e$'s direction. Finally,
    let $\eAut(G)$ be the group of such mappings.
\end{definition}
\noindent

Clearly an edge automorphism determines a traditional vertex automorphism, i.e., a mapping $\psi:V \to V$ such that $(\psi(u),\psi(v)) \in E$ if and only if $(u,v) \in E$. For simple graphs, the equivalence between the two types of automorphism is one-to-one. This more elaborate definition includes permuting parallel edges in a multigraph, as well as reversing the direction of a self-loop without moving its endpoint. In terms of the group $\Aut(G)$ of vertex automorphisms, we have
\begin{equation}
\label{eq:mult}
|\eAut(G)|
= 2^{\#\{\text{self-loops in $G$} \}}
\; \times \prod_{\substack{\textrm{bundles of $t$}\\ \textrm{parallel edges}}}
\hspace{-10pt} t! \;
\times \; |\Aut(G)| \, .
\end{equation}

\begin{proposition}
\label{prop:distinct-variance}
Let $G$ be a $p$-regular multigraph with
$b$ edges. Then
\begin{equation}
\label{eq:distinct-variance}
\Ex_{W \sim \Wig} {m^c_G(W)}^2
= n^{\underline{b}}
\, |\eAut(G)| \, .
\end{equation}
\end{proposition}

\begin{proof}
First suppose that $G$ has no Frobenii or self-loops. The nonzero contributions to the second moment consist of even colorings of the disjoint union of two copies of $G$, but now where the colors within each copy of $G$ are distinct. The only such colorings are those like Figure~\ref{fig:single-variance}, where each edge $e$ on the left copy of $G$ has a counterpart $\phi(e)$ on the right with the same index, and where each vertex $v$ on the left has a counterpart $\psi(v)$ on the right with the same neighborhood. This describes a vertex automorphism and a permutation of each bundle of parallel edges. Along with the $n!/(n-b)!$ distinct labelings of the left copy, we obtain
\[
\Ex_{W \sim \Wig} {m^c_G(W)}^2
= n^{\underline{b}}
\prod_{\substack{\textrm{bundles of $t$}\\ \textrm{parallel edges}}}
\hspace{-10pt} t! \;
\times \; |\Aut(G)| \, .
\]

Now suppose that some vertex $v$ has $t$ self-loops. Their indices can be permuted, but we already have the factor $t!$ since these count as a bundle of parallel edges. The other effect is that the entry $T_{i(v)}$ has variance larger than $1$ since some indices are repeated. Specifically, if $i(v)$ contains $t$ pairs of repeated indices, by Proposition~\ref{prop:wig-variance} we have $\Exp T_{i(v)}^2 = 2^t$. Taking the product over all vertices gives the factor $2^{\#\{\text{self-loops in $G$} \}}$.

Finally we consider the possibility that $G$ contains one or more Frobenii. As in Proposition~\ref{prop:centered-covariance}, for each Frobenius pair $\pi$ in the left copy of $G$, its indices $i(\pi)$ must appear in a counterpart pair on the right; otherwise $T_{i(\pi)}$ is independent of the other entries, and centering gives a factor of $\Exp[T_{i(\pi)}^2-1]=0$. The $p$ parallel edges of $\pi$ and its counterpart can be matched in $p!$ ways, and for each matching they contribute a factor of
\[
\Exp\!\left[ \left(T_{i(\pi)}^2-1 \right)^{\!2} \right] = 2 \, ,
\]
where we use the fact that $T_{i(\pi)} \sim \sN(0,1)$. Finally, if we have $t$ Frobenius pairs, they can be matched with each other in $t!$ ways, and their total contribution is $(2 p!)^t t!$. But this is also the number of edge automorphisms of the disjoint union of $t$ Frobenii, so~\eqref{eq:distinct-variance} continues to hold.
\end{proof}

\subsection{Finite Free Cumulants and Additive Free Convolution}

The orthogonality of the distinct-index graph moments suggests that they are a good basis for the space of polynomial functions of $T$ with respect to the Gaussian inner product. However, they are not invariant with respect to orthogonal basis changes. This motivates the following definition, where we symmetrize them with random orthogonal rotations.

\begin{definition}
Let $T$ be a symmetric $p$-ary tensor and let $G$ be a $p$-regular multigraph. Define the \emph{free cumulant} $\kappa_G(T)$ as
\begin{equation}
\label{eq:graph-cumulant}
    \kappa_G(T) = \Exp_Q \left[ m^!_G(Q \cdot T) \right] \, ,
\end{equation}
and the \emph{centered free cumulant} $\kappa^c_G(T)$ as
\begin{equation}
\label{eq:centered-graph-cumulant}
    \kappa^c_G(T) = \Exp_Q \left[ m^c_G(Q \cdot T) \right] \, ,
\end{equation}
where $Q$ is a Haar-random orthogonal matrix. If $G$ is the empty graph then $\kappa_G=\kappa^c_G=1$.
\end{definition}
\noindent

\begin{remark}
\label{rem:counting-tuples}
Recall that $m^!_G(T)$ and $m^c_G(T)$ are both summations over assignments of distinct indices to the edges of $G$. If there are $b$ edges, then there are $n^{\underline{b}}$ such assignments. After symmetrizing $T$ to $Q \cdot T$ and taking the expectation over $Q \sim \Haar(n)$, by symmetry each of these terms will be equal.
If we view the original summation as splitting each edge in half and connecting either end to the matrix $e_i \otimes e_i$, where $i$ is the label that that edge gets in that term of the summation, then the symmetrization yields an expectation where the $e_1, \dots, e_n$ are replaced by a Haar-random orthonormal basis of $\RR^n$.
In notation, we may fix $i \in [n]^E$ an arbitrary distinct index labelling of the edges, in which case we have
\begin{align}
    \kappa_G(T) &= n^{\underline{b}} \Ex_{Q \sim \Haar} \prod_{v \in V} (Q \cdot T)_{i(\partial v)}, \\
    \kappa_G^c(T) &= n^{\underline{b}} \Ex_{Q \sim \Haar}  \prod_{v \in V(G \setminus \Frob(G))} (Q \cdot T)_{i(\partial v)} \prod_{F \in \Frob(G)} ((Q \cdot T)^2_{i(F)} - 1).
\end{align}
\end{remark}

These functions $\kappa_G$ are manifestly invariant, and $\kappa_G$ coincides with $\kappa^c_G$ unless $G$ includes one or more Frobenii. We call them free cumulants because, if $p=2$ and $G$ is the cycle of size $t$, they coincide---up to an appropriate scaling factor---with the classic free cumulant $\kappa_t$ of free probability in the limit $n \to \infty$ (e.g.~\cite[Eq.27]{Maillard_2019}). If $G$ is a disjoint union of two or more cycles, the $n \to \infty$ limit yields the so-called \emph{higher-order free cumulants}~\cite{collins-mingo-sniady-speicher} which~\cite{semerjian2024matrix} used recently for orthogonally-invariant denoising problems in matrices.

One of the most useful properties of the free cumulants is additivity~\cite{Novak2011WhatIF}.
In classical statistics, cumulants are polynomials $\kappa$ in the moments of a random variable with the property that, if $A$ and $B$ are independent random variables, then the cumulant of their convolution $A+B$ is the sum of their cumulants: $\kappa(A+B)=\kappa(A)+\kappa(B)$. For instance, the first, second, and third cumulants are the expectation $\kappa_1(A) = \Exp[A]$, the variance $\kappa_2(A) = \Exp[A^2]-\Exp[A]^2$, and $\kappa_3(A) = \Exp[A^3] - 3 \Exp[A^2] \Exp[A] + 2 \Exp[A]^3$. The $t$th cumulant $\kappa_t$ is homogeneous in the sense that it is a linear combination of products of moments with total order $t$.

In the nonabelian setting of free probability (e.g.~\cite{mingo-speicher}), $A$ and $B$ are matrix-valued random variables in $\RR^{n \times n}$, and a central question is to determine the typical spectrum of $A + Q^\top B Q$ where $Q \in \sO(n)$ is a Haar-random orthogonal matrix. (One can also study complex-valued matrices where $R$ is unitary.) While $A + Q^\top B Q$ is a random matrix even for fixed $A$ and $B$, in the limit $n \to \infty$ its empirical spectral distribution converges to a particular probability measure under mild conditions provided the same convergence holds for $A$ and $B$.
Moreover, this limiting measure is a function of the limiting spectral distributions of $A$ and $B$ individually. If we call these distributions $\alpha$ and $\beta$ respectively, then the limiting law of the spectrum of $A + Q^{\top}BQ$ is written $\alpha \boxplus \beta$ and called the \emph{free convolution} of $\alpha$ and $\beta$. The free cumulants are like the classical cumulants in that they are polynomials in the moments of a measure, but they are additive with respect to this free convolution operation rather than ordinary convolution (which is the operation on distributions induced by addition of independent scalar random variables), i.e.,
\begin{equation}
\label{eq:free-additivity}
\kappa_t(\alpha \boxplus \beta) = \kappa_t(\alpha)+\kappa_t(\beta) \, .
\end{equation}
The analogy with classical cumulants is that, since conjugating by $Q$ randomizes $B$'s eigenvectors, $A$ and $Q^\top B Q$ are as independent as two matrices can be given their spectra. We say that they are \emph{freely independent}.

We call the $\kappa_G$ free cumulants of a tensor both because they coincide with matrix free cumulants in the limit $n \to \infty$ when $G$ is a cycle, and because they satisfy an exact additivity property even for finite $n$.\footnote{While it seems technical to verify, it is natural to conjecture that in the case $p = 2$, our $\kappa_{C_k}$ for $C_k$ the $k$-cycle should be closely related to the finite free cumulants defined by \cite{AP-2018-CumulantsFiniteFreeConvolution} in terms of characteristic polynomials.} The latter is closely related to the additivity property for distinct index moments we gave in Proposition~\ref{prop:spike-distinct}.

\begin{proposition}
\label{prop:additivity}
Let $G=(V,E)$ be a connected $p$-regular graph with $b=|E|$ edges, and let $A$ and $B$ be symmetric $p$-ary tensors. Let $Q \in \sO(n)$ be a Haar-random orthogonal matrix. Then
\begin{equation}
\label{eq:additivity}
\Exp_Q \kappa_G(A + Q \cdot B) = \kappa_G(A) + \kappa_G(B) \, .
\end{equation}
More generally, if $G$ is not necessarily connected,
\begin{align}
\Exp_Q \kappa_G(A + Q \cdot B)
&= \sum_{G_A \sqcup G_B = G} \frac{n^{\underline{b}}}{n^{\underline{b_A}}n^{\underline{b_B}}}
 \,\kappa_{G_A}(A) \, \kappa_{G_B}(B). \label{eq:additivity-components}
\end{align}
Here the sum is over pairs of graphs $G_A=(V_A,E_A)$ and $G_B=(V_B,E_B)$ such that $G=G_A \sqcup G_B$, and where $b_A=|E_A|$ and $b_B=|E_B|$ with $b_A+b_B=b$. In particular, $G_A$ and $G_B$ are disjoint unions of connected components of $G$.
\end{proposition}

\begin{proof}
We start with connected graphs. We will prove a partial additivity for the distinct-index moments,
\begin{equation}
\label{eq:half-add}
\Exp_Q m^!_G(A + Q \cdot B) = m^!_G(A) + \kappa_G(B) \, .
\end{equation}
Symmetrizing and replacing $A$ with $R \cdot A$ for a Haar-random $R$, and using the fact that $QR$ is also Haar-random, yields~\eqref{eq:additivity}.

\begin{figure}
    \centering
    \includegraphics[width=3.5in]{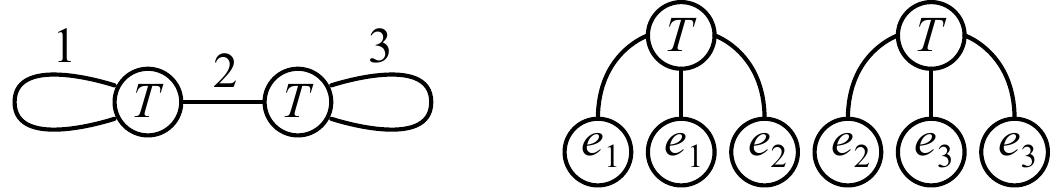}
    \caption{A given term in~\eqref{eq:distinct-uncentered} corresponds to placing basis vectors $e_i$ on each edge $i$, and contracting them with $i$'s endpoints. That is, $m!_G(T)$ is the sum of the diagram on the right over all distinct tuples  $e_1, e_2, e_3$ in the orthogonal basis.}
    \label{fig:distinct-indices}
\end{figure}

First we give a more diagrammatic picture of $m^!_G$. Extracting a given product of entries in~\eqref{eq:distinct-uncentered} corresponds to placing a basis vector $e_i$ on each edge $i$---or rather two copies of $e_i$, pointing in each direction---and contracting these vectors with the tensors at their endpoints as in Figure~\ref{fig:distinct-indices}.

Expanding $m^!_G(A+B)$ gives $2^{|V|}$ terms, where $A$ appears on some vertices and $B$ appears on the others. (Except for the distinct-index requirement, these cross-terms are mixed graph moments as in Definition~\ref{def:mixed-moments}.) Call these two sets of vertices $V_A$ and $V_B$ respectively.

Now consider the effect of replacing $B$ with $Q \cdot B$. This applies the projection operator $\Pi_\ell = \Exp Q^{\otimes \ell}$ defined in~\eqref{eq:pi-ell} with $\ell=d|V_B|$, which is a linear combination of outer products of matching vectors. However, unlike in the proof of Theorem~\ref{thm:tensor-invariant}, these matchings only apply to the half-edges incident to vertices in $V_B$.
Since $G$ is connected, if $V_A$ and $V_B$ are both nonempty, there is at least one edge $i$ in the cut between them.
But in each matching contributing to $\Pi_\ell$, the copy of $e_i$ incident to $V_B$ is matched with $e_j$ for some other edge $j$ incident to $V_B$ as shown in Figure~\ref{fig:additivity}. Since $\langle e_i, e_j \rangle = 0$, any such term is zero.

This implies that the only nonzero contributions to $m^!_G(A+Q \cdot B)$ come from the terms where $V_A$ or $V_B$ is empty, i.e., where the vertices are either all labeled with $A$ or all labeled with $Q \cdot B$. But these terms are exactly $m^!_G(A)$ and $\kappa_G(B)$ as stated.

For graphs with multiple connected components, the same argument implies that each component is labeled entirely with $A$ or with $B$. Call the unions of these components $G_A$ and $G_B$ respectively; then each such partition contributes $\kappa_{G_A}(A) \,\kappa_{G_B}(B)$. Since all the indices in $G$ are distinct, not just those within $G_A$ or $G_B$, we need to divide the prefactor $n^{\underline{b}}$
appearing in $\kappa_G$ (see Remark~\ref{rem:counting-tuples}) by the analogous prefactors
$n^{\underline{b_A}}$ and $n^{\underline{b_B}}$
for $\kappa_{G_A}$ and $\kappa_{G_B}$. This yields~\eqref{eq:additivity-components}.
\end{proof}

\begin{figure}
    \centering
    \includegraphics[width=6in]{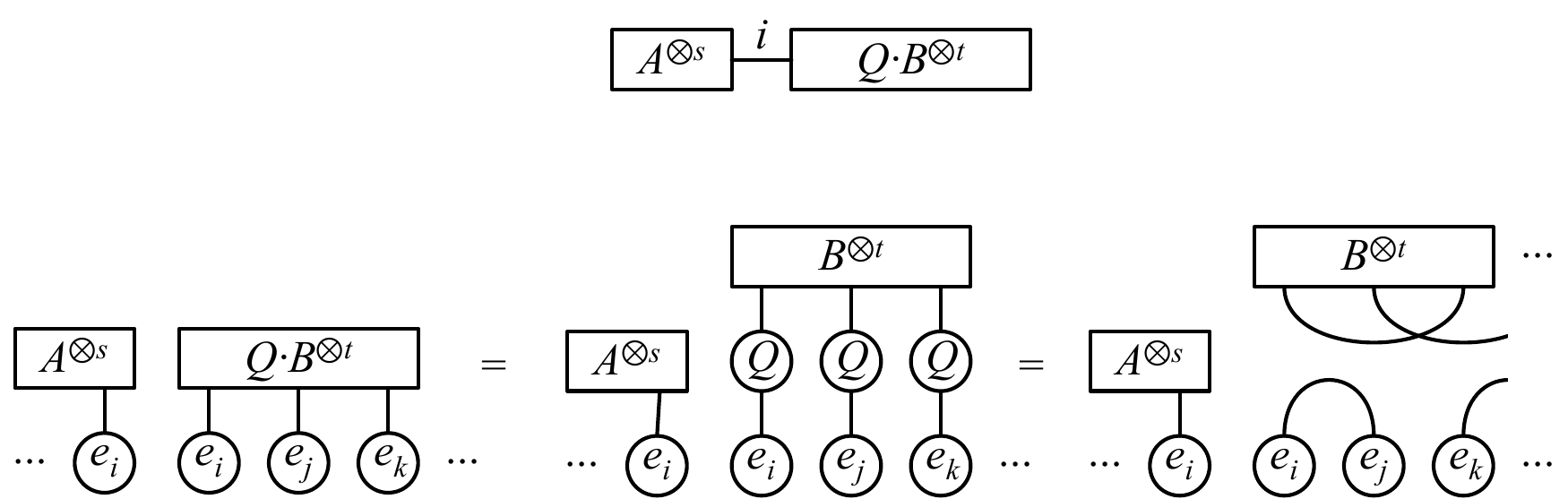}
    \caption{Illustrating the proof of Proposition~\ref{prop:additivity}. As in~\eqref{eq:half-add}, we consider $m^!_G(A + Q \cdot B)$. Expanding the tensor product creates $2^|V|$ cross-terms, where vertices in $V_A$ and $V_B$ are labeled with $A$ and $Q \cdot B$ respectively. Applying $Q$ to the half-edges of $V_B$ creates a linear combination of matchings. But if any edge $i$ crosses from $V_A$ to $V_B$, this matching creates an inner product between its vector $e_i$ with some other $e_j$, contributing $\langle e_i, e_j \rangle$. Thus only partitions where $V_A$ and $V_B$ are unions of connected components of $G$ contribute to $m^!_G(A + Q \cdot B)$. In particular, if $G$ is connected then either $V_A=V$ and $V_B=\emptyset$ or vice versa.}
    \label{fig:additivity}
\end{figure}

In fact, this additivity phenomenon is more general and a version of it is enjoyed both by the centered $\kappa_G^c$ and further generalizations thereof.
To present this version, it is convenient to work over a more general family of centered cumulants.
\begin{definition}
    Write $\conn(G)$ for the set of connected components of $G$.
    Given $x \in \RR^{\conn(G)}$, define
    \begin{align*}
        m_G^c(T; x) &\colonequals \sum_{\substack{i \in [n]^E \\ i_1, \dots, i_b \text{ distinct}}} \prod_{C \in \conn(G)}\left(\prod_{v \in C} T_{i(\partial v)} - x_C\right), \\
        \kappa_G^c(T; x) &\colonequals \Ex_Q m_G^c(Q \cdot T; x).
    \end{align*}
\end{definition}
\noindent
Our definition from earlier is $m^c_G(T) = m^c_G(T; x)$ where $x_C = -\One\{C \text{ is a Frobenius pair}\} \colonequals -\One_{\Frob}$.

\begin{proposition}[General additivity after centering]
    \label{prop:general-additivity}
    Suppose $x = x_A + x_B$.
    Then,
    \begin{equation}
        \Ex_{Q} \kappa_G^c(A + Q \cdot B; x) = \sum_{G_A \sqcup G_B = G} \frac{n^{\underline{b}}}{n^{\underline{b_A}} n^{\underline{b_B}}} \kappa_{G_A}^c(A; x_A) \kappa_{G_B}^c(B; x_B).
    \end{equation}
\end{proposition}
\noindent
The proof is essentially identical to that of Proposition~\ref{prop:additivity}.

Proposition~\ref{prop:additivity} is recovered as the special case $x = x_A = x_B = 0$, which is formally pleasing but not useful for analyzing low-degree polynomials, because one of the free cumulants is the symmetrized distinct index Frobenius pair $\sum_{i \text{ distinct}} (Q \cdot T)_{i}^2$, which is non-zero for any non-zero tensor.
But, by choosing $x, x_A,$ and $x_B$ prudently, we may arrange for even the expected centered free cumulants with $G$ including Frobenii to vanish exactly.

In particular, when we are using this with $B \sim \Wig$, then it is most natural to take $x_B = x = -\One_{\Frob}$ and $x_A = 0$, because of the following.
\begin{proposition}
    \label{prop:wigner-centered-cumulants-zero}
    For any $G$, $\EE_{W \sim \Wig} \kappa_G^c(W) = 0$.
\end{proposition}
\begin{proof}
    By orthogonal invariance of the law $\Wig$, we have $\EE_{W \sim \Wig} \kappa_G^c(W) = \EE_{W \sim \Wig} m_G^c(W)$, and the result then follows from Proposition~\ref{prop:distinct-zero}.
\end{proof}

However, for other applications other ways of ``distributing'' the correction $x$ are also useful, as we will see in Section~\ref{sec:clt} on computational central limit theorems and distinguishing Wigner from Wishart tensors.

\subsection{Graph Moment Expansion of Finite Free Cumulants}

Since the $\kappa_G$ are invariant, by Theorem~\ref{thm:tensor-invariant} they must admit an expansion in the graph moments $m_H$.
Below we give an explicit such expansion, which will also prove useful below.
\begin{definition}[Graph Weingarten function]
    \label{def:graph-weingarten}
    Say that a perfect matching $\mu$ of $[dp]$ \emph{realizes} a $p$-regular multigraph $G$ on $d$ unlabelled vertices if, when $\mu$ is viewed as a graph on $[dp]$ and the sets of vertices $\{1, \dots, p\}, \{p + 1, \dots, 2p\}$ and so forth are each identified into a single vertex, then the resulting graph is isomorphic to $G$.
    For two $p$-regular graphs $G, H$ on $a$ vertices, we write
    \begin{equation}
        \Wg_{G, H} \colonequals \sum_{\substack{\mu \text{ realizes } G \\ \nu \text{ realizes } H}} \Wg_{\mu, \nu}.
    \end{equation}
    As a convention, if $G$ and $H$ have different numbers of vertices, we set $\Wg_{G, H} \colonequals 0$.
\end{definition}

To attend to the normalizations that will appear in the calculations below, the following will be useful.
\begin{proposition}
    \label{prop:number-realizing-matchings}
    The number of matchings realizing $G \in \sG_{d, p}$ is $p!^d d! / |\eAut(G)|$.
\end{proposition}
\begin{proof}
    We may expand the claimed quantity as
    \begin{equation}
        \frac{p!^d d!}{|\eAut(G)|} = \frac{d!}{|\Aut(G)|} \cdot \frac{p!^d}{2^{\#\{\text{self-loops in } G\}} \prod_{\text{bundles of } t \text{ parallel edges in } G} t!}.
    \end{equation}
    Here the first term is the number of permutations of vertices that map any matching realizing $G$ to another different matching realizing $G$, and the second term is the same for permutations of half-edges around each vertex.
\end{proof}

\begin{lemma}
    \label{lem:kappa-expansion}
    For any $p$-regular multigraph $G$ with $d$ vertices and $b$ edges,
    \begin{equation}
        \kappa_G(T) = \frac{n^{\underline{b}}}{p!^{d} d!} \cdot  |\eAut(G)| \cdot \sum_H \Wg_{G, H} m_H(T).
    \end{equation}
    The sum may be taken over all $p$-regular multigraphs $H$, but only those with $d$ vertices will contribute.
\end{lemma}
\begin{proof}
    The graph moments may be defined in terms of matching vectors: if $\mu$ is a matching realizing $G$, then $m_G(T) = \langle T^{\otimes d}, w(\mu) \rangle$.
    To be more explicit, for $i \in [n]^{pd}$, call $i$ \emph{compatible with} $\mu$ if $i$ is constant on each pair matched under $\mu$.
    Then, $w(\mu) = \sum_{i \text{ compatible with } \mu} e_i$, which recovers Definition~\ref{def:moments} of $m_G(T)$.

    Let us describe a variant of this for the distinct index moments.
    We say that $i$ is \emph{distinctly compatible with $\mu$} if $i$ is compatible with $\mu$ and if $i$ is distinct on distinct pairs in $\mu$: that is, the level sets of $i$ are exactly the pairs in $\mu$.
    Then if we define
\[
w^!(\mu) \colonequals \sum_{\substack{\text{$i$  distinctly} \\ \text{compatible with $\mu$}}} e_i \, ,
\]
then by the same token we have
\[ m^!_G(T) = \langle T^{\otimes d}, w^!(\mu) \rangle. \]

Moreover, in the symmetrization we may take the adjoint of the application of $Q$ and find
\begin{align*}
\kappa_G(T)
&= \Ex_Q m^!_G(Q \cdot T) \\
&= \Ex_Q \langle Q \cdot T^{\otimes d}, w^!(\mu) \rangle \\
&= \Ex_Q \langle T^{\otimes d}, Q^{\top} \cdot w^!(\mu) \rangle \\
&= \Ex_Q \langle T^{\otimes d}, Q^{\otimes dp} w^!(\mu)\rangle \\
&= \sum_{\substack{\text{$i$  distinctly} \\ \text{compatible with $\mu$}}} \Ex_Q \langle T^{\otimes d}, Q^{\otimes dp} e_i \rangle
\intertext{and using that every term is equal and does not depend on $i$ (as may be seen by multiplying $Q$ by a permutation matrix),}
&= n^{\underline{b}} \, \Ex_Q \langle T^{\otimes d}, Q^{\otimes dp} e_{\mu} \rangle \\
&= n^{\underline{b}} \, \Ex_Q \langle T^{\otimes d}, \Pi_d e_{\mu} \rangle
\intertext{where we write $e_{\mu} = e_i$ for some choice of $i = i(\mu)$ compatible with $\mu$. Next, substituting in the Weingarten function expression for $\Pi_d$ (see Appendix~\ref{app:weingarten}) gives}
&= n^{\underline{b}} \, \sum_{\rho, \nu} \Wg_{\rho, \nu} \, \langle w(\rho), e_{\mu} \rangle \, \langle T^{\otimes d}, w(\nu) \rangle
\intertext{and we have $\langle w(\rho), e_{\mu} \rangle = \One\{\rho = \mu\}$, so this reduces to}
&= n^{\underline{b}} \, \sum_{\nu} \Wg_{\mu, \nu} \, \langle T^{\otimes d}, w(\nu) \rangle \\
&= n^{\underline{b}} \, \sum_{\nu} \Wg_{\mu, \nu} \, m_{G(\nu)}(T),
\intertext{where, as before, $G(\nu)$ is the graph resulting from identifying groups of $p$ vertices suitably in $\nu$. We notice that a given $m_H(T)$ appears many times in this formula, once for each $\nu$ such that $H = G(\nu)$, i.e., for each $\nu$ that realizes $H$. We can give a proper expansion of $\kappa_G(T)$ in graph moments by combining these terms,}
&= n^{\underline{b}} \, \sum_{H} \left(\sum_{\nu \text{ realizes } H} \Wg_{\mu, \nu}\right) \, m_H(T) \, .
\end{align*}

This expansion does not depend on the choice of $\mu$ realizing $G$.
Indeed, this independence even holds over the inner summations giving the coefficient of each $m_H(T)$: choosing a different $\mu$ does not affect the terms in this summation, since $\Wg_{\mu, \nu} = \Wg_{\pi \mu, \pi \nu}$ for a permutation $\pi$ of the indices being matched.
By Proposition~\ref{prop:number-realizing-matchings}, we may then rewrite
\begin{align*}
\kappa_G(T)
&= \frac{n^{\underline{b}}}{\#\{\mu: \mu \text{ realizes } G\}}\sum_H \Wg_{G, H} m_H(T) \\
&= \frac{n^{\underline{b}}}{p!^d d!} \cdot |\eAut(G)| \cdot \sum_H \Wg_{G, H} m_H(T),
\end{align*}
as claimed.
\end{proof}

Lastly, let us derive the simple relationship between the $\kappa^c_G$ and the $\kappa_G$.
\begin{proposition}
    \label{prop:kappa-centered-expansion}
    For any $p$-regular multigraph $G$ on $d$ vertices and $b$ edges,
    \begin{equation}
        \kappa^c_G(T) = \sum_{S \subseteq \Frob(G)} (-1)^{|S|} (n - b + p|S|)^{\underline{p|S|}} \, \kappa_{G \setminus S}(T).
    \end{equation}
\end{proposition}
\begin{proof}
    Expanding the product in the definition of $m_G^c(T)$, we have
    \begin{equation}
        m_G^c(T) = \sum_{S \subseteq \Frob(G)} (-1)^{|S|} (n - b + p|S|)^{\underline{p|S|}} \, m_{G \setminus S}(T),
    \end{equation}
    since there are $p|S|$ indices in the summation forced to be distinct from the others but otherwise playing no rule in the summands.
    Projecting each side to invariant polynomials (by averaging over $Q$ the value applied to $Q \cdot T$) gives the result.
\end{proof}

\subsection{Inner Products and Approximate Orthogonality}

We now turn to computing the inner products $\EE_{W \sim \Wig} \kappa_G^c(W)\kappa_H^c(W)$, which will show that the $\kappa_G^c$ form an approximately orthogonal basis for the invariant polynomials.
We first establish a simple preliminary.
\begin{proposition}
    \label{prop:mGc-mG-inner}
    Suppose $|V(G)| \geq |V(H)|$.
    Then, $\EE_{W \sim \Wig} m_G^c(W) m_H(W) = 0$ unless $G$ and $H$ are isomorphic, and
    \begin{equation}
        \Ex_{W \sim \Wig} m_G^c(W) m_G(W)= \Ex_{W \sim \Wig} m_G^c(W)^2 = n^{\underline{b}} \cdot |\eAut(G)|.
    \end{equation}
\end{proposition}
\noindent
The proof is a simple extension of the argument for Propositions~\ref{prop:centered-covariance} and \ref{prop:distinct-variance} on the values of $\EE_{W \sim \Wig} m_G^c(W) m_H^c(W)$, as changing $m_H^c$ to $m_H$ does not change the value of these inner products.

The following is the main calculation of the inner products we are interested in.
In fact, these inner products are very close to the values of the graph Weingarten function.
\begin{lemma}
    \label{lem:kappaGc-inner}
    For $G, H \in \sG_{d, p}$, if $|V(G)| \neq |V(H)|$, then $\EE_{W \sim \Wig} \kappa_G^c(W)\kappa_H^c(W) = 0$.
    If $|V(G)| = |V(H)| = d$, then
    \begin{equation}
        \Ex_{W \sim \Wig} \kappa_G^c(W)\kappa_H^c(W) = \frac{(n^{\underline{b}})^2}{p!^d d!} \cdot |\eAut(G)| \cdot |\eAut(H)| \cdot \Wg_{G, H}.
    \end{equation}
\end{lemma}
\begin{proof}
    Suppose without loss of generality that $|V(G)| \geq |V(H)|$.
    We have
    \begin{align*}
        &\hspace{-0.5cm}\Ex_{W \sim \Wig} \kappa_G^c(W)\kappa_H^c(W) \\
        &= \Ex_{\substack{W \sim \Wig \\ U, V \sim \Haar}} m_G^c(U \cdot W) m_H^c (V \cdot W) \\
        &= \Ex_{\substack{W \sim \Wig \\ V \sim \Haar}} m_G^c(W) m_H^c (V \cdot W) \\
        &= \Ex_{W \sim \Wig} m_G^c(W)\kappa_H^c(W)
        \intertext{and using Proposition~\ref{prop:kappa-centered-expansion} followed by Lemma~\ref{lem:kappa-expansion}, we find}
        &= \sum_{S \subseteq \Frob(H)} (-1)^{|S|} (n - b + p|S|)^{\underline{p|S|}}\Ex_{W \sim \Wig} m_G^c(W)\kappa_{H \setminus S}^c(W) \\
        &= \sum_{S \subseteq \Frob(G)} (-1)^{|S|} (n - b + p|S|)^{\underline{p|S|}} \frac{n^{\underline{b}}}{p!^d d!} |\eAut(H \setminus S)| \sum_{K} \Wg_{H \setminus S, K} \Ex_{W \sim \Wig} m_G^c(W)m_K(W)
        \intertext{By Proposition~\ref{prop:mGc-mG-inner}, we immediately find that all terms are zero unless $|V(G)| = |V(H)|$. If this holds, then all terms with $S \neq \emptyset$ are zero, and among those the only non-zero term has $K = G$, and we are left with simply}
        &= \frac{n^{\underline{b}}}{p!^d d!} |\eAut(H)| \sum_{K} \Wg_{G, H} \Ex_{W \sim \Wig} m_G^c(W)m_G(W) \\
        &= \frac{(n^{\underline{b}})^2}{p!^d d!} \cdot |\eAut(G)| \cdot |\eAut(H)| \cdot \Wg_{G, H},
    \end{align*}
    completing the proof.
\end{proof}

Finally, we may construct a nearly \emph{orthonormal} basis by normalizing these polynomials as follows.
\begin{definition}
    \label{def:hat-kappaGc}
    For each $G$ a $p$-regular multigraph, define
    \begin{equation}
        \what{\kappa_G^c}(T) \colonequals \frac{n^{b/2}}{n^{\underline{b}} \sqrt{|\eAut(G)|}} \kappa_G^c(T).
    \end{equation}
\end{definition}
\noindent
For the sake of simplicity, we do not try to exactly normalize these polynomials to have norm 1.
On the other hand, they obey the following conditioning property.
\begin{lemma}
    \label{lem:kappaGc-gram}
    If $|V(G)| \neq |V(H)|$, then $\EE_{W \sim \Wig} \what{\kappa_G^c}(W) \what{\kappa_H^c}(W) = 0$.
    Let $M \in \RR^{\sG_{d, p} \times \sG_{d, p}}$ have entries $M_{G, H} \colonequals \EE_{W \sim \Wig} \what{\kappa_G^c}(W) \what{\kappa_H^c}(W)$.
    Suppose that $d \leq \sqrt{n / 2p^2}$.
    Then,
    \begin{equation}
        \frac{1}{2} \leq \lambda_{\min}(M) \leq \lambda_{\max}(M) \leq 2.
    \end{equation}
\end{lemma}
\begin{proof}
    By Lemma~\ref{lem:kappaGc-inner},
    \begin{equation}
        M_{G, H} = n^b \cdot \frac{\sqrt{|\eAut(G)| \cdot |\eAut(H)|}}{p!^d d!} \Wg_{G, H}.
    \end{equation}
    The result then follows by the result of Corollary~\ref{cor:wg-graph} on the conditioning of the graph Weingarten function.
\end{proof}
\noindent
Roughly, the idea of the normalization in Definition~\ref{def:hat-kappaGc} is that $\Wg \asymp n^{-\ell / 2} \id$, so if we want the diagonal of $M$ in the proof above to be roughly constant, we must make the diagonal entries normalized by $\#\{\mu \text{ realizes } G\}$ for each $G$, which is given in terms of $|\eAut(G)|$ in Proposition~\ref{prop:number-realizing-matchings}.
See also Remark~\ref{rem:sniady} for an argument showing that the upper bound on $d$ is necessary for such a condition to hold.

\subsection{Advantage Bound for Invariant Detection Problems}

Using the tools developed above, we derive the following general bound on the advantage between the Wigner tensor law and any invariant distribution, which for the purposes of applications to hypothesis testing is the final prize of our work with finite free cumulants.
\begin{corollary}
    \label{cor:invariant-advantage-bound}
    Let $\PP$ be any probability measure on symmetric tensors that is orthogonally invariant.
    Suppose that $D \leq \sqrt{n / 2p^2}$.
    Then,
    \begin{equation}
        \Adv_{\leq D}(\Wig, \PP)^2 \asymp \sum_{d = 0}^D \frac{1}{n^{pd/2}} \sum_{G \in \sG_{d, p}} \frac{1}{|\eAut(G)|} \left(\Ex_{Y \sim \PP} \kappa_G^c(Y)\right)^2,
    \end{equation}
    where the sum is over $p$-regular multigraphs up to isomorphism.
\end{corollary}
\noindent
Concretely, the constant in an upper bound may be taken to be $2$, and the constant in a lower bound to be $\frac{1}{2}\exp(-\frac{1}{8}) \geq \frac{1}{3}$.

\begin{proof}
    Recall that, by Proposition~\ref{prop:adv-invariant}, under these assumptions we have the formula for the advantage
    \begin{equation}
    \Adv_{\leq D}(\Wig, \PP) = \sup_{\substack{f \in \RR[Y]_{\leq D} \\ f \text{ invariant} \\ \EE_{W \sim \Wig} f(W)^2 \neq 0}} \frac{\EE_{Y \sim \PP} f(Y)}{\sqrt{\EE_{W \sim \Wig} f(W)^2}}.
    \end{equation}
    By Lemma~\ref{lem:kappaGc-gram}, the $\what{\kappa_G^c}$ over $G \in \sG_{0, p} \sqcup \cdots \sqcup \sG_{d, p}$ are a basis for the invariant polynomials in $\RR[Y]_{\leq D}$.
    Thus, consider $f$ in the optimization stated, and write the expansion
    \begin{equation}
        f(T) = \sum_{G \in \sG_{0, p} \sqcup \cdots \sqcup \sG_{d, p}} \alpha_G \cdot \what{\kappa_G^c}(T).
    \end{equation}
    Again by Lemma~\ref{lem:kappaGc-gram}, we have
    \begin{equation}
        \frac{1}{2}\|\alpha\|^2 = \frac{1}{2} \sum_{G \in \sG_{0, p} \sqcup \cdots \sqcup \sG_{d, p}} \alpha_G^2 \leq \EE_{W \sim \Wig} f(W)^2 \leq 2 \sum_{G \in \sG_{0, p} \sqcup \cdots \sqcup \sG_{d, p}} \alpha_G^2 = 2\|\alpha\|^2.
    \end{equation}
    Define
    \begin{equation}
        \beta_G \colonequals \Ex_{Y \sim \PP} \what{\kappa_G^c}(Y).
    \end{equation}
    Then, we have by linearity
    \begin{equation}
        \Ex_{Y \sim \PP} f(Y) = \langle \alpha, \beta \rangle,
    \end{equation}
    and thus by basic linear algebra, with lower and upper bound constants $\frac{1}{2}$ and 2 respectively,
    \begin{align*}
        \Adv_{\leq D}(\Wig, \PP)^2
        &\asymp \left(\sup_{\|\alpha\| \neq 0} \frac{\langle \alpha, \beta \rangle}{\|\alpha\|}\right)^2 \\
        &= \|\beta\|^2 \\
        &= \sum_G \left(\Ex_{Y \sim \PP} \what{\kappa_G^c}(Y)\right)^2
        \intertext{and substituting in the normalization of $\what{\kappa_G^c}$,}
        &= \sum_G \frac{n^{b_G}}{(n^{\underline{b_G}})^2}\frac{1}{|\eAut(G)|} \left(\Ex_{Y \sim \PP} \kappa_G^c(Y)\right)^2 \\
        &= \sum_{d = 0}^D \frac{n^{pd/2}}{(n^{\underline{pd/2}})^2} \sum_{G \in \sG_{d, p}}\frac{1}{|\eAut(G)|}\left(\Ex_{Y \sim \PP} \kappa_G^c(Y)\right)^2
        \intertext{and finally by Proposition~\ref{prop:falling-factorial} on the falling factorial,}
        &\asymp \sum_{d = 0}^D \frac{1}{n^{pd/2}} \sum_{G \in \sG_{d, p}}\frac{1}{|\eAut(G)|}\left(\Ex_{Y \sim \PP} \kappa_G^c(Y)\right)^2,
    \end{align*}
    as stated, with lower and upper bound constants $\exp(-\frac{1}{8})$ and 1, respectively, in the last step.
\end{proof}

At a high level, this bound is making the following reasonable claim: the $\EE \kappa_G^c(Y)$ are quantities that are zero when $Y \sim \Wig$, and their magnitude for a given distribution---a kind of ``cumulant distance'' between a distribution and the Wigner law---controls the difficulty of hypothesis testing against a Wigner null hypothesis.
Let us give some more intuition about what ``critical scaling'' of the cumulants---around the threshold of computational hardness---this implies.

As we will see, in the models we will consider all cumulants $\EE\kappa_G^c(T)$ with $|V(G)| = d$ will have roughly the same size.
The number of $p$-regular multigraphs on $d$ vertices, as we will see later in Proposition~\ref{prop:count-multigraphs}, grows roughly as $|G_{d, p}| \approx d^{O(d)}$ for fixed $d$.
Thus, in order for the advantage to be bounded, it will suffice to have $|\EE \kappa_G^c(T)| \lesssim n^{pd / 4} / \mathrm{poly}(d)$, as then the sum will be a truncation of a convergent geometric series.
This is sensible: recall that the cumulant is a sum of $n^{pd/2}$ centered terms (as this is the number of labellings of the edges of $G$ by $[n]$), so $n^{pd/4}$ is the ``usual'' value of the cumulants of a law $\PP$ whose entries are on the same scale as $\Wig$.

\subsection{Bases for Vector-Valued Equivariant Polynomials}

Finally, we develop some parallel theory for the case of equivariant polynomials and open multigraphs.
For the sake of simplicity, we only work with 1-open multigraphs, i.e., ones with a single open edge.
Similar results should hold for general $t$-open multigraphs, albeit with more elaborate combinatorics.
We will be brisk in our presentation, as the ideas are similar to the case of closed multigraphs.

We slightly abuse notation and reuse the notations $m_G, m_G^{!}, m_G^c, \kappa_G$, $\kappa_G^c$, and $\sG_{d, p}$.
However, to remind the reader that we are working with vector-valued functions and open multigraphs, we replace ``$G$'' with ``$G\!\to$'', a visual indication of the one ``loose'' or ``dangling'' edge on a 1-open multigraph.
We also use:
\begin{definition}[1-open $p$-regular multigraphs]
    Write $\sG_{d, p \to}$ for the set of (non-isomorphic) 1-open $p$-regular multigraphs on $d$ unlabelled vertices.
    Equivalently, these may be viewed as multigraphs on $d + 1$ vertices, all of which have degree $p$ except for one, which has degree $p - 1$.
\end{definition}

\begin{remark}[Arity parity]
    In order for $\sG_{d, p \to}$ to not be empty, $pd$ must be odd, and in particular $p$, the tensor arity, must be odd.
\end{remark}

For $G \in \sG_{d, p\to}$, we write $m_{G\to}: \Sym^p(\RR^n) \to \RR^n$ for the ordinary 1-open graph moments previously defined in Definition~\ref{def:open-moments}.
We also write $m_{G\to}^!(T) \in \RR^n$ for the same quantities with all indices in the summation involved restricted to be distinct: letting $e_0$ be the open edge of $G$,
\begin{equation}
    m_{G\to}^!(T)_i \colonequals \sum_{\substack{j \in [n]^{E} \\ j(e_0) = i \\ j_1, \dots, j_b \text{ distinct}}} \prod_{v \in V} T_{j(\partial v)} \,.
\end{equation}
Note that this restriction means that the summation defining $m_{G\to}^!(T)_i$ depends on $i$, since all other indices in this summation must be different than $i$.
We similarly define the centered version of these summations, where our handling of Frobenius pairs only affects connected components other than the one containing $e_0$:
\begin{equation}
    m_{G\to}^c(T)_i \colonequals \sum_{\substack{j \in [n]^{E} \\ j(e_0) = i \\ j_1, \dots, j_b \text{ distinct}}} \prod_{v \in V(G \setminus \Frob(G))} T_{j(\partial v)} \prod_{F \in \Frob(G)} (T_{j(F)}^2 - 1).
\end{equation}

We likewise define the symmetrized versions $\kappa_{G\to}$ and $\kappa_{G\to}^c$.
Note here that, as we used previously in Proposition~\ref{prop:corr-equivariant} on the low-degree correlation, the ``right'' mapping from a general function to an equivariant one is $f(T) \mapsto \EE_{Q \sim \Haar} \, Q^{\top} f(Q \cdot T)$.
Indeed, 1-open graph moments, being equivariant already, are easily checked to be unchanged by this operation.
For the other two quantities, we define analogs of the closed cumulants, which we call \emph{1-open finite free cumulants}:
\begin{align*}
    \kappa_{G\to}(T) &\colonequals \Ex_{Q \sim \Haar} Q^{\top} m_{G \to}^!(Q \cdot T), \\
    \kappa_{G\to}^c(T) &\colonequals \Ex_{Q \sim \Haar} Q^{\top} m_{G \to}^c(Q \cdot T).
\end{align*}

We now derive analogs of our results for closed graph moments and the associated finite free cumulants for these quantities.
First, we give the analogous additivity property.
A general statement like Proposition~\ref{prop:general-additivity} holds, but we will not need it, so for the sake of simplicity we restrict our attention to the following special case.
\begin{proposition}
    \label{prop:additivity-cumulants}
    For any $G \in \sG_{d, p\to}$ and $A, B \in \Sym^p(\RR^n)$,
    \begin{align}
        \Ex_{Q \sim \Haar} \kappa_{G\to}(A + Q \cdot B) = \sum_{\substack{G = G_A \sqcup G_B \\ G_A \text{ open} \\ G_B \text{ closed}}} \frac{n^{\underline{b - 1}}}{n^{\underline{b_A - 1}}n^{\underline{b_B}}}  \kappa_{G_A\to}(A) \kappa_{G_B}(B), \\
        \Ex_{Q \sim \Haar} \kappa_{G\to}^c(A + Q \cdot B) = \sum_{\substack{G = G_A \sqcup G_B \\ G_A \text{ open} \\ G_B \text{ closed}}} \frac{n^{\underline{b - 1}}}{n^{\underline{b_A - 1}}n^{\underline{b_B}}}  \kappa_{G_A\to}^c(A) \kappa_{G_B}(B).
    \end{align}
\end{proposition}
\noindent
Again, the proof is a simple variation on that of Proposition~\ref{prop:additivity}.

Next, we describe the expansion of each of the $\kappa_{G\to}$ and $\kappa_{G\to}^c$ in open graph moments.
\begin{definition}
    For $G \in \sG_{d,p\to}$, write $\chop(G)$ for the (closed) graph formed by deleting the open edge of $G$.
    This is a graph where every vertex has degree $p$, except for one vertex (the one that used to be the only endpoint of the open edge in $G$) that has degree $p - 1$.
\end{definition}

We will need some generalizations of the ideas discussed previously involving the Weingarten function.
First, note that the notion of a matching \emph{realizing} a graph is sensible for a graph with any fixed degree sequence: for a graph $G$ with degree sequence $p_1, \dots, p_d$, we may similarly say that $\mu \in \sM([p_1 + \cdots + p_d])$ realizes $G$ if $G$ is isomorphic to the graph formed by identifying $\{1, \dots, p_1\}$, $\{p_1 + 1, \dots, p_1 + p_2\}$, and so forth in $\mu$.
Accordingly, the graph Weingarten function of Definition~\ref{def:graph-weingarten} may be generalized to graphs with arbitrary specified degree sequence, and in particular $\Wg_{\chop(G), \chop(H)}$ is defined for $G, H \in \sG_{d, p\to}$.
Lastly, the notion of edge automorphisms from Definition~\ref{def:eAut} is still valid for graphs that are not $p$-regular, and we may again make sense of $\eAut(\chop(G))$ for $G \in \sG_{d, p\to}$.
With these tools in hand, we may formulate:
\begin{proposition}
    \label{prop:number-realizing-matchings-chop}
    For $G \in \sG_{d, p\to}$, the number of matchings realizing $\chop(G)$ is $p!^{d - 1}(p - 1)!(d - 1)! / |\eAut(\chop(G))|$.
\end{proposition}
\begin{proof}
    The proof is the same as for Proposition~\ref{prop:number-realizing-matchings}, except that permutations of vertices must preserve the unique vertex of degree $p - 1$, and permutations of the half-edges around this vertex contribute $(p - 1)!$ rather than $p!$.
\end{proof}

\begin{lemma}
    \label{lem:kappa-1open-expansion}
    For any $G \in \sG_{d,p\to}$ with $d$ vertices and $b$ edges (including the open edge),
    \begin{equation}
        \kappa_{G\to}(T) = \frac{n^{\underline{b - 1}}}{p!^{d - 1}(p - 1)! (d - 1)!} \cdot |\eAut(\chop(G))| \cdot \sum_H \Wg_{\chop(G), \chop(H)} m_{H\to}(T).
    \end{equation}
\end{lemma}
\begin{proof}
    As in Lemma~\ref{lem:kappa-expansion}, we think in terms of matchings.
    We will use the notations $w(\mu)$ and $w^!(\mu)$ used there for tensors associated to matchings $\mu$.

    If $\mu$ is a matching of $[pd - 1] = [2(b - 1)]$ realizing $\chop(G)$, then $m_{G\to}(T)_i = \langle T^{\otimes d}, w(\mu) \otimes e_i \rangle$, and $m_{G\to}^!(T)_i = \langle T^{\otimes d}, w^!(\mu) \otimes e_i \rangle$.
    Following the argument from Lemma~\ref{lem:kappa-expansion}, we find
    \begin{align*}
        \kappa_{G\to}(T)_i
        &= \sum_{\substack{j \text{ distinctly} \\ \text{compatible with } \mu}} \Ex_{Q \sim \Haar} \langle T^{\otimes d}, (Q^{\otimes pd - 1} e_j) \otimes e_i \rangle \\
        &= n^{\underline{b - 1}} \langle T^{\otimes d}, (\Pi_{pd - 1}e_{\mu}) \otimes e_i \rangle
        \intertext{and by the Weingarten formula,}
        &= n^{\underline{b - 1}} \sum_{\nu} \Wg_{\mu, \nu} \langle T^{\otimes d}, w(\nu) \otimes e_i \rangle \\
        &= n^{\underline{b - 1}} \sum_{\nu} \Wg_{\mu, \nu} m_{G(\nu)\to}(T)_i,
    \end{align*}
    and the rest of the calculations proceed identically to Lemma~\ref{lem:kappa-expansion}.
\end{proof}

The following is the analog of Proposition~\ref{prop:kappa-centered-expansion} for 1-open multigraphs, and follows from exactly the same simple expansion.
\begin{proposition}
    \label{prop:kappa-centered-1open-expansion}
    For any $p$-regular open multigraph $G$ on $d$ vertices and $b$ edges,
    \begin{equation}
        \kappa_{G\to}^c(T) = \sum_{S \subseteq \Frob(G)} (-1)^{|S|} (n - b + p|S|)^{\underline{p|S|}} \kappa_{G \setminus S \to}(T).
    \end{equation}
\end{proposition}

\begin{proposition}
    \label{prop:mGc-mG-1open-inner}
    Suppose that $G$ and $H$ are 1-open $p$-regular multigraphs with $|V(G)| \geq |V(H)|$.
    Then, $\EE_{W \sim \Wig} \langle m_{G\to}^c(W), m_{H\to}(W) \rangle = 0$ unless $G$ and $H$ are isomorphic, and
    \begin{equation}
        \Ex_{W \sim \Wig} \langle m_{G\to}^c(W), m_{G\to}(W) \rangle = \Ex_{W \sim \Wig} \langle m_{G\to}^c(W), m_{G\to}^c(W) \rangle = n^{\underline{b}} \cdot |\eAut(\chop(G))|.
    \end{equation}
\end{proposition}

\begin{lemma}
    \label{lem:kappaGc-1open-inner}
    Suppose that $G$ and $H$ are 1-open $p$-regular multigraphs. If $|V(G)| \neq |V(H)|$, then $\EE_{W \sim \Wig} \langle \kappa_{G\to}^c(W), \kappa_{H\to}^c(W)\rangle = 0$.
    If $|V(G)| = |V(H)| = d$, then
    \begin{align*}
        &\Ex_{W \sim \Wig} \langle \kappa_{G\to}^c(W), \kappa_{H\to}^c(W)\rangle \\
        &\hspace{1cm}= \frac{n^{\underline{b - 1}}n^{\underline{b}}}{p!^{d - 1}(p - 1)!(d - 1)!} \cdot |\eAut(\chop(G))| \cdot |\eAut(\chop(H))| \cdot \Wg_{\chop(G), \chop(H)}.
    \end{align*}
\end{lemma}
\noindent
If the overall scaling factor of $n^{2b - 1}$ appears unusual, the reader may consider, at an intuitive level, that while the inner product of two closed graph moments each on $b$ edges has $2b$ edges, the inner product of two 1-open graph moments each on $b$ edges only has $2b - 1$ edges, since the two open edges ``fuse'' to become one.

\begin{proof}
    We again proceed similarly to the closed case in Lemma~\ref{lem:kappaGc-inner}.
    Suppose without loss of generality that $|V(G)| \geq |V(H)|$.
    We have:
    \begin{align*}
        &\hspace{-0.5cm}\Ex_{W \sim \Wig} \langle \kappa_G^c(W),  \kappa_H^c(W) \rangle \\
        &= \Ex_{\substack{W \sim \Wig \\ Q, R \sim \Haar}} \langle Q^{\top} m_G^c(Q \cdot W), R^{\top} m_H^c(R \cdot W) \rangle \\
        &= \Ex_{\substack{W \sim \Wig \\ Q, R \sim \Haar}} \langle m_G^c(Q \cdot Q^{\top} \cdot W), QR^{\top} m_H^c(R \cdot Q^{\top} \cdot W) \rangle \\
        &= \Ex_{\substack{W \sim \Wig \\ Q \sim \Haar}} \langle m_G^c(W), Q^{\top} m_H^c(Q \cdot W) \rangle \\
        &= \Ex_{W \sim \Wig} \langle m_G^c(W),  \kappa_H^c(W) \rangle
        \intertext{and, using Lemma~\ref{lem:kappa-1open-expansion} and Proposition~\ref{prop:kappa-centered-1open-expansion},}
        &= \sum_{S \subseteq \Frob(H)} (-1)^{|S|} (n - b_H + p|S|)^{\underline{p|S|}} \frac{n^{\underline{b_H - 1}}}{p!^{d - 1}(p - 1)(d - 1)!} \\
        &\hspace{1cm} |\eAut(\chop(H))| \cdot \sum_K \Wg_{\chop(H \setminus S), \chop(K)}  \Ex_{W \sim \Wig} \langle m_{G\to}^c(W),  m_{K\to}(W) \rangle
        \intertext{Here, by Proposition~\ref{prop:mGc-mG-1open-inner} all terms are zero unless $|V(G)| = |V(H)|$.
    In that case, only the term $K = G$ and $S = \emptyset$ contributes, and we are left with}
        &= \frac{n^{\underline{b - 1}}n^{\underline{b}}}{p!^{d - 1}(p - 1)!(d - 1)!} \cdot |\eAut(\chop(G))| \cdot |\eAut(\chop(H))| \cdot \Wg_{\chop(G), \chop(H)},
    \end{align*}
    completing the calculation.
\end{proof}

Lastly, as in the closed case, we construct a nearly orthonormal basis with respect to this inner product.
\begin{definition}
    For each $G$ a $p$-regular 1-open multigraph having $b$ edges, define
    \begin{equation}
        \what{\kappa_{G\to}^c}(T) \colonequals \frac{n^{b/2}}{n^{\underline{b}} \cdot \sqrt{|\eAut(\chop(G))|}} \kappa_{G\to}^c(T).
    \end{equation}
\end{definition}

\begin{lemma}
    \label{lem:kappaGc-1open-gram}
    If $|V(G)| \neq |V(H)|$, then $\EE_{W \sim \Wig} \langle \what{\kappa_{G\to}^c}(W), \what{\kappa_{H\to}^c}(W)\rangle = 0$.
    Let $M \in \RR^{\sG_{d, p\to} \times \sG_{d, p\to}}$ have entries
    \begin{equation}
        M_{G, H} \colonequals \EE_{W \sim \Wig} \langle \what{\kappa_{G\to}^c}(W), \what{\kappa_{H\to}^c}(W)\rangle.
    \end{equation}
    Suppose that $d \leq \sqrt{n / 2p^2}$.
    Then,
    \begin{equation}
        \frac{1}{2} \leq \lambda_{\min}(M) \leq \lambda_{\max}(M) \leq 2.
    \end{equation}
    In particular, the $\kappa_{G\to}^c(T)$ are linearly independent as vectors of polynomials.
\end{lemma}
\begin{proof}
    By Lemma~\ref{lem:kappaGc-1open-inner}, we have
    \begin{equation}
        M_{G, H} = n^{b - 1} \cdot \frac{\sqrt{|\eAut(\chop(G))| \cdot |\eAut(\chop(H))|}}{p!^{d - 1}(p - 1)!(d - 1)!} \Wg_{\chop(G), \chop(H)}.
    \end{equation}
    The result then follows by Corollary~\ref{cor:wg-graph-chop}, where we observe that $b = (pd + 1) / 2$ for $G \in \sG_{d, p\to}$.
\end{proof}

\section{Application 1: Tensor PCA}

\subsection{Warmup: Detection with Individual Graph Moments}
\label{sec:glimpse}

Before proceeding to the analysis of general low-degree polynomials, let us consider the simpler question of whether individual graph moments $m_G(T)$ can distinguish between the two distributions associated with tensor PCA.
In order to do this, we will need to understand how these moments behave under both the null and planted distributions.

The question for the null model is perhaps of independent interest, and we will see that it holds some surprises already.
Recall that, for $k = 2$ and $G = C_\ell$ a cycle of length $\ell$, classical random matrix theory tells us that
\begin{equation}
    \Ex_{T \sim \Wig(2, n, 1)}[m_G(T)] = \Ex_{T \sim \mathrm{GOE}(n)}[\tr(T^\ell)] = \left(1 + O\left(\frac{1}{n}\right)\right) \One\{\ell \text{ even}\} \, \Cat(\ell/2) \, n^{\ell/2 + 1} \, ,
\end{equation}
where $\Cat(t) = \frac{1}{t+1} \binom{2t}{t}$ is the $t$'th Catalan number.
What is the tensor analog of this calculation?

\begin{definition}
\label{def:even-colorings}
    Let $G = (V, E)$ be a $p$-regular graph.
    For $C$ a finite set of colors, let $\sigma: E \to C$ be an edge coloring.
    We call $\sigma$ an \emph{even} coloring if, among the $\sigma(\partial v)$ over $v \in V$, viewed as multisets of colors, each multiset occurs an even number of times.
    We write $c_{\max}(G)$ for the largest possible number of colors in an even edge coloring of $G$.
    For $\sigma$ an even edge coloring, suppose that the distinct colorings of neighborhoods (i.e., the multisets $\sigma(\partial v)$ over vertices $v$) occurring in $\sigma$ are $N_1, \dots, N_m$, where $N_i$ occurs $f_i \in 2\NN$ times.
    Also, write $c_j(N_i)$ for the number of times that $j \in C$ occurs in $N_i$.
    We call the \emph{weight} of $\sigma$ the quantity
    \begin{equation}
        w(\sigma) \colonequals \prod_{i = 1}^m (f_i - 1)!! \prod_{j \in C} c_j(N_i)!^{f_i / 2}.
    \end{equation}
\end{definition}

\begin{proposition}
\label{prop:even-colorings}
    Define
    \begin{equation}
        w_{\max}(G) \colonequals \sum_{\substack{\sigma \text{ even edge coloring of } G \\ |\sigma(E)| = c_{\max}(G)}} w(\sigma),
    \end{equation}
    where the sum is over non-isomorphic edge colorings with the given number of colors.
    Then,
    \begin{equation}
        \Ex_{T \sim \Wig(p, n, 1)}[m_G(T)] = \left(1 + O_G\left(\frac{1}{n}\right)\right) \,  w_{\max}(G) \, n^{c_{\max}(G)}.
    \end{equation}
\end{proposition}
\noindent
Actually, as we show in Propositions~\ref{prop:max-coloring-unique-within-nbd} and \ref{prop:max-coloring-unique-across-nbd}, any maximum even edge coloring $\sigma$ has $w(\sigma) = 1$ (that is, no vertex is adjacent to several edges of the same color, and every colored neighborhood occurs exactly twice).
Therefore, $w_{\max}(G)$ is purely a counting problem of a particular type of coloring.

\begin{proof}
    We expand the definition of the tensor moment and use linearity:
    \begin{align}
        \Ex_{T \sim \Wig(p, n, 1)}m_G(T)
        &= \sum_{i \in [n]^E} \EE \prod_{v \in V} T_{i(\partial v)} \\
        &= \sum_{\sigma \text{ even edge coloring of } G} n^{\underline{|\sigma(E)|}} \, w(\sigma) \\
        &= \left(1 + O_G\left(\frac{1}{n}\right)\right)\sum_{\sigma \text{ even edge coloring of } G} n^{|\sigma(E)|} w(\sigma),
    \end{align}
    and extracting the leading order contribution completes the proof.
    Here we sum over all non-isomorphic even edge colorings $\sigma$; labeling the colors of such colorings by distinct values in $[n]$ enumerates all $i \in [n]^E$ that contribute to the sum.
\end{proof}

\begin{figure}
\centering
\includegraphics[width=\columnwidth]{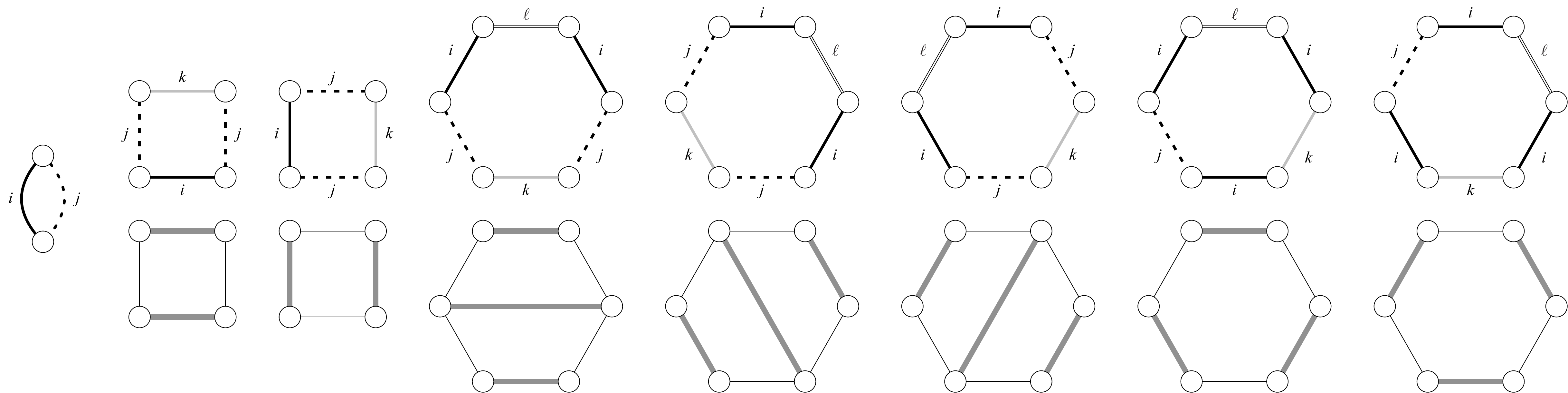}
\caption{Examples of legal edge colorings for the matrix case $p=2$. In the cycle of length $2$ there is one legal coloring with two distinct colors (up to permutations of the colors) recovering the fact that $\tr A^2 = \sum_{ij} A_{ij} A_{ji} = n^2$. There are two legal colorings of the cycle of length $4$: the one on the left, for instance, corresponds to the term $\sum_{ijk} A_{ij} A_{jk} A_{kj} A_{ji}$. There are five legal colorings of the cycle of length $6$. Each coloring corresponds to the matching of vertices shown below it, where we show the pairs of vertices with the same neighborhoods: these matchings are the non-crossing partitions familiar in combinatorial free probability \cite{NS-2006-LecturesCombinatoricsFreeProbability}. This recovers the leading behavior $\tr A^{\ell} = \Cat(\ell/2) n^{\ell/2+1}$ where $\Cat(t)=1, 2, 5, \ldots$ is the $t$th Catalan number.}
\label{fig:matrix-examples}
\end{figure}

As an aside, we mention the following ancillary result.
Despite this characterization, unlike in the matrix case, and perhaps surprisingly, it is difficult in general to compute even the scaling of a given moment of a Wigner tensor for large $n$.

\begin{theorem}
    \label{thm:wigner-moments-hard}
    It is NP-hard to decide whether $c_{\max}(G) \geq c$ given $c \geq 0$ and a regular graph $G$.
\end{theorem}
\noindent
We give a proof in Appendix~\ref{app:wigner}.

In particular, we may understand the scaling of \emph{variances} of $m_G(T)$ as follows.
\begin{corollary}
    \label{cor:mG-var-lb}
    For any $G$, for $W \sim \Wig$, $\Var m_G(W) \geq n^{\underline{|E(G)|}}$.
\end{corollary}
\begin{proof}
    First, note that, in general, $m_G(T) m_H(T) = m_{G \sqcup H}(T)$, where $\sqcup$ denotes disjoint union of graphs.
    Thus
    \begin{equation}
        \Ex_{W \sim \Wig} m_G(W)^2 = \Ex_{W \sim \Wig} m_{G \sqcup G}(W) = \sum_{\sigma \text{ even coloring of } G \sqcup G} n^{\underline{|\sigma(E)|}} \, w(\sigma).
        \label{eq:mGG}
    \end{equation}
    On the other hand, we have
    \begin{equation}
        \left(\Ex_{W \sim \Wig} m_G(W)\right)^2 = \left(\sum_{\sigma \text{ even coloring of } G} n^{\underline{|\sigma(E)|}} \, w(\sigma)\right)^2
    \end{equation}
    which is at most the sum of the terms in \eqref{eq:mGG} where each colored vertex neighborhood occurs at least twice within one of the two copies of $G$.
    There is another coloring in the summation in \eqref{eq:mGG}, which is not of this kind: it is the coloring that colors every edge in one copy of $G$ with a different color, and then repeats the same coloring on the other copy of $G$.
    This is an even coloring which uses $|E(G)|$ colors, and which is not cancelled in computing $\Var m_G(W) = \Ex_{W \sim \Wig} m_G(W)^2 - (\Ex_{W \sim \Wig} m_G(W))^2$.
    See Figure~\ref{fig:single-variance} for an illustration.
\end{proof}

For the planted model, we are interested in the special case of \emph{rank one} tensors.
On these, evaluating the graph moments is easy:
\begin{proposition}
\label{prop:all-spike}
    Let $v^{\otimes p}$ be the $p$th tensor power of $v$: that is, $v^{\otimes p}_{i_1,\ldots,i_p} = v_{i_1} v_{i_2} \cdots v_{i_p}$. Then if $G=(V,E)$, we have $m_G(v^{\otimes p}) = \|v\|^{2|E|} = \|v\|^{p|V|}$.
\end{proposition}
\begin{proof}
The summation over the index associated to each edge gives a factor of $\langle v, v \rangle = \|v\|^2$ on each edge, for a total of $\|v\|^{2|E|}$.
Since $G$ is $p$-regular, $|E| = p|V|/2$.
\end{proof}
We now sketch an argument of~\cite{Ouerfelli-Tamaazousti-Rivasseau-2022} that no single graph moment can solve the detection problem for tensor PCA below the conjectured threshold.

\begin{proposition}
\label{prop:single-variance}
Let $G$ be a $p$-regular graph. If $\lambda \ll n^{-p/4}$, then
\begin{equation}
    \left| \Ex_{Y \sim \PP} m_G(Y) - \Ex_{Y \sim \QQ} m_G(Y)\right| \ll \sqrt{\Varx_{Y \sim \QQ} m_G(Y)}.
\end{equation}
\end{proposition}

\begin{proof}[Proof Sketch]
Suppose $G$ has $d$ vertices and $b$ edges. First we consider the expected difference in $m_G$ between the two models. If $Y = \lambda v^{\otimes p} + W$, then expanding $m_G(Y)$ creates $2^d$ terms, including cross-terms where the spike $\lambda v^{\otimes p}$ appears at some vertices and the noise $W$ appears at others. For some graphs these cross-terms can be neglected~\cite{Ouerfelli-Tamaazousti-Rivasseau-2022}. In this proof sketch we look only at the contribution of the all-spike term. By Proposition~\ref{prop:all-spike} this is $\lambda^d n^b$, since we get a factor of $\lambda$ at each vertex and a factor of $|v|^2=n$ on each edge.

On the other hand, by Corollary~\ref{cor:mG-var-lb}, we have $\Var m_G(Y) \gtrsim n^{b}$.
Comparing the spike's contribution to the expectation with the standard deviation in the null model, we see that $m_G$ fails as a test statistic whenever
\begin{equation}
\label{eq:ll}
\lambda^a n^b \ll n^{b/2} \, .
\end{equation}
But since $G$ is $p$-regular, we have $b=pa/2$. Then~\eqref{eq:ll} becomes
\[
\lambda^a n^{pa/2} \ll n^{pa/4} \, ,
\]
and solving for $\lambda$ gives $\lambda \ll n^{-p/4}$ as stated.
\end{proof}

\begin{figure}
    \centering
\includegraphics[width=3in]{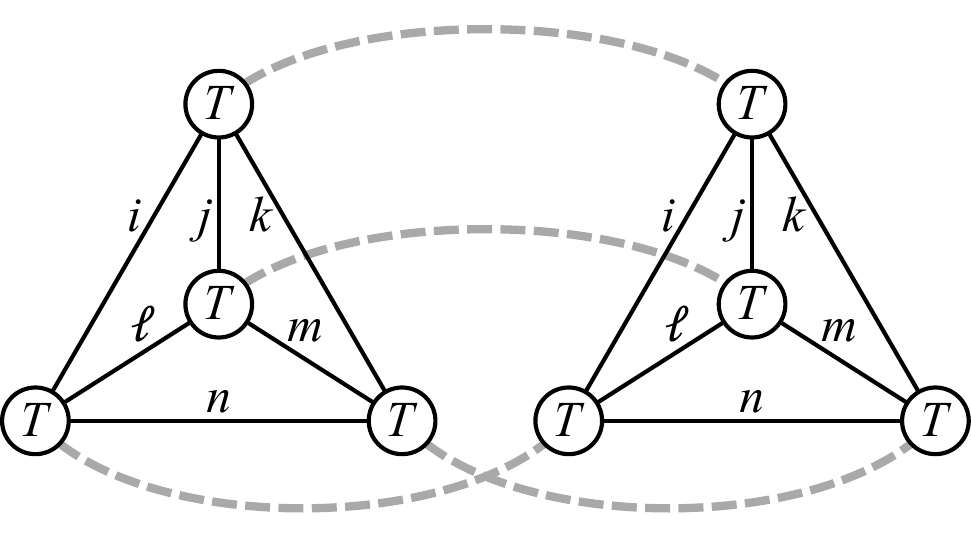}
    \caption{Even colorings of the disjoint union of two copies of $K_4$ where each vertex is matched to its counterpart. The colors or indices $i,j,k,\ell,m,n$ in one copy must match their counterparts, but within each copy these indices can range freely over $[n]$. Thus they contribute $n^6$ to the variance of $T^{K_4}$ when $T$ is a Wigner random tensor.}
\label{fig:single-variance}
\end{figure}

This argument is tight for some graphs~\cite{Ouerfelli-Tamaazousti-Rivasseau-2022}.
However, while graph moments are fascinating, for general $G$ they are not ideal quantities to prove upper and lower bounds on problems such as tensor PCA, for three reasons:

\begin{enumerate}
\item By Theorem~\ref{thm:wigner-moments-hard}, even computing how the expectations of $m_G(W)$ scales with $n$ in the null model $W \sim \Wig$ is computationally hard.
\item As discussed above, the expectation of $m_G(Y)$ in spiked models involves cross-terms where the spike appears at some vertices and the noise appears at others. These mixed moments are often nonzero, making precise calculations combinatorially challenging.
\item Different graph moments $m_G$ and $m_H$ can be correlated in the Wigner model, even when $G$ and $H$ are non-isomorphic. This raises the possibility that, while the variance of any one graph moment is too large to beat the threshold, it might be possible to construct linear combinations of small graphs that cancel out much of each others' variance. Indeed, Proposition~\ref{prop:single-variance} applies to graphs of any size, even where $a$ grows with $n$, but there are subexponential-time algorithms of growing degree that succeed even when $\lambda \ll n^{-p/4}$~\cite{BhattiproluGGLT16,BhattiproluGL17,wein-elalaoui-moore} (see also~\cite{raghavendra-rao-schramm} for analogous algorithms for random XOR-SAT). Low-degree algorithms matching this performance must somehow involve linear combinations of many graph moments.
\end{enumerate}

We therefore proceed to full lower bounds against low-degree polynomials, now leaning on the machinery we have developed based on finite free cumulants, which will allow us to circumvent all of these difficulties.

\subsection{Detection with General Low-Degree Polynomials: Proof of Theorem~\ref{thm:tensor-pca-detection}}

Before proceeding, let us state the formal version of Theorem~\ref{thm:tensor-pca-detection} that we will prove here. Recall we are in the setting where $\QQ = \Wig$ and $\PP$ the spiked tensor model $Y = \lambda v^{\otimes p} + W$.
\begin{theorem}[Low-degree analysis for tensor PCA detection]
    \label{thm:tensor-pca-detection-formal}
    Let $D = D(n) \in \NN$ have $D \leq \sqrt{n / 2p^2}$.
    There are constants $a_p, b_p > 0$ such that:
    \begin{enumerate}
    \item If $\lambda \leq a_p n^{-p/4} D^{-(p - 2)/4}$, then $\Adv_{\leq D}(\QQ = \Wig, \PP) = O(1)$.
    \item If $\lambda \geq b_p n^{-p/4} D^{-(p - 2)/4}$ and $D = \omega(1)$, then $\Adv_{\leq D}(\QQ = \Wig, \PP) = \omega(1)$.
    \end{enumerate}
\end{theorem}

First we compute the expectation of a graph cumulant in the spiked model.
In fact, the cumulant of $v^{\otimes p}$ is the same as its ordinary graph moment $m_G(v^{\otimes p})$ given by Proposition~\ref{prop:all-spike} with a correction term.

\begin{proposition}
\label{prop:all-spike-cumulant}
Let $v \in \RR^n$. If $G=(V,E)$ with $|V|=a$ and $|E|=b$, then
\[
\kappa_G(v^{\otimes p}) = \frac{n^{\underline{b}}}{(n + 2b - 2)^{\dunderline{b}}}\|v\|^{2b}.
\]
\end{proposition}

\begin{proof}
Recall that $\kappa_G(T) = \Exp_Q m^!_G(Q \cdot T)$, and that $Q \cdot v^{\otimes p} = (Q \cdot v)^{\otimes p}$. Since $Q$ is Haar-random, $u = Q \cdot v$ is uniformly chosen from the sphere of radius $\|v\|$. As in Proposition~\ref{prop:all-spike}, we have a copy of $u$ on the outgoing half-edges of each vertex. But  rather than the inner product $\langle u,u \rangle = \|v\|^2$ on each edge, we have $\langle u,e \rangle \langle e, u \rangle = \langle e, u \rangle^2$ where $e$ is the basis vector associated with that edge. Since all
$n^{\underline{b}}$
tuples of distinct indices will contribute the same expectation, we can fix an orthogonal basis $e_1,\ldots,e_b$ and write
\begin{equation}
\label{eq:spike-cumulant}
\kappa_G(v^{\otimes p})
=
n^{\underline{b}}
\,\Exp_u \prod_{i=1}^b \langle e_i, u \rangle^2 \, .
\end{equation}
We can write this product as an inner product of tensors,
\begin{equation}
\label{eq:spike-inner}
\Exp_u \prod_{i=1}^b \langle e_i, u \rangle^2
= \left\langle
\left( \bigotimes_{i=1}^b e_i \right)^{\!\otimes 2} , \,
\Exp_u u^{\otimes 2b}
\right\rangle \,.
\end{equation}

Since $\Exp_u u^{\otimes 2b}$ is fixed under conjugation by any $Q \in \sO(n)$, it lies in the trivial subspace of $\RR^{\otimes 2b}$. As in the proof of Theorem~\ref{thm:tensor-invariant}, this implies that it is a linear combination of matchings of the $2b$ half-edges, each of which corresponds to a ``rewiring'' of $G$ \`{a} la the configuration model. Moreover, by symmetry, all of these matchings have equal weight. Using the representation theory of the orthogonal group (e.g.~\cite{MR-2011-GraphIntegralCircuit}) we can obtain
\begin{equation}
\label{eq:moore-russell}
\Exp_u u^{\otimes 2b} =
\frac{\|u\|^{2b}}{n(n+2)(n+4)\cdots(n+2b-2)}
\sum_{\textrm{$\mu$: matching of $[2b]$}} w_\mu \, .
\end{equation}
This is analogous to Isserlis's or Wick's theorem for the Gaussian measure~\cite{isserlis,wick} where the denominator would simply be $n^b$.

However, since the $e_i$ are orthogonal, any matching $\mu$ that partners one with another yields a zero inner product. Thus the only nonzero contribution to~\eqref{eq:spike-inner} comes from the matching given by the original wiring of $G$, where for each $i \in [b]$ the two copies of $e_i$ in $(\bigotimes_i e_i)^{\otimes 2}$ are matched with each other. That matching contributes $\prod_i \langle e_i, e_i \rangle = 1$ to~\eqref{eq:spike-inner}, and combining this with~\eqref{eq:spike-cumulant} and~\eqref{eq:moore-russell} gives
\[
\kappa_G(v^{\otimes p})
= \|v\|^{2b}
\frac{n(n-1)(n-2)\cdots(n-b+1)}
{n(n+2)(n+4)\cdots(n+2b-2)}
= \|v\|^{2b}
\frac{n^{\underline{b}}}{(n+2b-2)^{\dunderline{b}}} \, .
\]
This correction factor is roughly $\e^{-(3/2)b^2/n}$, and in any case is $1-O(b^2/n)$.
\end{proof}

\begin{proposition}
    \label{prop:count-multigraphs}
    The number of (non-isomorphic, unlabelled) $p$-regular multigraphs on $d$ vertices is asymptotically as $d \to \infty$
    \begin{equation}
        |\sG_{d, p}| \sim \frac{(pd - 1)!!}{d! p!^d} \sim \sqrt{\frac{p}{\pi}} \left(\frac{p^{p/2}}{p!} e^{-\frac{p - 2}{2}} d^{\frac{p - 2}{2}}\right)^d.
    \end{equation}
    We also have the concrete upper bound:
    \begin{equation}
        |\sG_{d, p}| \leq \frac{(pd - 1)!!}{d! p!^d} \leq \left(e^{p + 1} p^{-p/2} d^{\frac{p - 2}{2}}\right)^d.
    \end{equation}
\end{proposition}
\begin{proof}
    The asymptotics follow from the proof of \cite{Bollobas-1982-AsymptoticUnlabelledRegularGraphs}, repeated without the restriction to simple graphs, and Stirling's approximation.
    The initial upper bound follows since every $p$-regular multigraph on $d$ vertices corresponds to at least $d! p!^d$ perfect matchings of $[pd]$, as in the configuration model.
    The second bound follows by a non-asymptotic version of Stirling's approximation.
\end{proof}

\begin{proof}[Proof of Theorem~\ref{thm:tensor-pca-detection-formal}]
We begin by computing the centered cumulants under $\PP$:
\begin{align*}
    \Ex_{T \sim \PP} \kappa_G^c(T)
    &= \Ex_{\substack{v \sim \Unif(\SS^{n -1}(\sqrt{n})) \\ W \sim \Wig}} \kappa_G^c(\lambda v^{\otimes p} + W) \\
    &= \Ex_{W \sim \Wig} \kappa_G^c(W) + \Ex_{v \sim \Unif(\SS^{n - 1}(\sqrt{n}))} \kappa_G(\lambda v^{\otimes p}) \\
    &= \Ex_{v \sim \Unif(\SS^{n - 1}(\sqrt{n}))} \kappa_G(\lambda v^{\otimes p}) \tag{Proposition~\ref{prop:wigner-centered-cumulants-zero}} \\
    &\asymp \lambda^{d} n^{pd/2} \,.\tag{Proposition~\ref{prop:all-spike-cumulant}}
\end{align*}
Now, substituting this into Corollary~\ref{cor:invariant-advantage-bound},
\begin{align*}
    \Adv_{\leq D}(\Wig, \PP)
    &\asymp \sum_{d = 0}^D \frac{1}{n^{pd/2}} \cdot \lambda^{2d} n^{pd} \sum_{|V(G)| = d}
    \frac{1}{|\eAut(G)|} \\
    &\leq \sum_{d = 0}^D d! \left(\lambda^2 n^{p/2}\right)^d \#\{p\text{-regular multigraphs on } d \text{ vertices}\}
    \intertext{and bounding the number of $p$-regular multigraphs by Proposition~\ref{prop:count-multigraphs}, we have}
    &\leq \sum_{d = 0}^D d! \left(\lambda^2 n^{p/2}\right)^d \frac{(pd - 1)!!}{d!} \\
    &\leq \sum_{d = 0}^D \left(\lambda^2 n^{p/2}\right)^d \frac{pd^{pd/2}}{(d/e)^d} \\
    &\leq \sum_{d = 0}^D \left(\lambda^2 \cdot e^{p + 1} p^{-p/2} d^{(p - 2)/2} n^{p/2}\right)^d,
\end{align*}
and the result follows since, under the stated assumptions and by a suitable choice of the constant $a_p$, the quantity being raised to a power will be, say, smaller than $1/2$.
For the lower bound, we may use our asymptotic formula for the advantage and use that a positive proportion of $d$-regular multigraphs are simple and almost all have trivial automorphism group (see, e.g., \cite{Bollobas-1982-AsymptoticUnlabelledRegularGraphs}).
\end{proof}

We elaborate on a remark from the Introduction: the factor of $d^{\frac{p - 2}{2}d}$ in the final summation governs the tradeoff between the power of detection algorithms and their subexponential runtime budget.
In our calculation, this quantity has a clear and direct source: it is just the asymptotic value of the number of non-isomorphic $p$-regular multigraphs (Proposition~\ref{prop:count-multigraphs}), and therefore also the dimension of the space of degree $d$ invariant polynomials.

\subsection{Reconstruction with Low-Degree Polynomials}

We first state the formal result we will show.
\begin{theorem}[Low-degree lower bound for tensor PCA reconstruction]
    \label{thm:tensor-pca-recovery-formal}
    Let $D = D(n) \in \NN$ have $D \leq \sqrt{n / 2p^2}$.
    For all odd $p \geq 3$, there is a constant $c_p > 0$ such that if $\lambda \leq c_p n^{-p/4} D^{-(p - 2) / 4}$, then $\Corr_{\leq D}(\PP)^2 = O(\sqrt{n})$, and therefore $\MMSE_{\leq D}(\PP) = n - O(\sqrt{n})$.
\end{theorem}
\begin{remark}[Better-than-random reconstruction]
    The $O(\sqrt{n})$ scaling of the correlation even in the computationally hard regime is perhaps surprising, since a uniformly random estimator $\what{v} \in \SS^{n - 1}(\sqrt{n})$ achieves squared correlation with the signal $v \in \SS^{n - 1}(\sqrt{n})$ of $\langle v, \what{v} \rangle^2 / \|\what{v}\|^2 = O(1)$ with high probability.
    Indeed, estimation slightly better than this is always possible: consider the estimator $\what{v} = \what{v}(Y)$ formed by the open graph moment $m_{G \to}(Y)$ where $G$ is the graph on one vertex with $(p - 1) / 2$ self-loops and one open edge (for $p$ odd).
    This is a linear function, so $m_{G\to}(\lambda v^{\otimes p} + W) = \lambda \|v\|^{p - 1} v + m_{G \to}(W) = \lambda n^{(p - 1) / 2} v + m_{G \to}(W)$.
    Let us write $g \colonequals m_{G \to}(W)$.
    This is independent of $v$ and distributed roughly as $\sN(0, n^{(p - 1)/2}I)$.
    Thus, proceeding heuristically, the squared correlation is typically
    \begin{equation}
        \frac{\langle v, \what{v} \rangle^2}{\|\what{v}\|^2} \approx \frac{\lambda^2 n^{p + 1}}{\lambda^2 n^{p} + n^{(p + 1)/2}} \,.
    \end{equation}
    When $\lambda \ll n^{-p/4}$, the first term in the denominator is smaller than the second, so this is roughly $\lambda^2 n^{(p + 1)/2}$, which when $\lambda \sim n^{-p/4}$ becomes as large as $n^{1/2}$, as claimed.
    Indeed, once $\lambda \gg n^{-(p - 1)/4}$, this expression is $\Omega(n)$, so this simple estimator achieves linear correlation.
\end{remark}

We begin with a number of preliminaries.
We will use the following analog of Proposition~\ref{prop:count-multigraphs} for counting open multigraphs.
\begin{proposition}
    \label{prop:count-open-multigraphs}
    The number of (non-isomorphic, unlabelled) 1-open $p$-regular multigraphs on $d$ vertices is bounded by
    \begin{equation}
        |\sG_{d, p\to}| \leq \frac{(pd - 2)!!}{(d - 1)! p!^{d - 1}(p - 1)!} \leq \left((2e)^{p + 1} (p - 1)^{-(p - 1)/2} d^{\frac{p - 2}{2}}\right)^d .
    \end{equation}
\end{proposition}
\begin{proof}
    The initial upper bound follows first by associating a 1-open multigraph to a closed multigraph with all vertices having degree $p$ except for one having degree $p - 1$.
    The total number of edges in this graph is $\frac{pd - 1}{2}$, and the remaining bounds follow as in Proposition~\ref{prop:count-multigraphs}.
\end{proof}

\begin{proposition}
    For any 1-open $p$-regular multigraph $G$ and any $T \in \Sym^p(\RR^n)$,
    \begin{equation}
        \Ex_{v \sim \Unif(\SS^{n - 1}(\sqrt{n}))} \kappa_{G\to}^c(\lambda v^{\otimes p} + T) = \sum_{\substack{G = G_A \sqcup G_B \\ G_A \text{ open} \\ G_B \text{ closed}}} \frac{n^{\underline{b - 1}}}{n^{\underline{b_A - 1}} \, (n + 2b_B - 2)^{\dunderline{b_B}}} \ \lambda^{d_B} n^{b_B} \kappa_{G_A\to}^c(T).
    \end{equation}
\end{proposition}
\begin{proof}
    We combine Proposition~\ref{prop:additivity-cumulants} with Proposition~\ref{prop:all-spike-cumulant}:
    \begin{align*}
        \Ex_{v \sim \Unif(\SS^{n - 1}(\sqrt{n}))} \kappa_{G\to}^c(\lambda v^{\otimes p} + T)
        &= \sum_{\substack{G = G_A \sqcup G_B \\ G_A \text{ open} \\ G_B \text{ closed}}} \frac{n^{\underline{b - 1}}}{n^{\underline{b_A - 1}}\, n^{\underline{b_B}}} \left(\Ex_v\kappa_{G_B}(\lambda v^{\otimes p})\right) \kappa_{G_A\to}^c(T) \\
        &= \sum_{\substack{G = G_A \sqcup G_B \\ G_A \text{ open} \\ G_B \text{ closed}}} \frac{n^{\underline{b - 1}}}{n^{\underline{b_A - 1}} \, n^{\underline{b_B}}} \cdot \lambda^{d_B} \cdot \frac{n^{\underline{b_B}}}{(n + 2b_B - 2)^{\dunderline{b_B}}} n^{b_B} \kappa_{G_A\to}^c(T),
    \end{align*}
    and simplifying gives the stated result.
\end{proof}

\begin{proposition}
    \label{prop:reconstruction-cumulant-signal}
    For any 1-open $p$-regular multigraph $G$,
    \begin{equation}
        \Ex_{\substack{v \sim \Unif(\SS^{n - 1}(\sqrt{n})) \\ W \sim \Wig}} \langle v, \kappa_{G\to}^c(\lambda v^{\otimes p} + W)\rangle  = \frac{n^{\underline{b}}}{(n + 2b - 2)^{\dunderline{b}}} \lambda^d n^b.
    \end{equation}
\end{proposition}
\begin{proof}
    Proceeding similarly to the above using Proposition~\ref{prop:additivity-cumulants} with respect to the randomness in $W$,
    \begin{align*}
        &\hspace{-0.5cm}\Ex_{\substack{v \sim \Unif(\SS^{n - 1}(\sqrt{n})) \\ W \sim \Wig}} \langle v, \kappa_{G\to}^c(\lambda v^{\otimes p} + W)\rangle \\
        &= \Ex_{\substack{v \sim \Unif(\SS^{n - 1}(\sqrt{n})) \\ W \sim \Wig}} \sum_{\substack{G = G_A \sqcup G_B \\ A \text{ open} \\ B \text{ closed}}} \frac{n^{\underline{b - 1}}}{n^{\underline{b_A - 1}}n^{\underline{b_B}}} \kappa_{G_B}^c(W) \langle v, \kappa_{G_A\to}(\lambda v^{\otimes p}) \rangle
        \intertext{where the only non-zero term is when $B = \emptyset$, so that}
        &= \Ex_{v \sim \Unif(\SS^{n - 1}(\sqrt{n}))} \langle v, \kappa_{G\to}(\lambda v^{\otimes p}) \rangle \\
        &= \lambda^d \Ex_{v \sim \Unif(\SS^{n - 1}(\sqrt{n}))} \langle v, \kappa_{G\to}(v^{\otimes p}) \rangle \\
        &= \lambda^d \Ex_{v \sim \Unif(\SS^{n - 1}(\sqrt{n}))} \langle v, m_{G\to}^!(v^{\otimes p}) \rangle
        \intertext{and now the same calculations as in Proposition~\ref{prop:all-spike-cumulant} apply, giving}
        &= \frac{n^{\underline{b}}}{(n + 2b - 2)^{\dunderline{b}}} \lambda^d n^b,
    \end{align*}
    which is the final result.
\end{proof}

\begin{proof}[Proof of Theorem~\ref{thm:tensor-pca-recovery-formal}]
Recall that we are interested in producing an upper bound on the correlation, which we may restrict to equivariant polynomials without loss of generality by Proposition~\ref{prop:corr-equivariant}.
\begin{equation}
\Corr_{\leq D}(\PP) = \sup_{\substack{f \in \RR[Y]_{\leq D}^n \\ f \text{ equivariant} \\ \EE_{Y \sim \PP} \|f(Y)\|^2 \neq 0}} \frac{\EE_{(v, Y) \sim \PP} \langle v, f(Y) \rangle}{\sqrt{\EE_{Y \sim \PP} \|f(Y)\|^2}}.
\end{equation}
Consider a particular equivariant $f$.
By Theorem~\ref{thm:tensor-equivariant} and Lemma~\ref{lem:kappaGc-1open-gram}, it admits an expansion in the $\what{\kappa_{G\to}^c}$,
\begin{equation}
    f(T) = \sum_{G} \alpha_G \what{\kappa_{G\to}^c}(T) = \sum_G \frac{n^{b_G/2}}{n^{\underline{b_G}} \cdot \sqrt{|\eAut(\chop(G)|}} \alpha_G \kappa_{G\to}^c(T).
\end{equation}
where the sum is over $p$-regular 1-open multigraphs on at most $D$ vertices.
We now apply the technique of \cite{SW-2020-LowDegreeEstimation}, using Jensen's inequality to give a lower bound on the norm of $f$:
\begin{align}
    \Ex_{Y \sim \PP} \|f(Y)\|^2
    &= \Ex_{\substack{v \sim \Unif(\SS^{n - 1}(\sqrt{n})) \\ W \sim \Wig}} \| f(\lambda v^{\otimes p} + W) \|^2 \\
    &\geq \Ex_{W \sim \Wig} \left\| \Ex_{v \sim \Unif(\SS^{n - 1}(\sqrt{n}))} f(\lambda v^{\otimes p} + W) \right\|^2.
\end{align}
We can write the inner quantity as:
\begin{align*}
    &\Ex_{v \sim \Unif(\SS^{n - 1}(\sqrt{n}))} f(\lambda v^{\otimes p} + W) \\
    &= \sum_{G} \frac{n^{b_G/2}}{n^{\underline{b_G}} \cdot \sqrt{|\eAut(\chop(G))|}} \alpha_G \Ex_{v \sim \Unif(\SS^{n - 1}(\sqrt{n}))} \kappa_{G\to}^c(\lambda v^{\otimes p} + T) \\
    &= \sum_{G} \frac{n^{b_G/2}}{n^{\underline{b_G}} \cdot \sqrt{|\eAut(\chop(G))|}} \alpha_G \sum_{\substack{G = G_A \sqcup G_B \\ G_A \text{ open} \\ G_B \text{ closed}}} \frac{n^{\underline{b_G - 1}}}{n^{\underline{b_A - 1}}(n + 2b_B - 2)^{\dunderline{b_B}}} \ \lambda^{d_B} n^{b_B} \kappa_{G_A\to}^c(T) \\
    &= \sum_{G} \frac{n^{b_G/2}}{(n - b_G + 1) \cdot \sqrt{|\eAut(\chop(G))|}} \alpha_G \sum_{\substack{G = G_A \sqcup G_B \\ G_A \text{ open} \\ G_B \text{ closed}}} \frac{1}{n^{\underline{b_A - 1}}(n + 2b_B - 2)^{\dunderline{b_B}}} \ \lambda^{d_B} n^{b_B} \kappa_{G_A\to}^c(T) \intertext{}
    &= \sum_{G} \frac{n^{b_G/2}}{(n - b_G + 1) \cdot \sqrt{|\eAut(\chop(G))|}} \alpha_G \sum_{\substack{G = G_A \sqcup G_B \\ G_A \text{ open} \\ G_B \text{ closed}}} \frac{n^{\underline{b_A}}\sqrt{|\eAut(\chop(G_A))|}}{n^{b_A / 2} n^{\underline{b_A - 1}}(n + 2b_B - 2)^{\dunderline{b_B}}} \ \lambda^{d_B} n^{b_B} \what{\kappa_{G_A\to}^c}(T) \\
    &= \sum_{G} \frac{1}{(n - b_G + 1) \cdot \sqrt{|\eAut(\chop(G))|}} \alpha_G \\
    &\hspace{1cm} \sum_{\substack{G = G_A \sqcup G_B \\ G_A \text{ open} \\ G_B \text{ closed}}} \frac{(n - b_A + 1)\sqrt{|\eAut(\chop(G_A))|}}{(n + 2b_B - 2)^{\dunderline{b_B}}} \ \lambda^{d_B} n^{3b_B/2} \what{\kappa_{G_A\to}^c}(T).
\end{align*}
Switching the order of summations, we find that we can write
\begin{equation}
    \Ex_{v \sim \Unif(\SS^{n - 1}(\sqrt{n}))} f(\lambda v^{\otimes p} + W) = \sum_G (L\alpha)_G \cdot \what{\kappa_{G\to}^c}(T) = \sum_G \left(\sum_H R_{G, H} \alpha_H\right) \what{\kappa_{G\to}^c}(T).
\end{equation}
Here, $R$ is an upper-triangular matrix when graphs are ordered by increasing number of vertices, with rows and columns indexed by $\sG_{1, p\to} \sqcup \cdots \sqcup \sG_{D, p\to}$.
More specifically, we have $R_{G, H} \neq 0$ only when a subset of the connected components of $H$ yield $G$.
When this is so, then the entries of $L$ are given by
\begin{equation}
    R_{G, H} = c_{G, H} \cdot \frac{n - b_G + 1}{n - b_H + 1} \cdot \sqrt{\frac{|\eAut(\chop(G))|}{|\eAut(\chop(H))|}} \cdot \frac{n^{3b_{H \setminus G} / 2}}{(n + 2b_{H \setminus G} - 2)^{\dunderline{b_{H \setminus G}}}} \cdot \lambda^{d_{H \setminus G}},
\end{equation}
where $c_{G, H}$ is the number of ways that it is possible to choose the connected components of $G$ out of those of $H$ (i.e., the product of $\binom{m_H(C)}{m_G(C)}$ over all connected components $C$ of $G$, where $m_G(C)$ is the number of times $C$ occurs in $G$ and likewise for $m_H(C)$).

Now, using Lemma~\ref{lem:kappaGc-1open-inner} on the near-orthonormality of the $\what{\kappa_{G \to}^c}$, we find
\begin{equation}
    \Ex_{Y \sim \PP} \|f(Y)\|^2 \geq \frac{1}{2} \|R\alpha\|^2 = \frac{1}{2}\alpha^{\top} (R^{\top}R) \alpha.
\end{equation}
We define and rewrite using Proposition~\ref{prop:reconstruction-cumulant-signal}
\begin{align*}
    \beta_G
    &\colonequals \Ex_{\substack{v \sim \Unif(\SS^{n - 1}(\sqrt{n})) \\ W \sim \Wig}} \langle v, \what{\kappa_{G\to}^c}(\lambda v^{\otimes p} + W)\rangle \\
    &= \frac{n^{b_G/2}}{n^{\underline{b_G}} \sqrt{|\eAut(\chop(G))|}}\frac{n^{\underline{b_G}}}{(n + 2b_G - 2)^{\dunderline{b_G}}} \lambda^{d_G} n^{b_G} \\
    &= \frac{n^{3b_G/2}}{(n + 2b_G - 2)^{\dunderline{b_G}}\sqrt{|\eAut(\chop(G))|}} \lambda^{d_G}.
\end{align*}

We have $R_{G, G} = 1$ for all $G$, so $R$ is invertible.
For our bound on the correlation, we then may compute
\begin{equation}
    \Corr_{\leq D}(\PP)^2 \leq 2 \sup_{\alpha \neq 0} \frac{\langle \alpha, \beta \rangle^2}{\|R\alpha\|^2} = 2 \|R^{\top^{-1}}\beta\|^2.
\end{equation}
Let us compute the inverse $\gamma \colonequals R^{\top^{-1}}\beta$.
$R^{\top}$ is lower triangular and its diagonal is identically equal to 1, so we may compute the inverse recursively by back substitution:
\begin{equation}
    \gamma_H = \beta_H - \sum_{G < H} R_{G, H} \gamma_G.
\end{equation}
We may bound both the entries of $\beta$ and of $R$ as:
\begin{align*}
    |\beta_G|
    &\lesssim \frac{n^{b_G / 2}\lambda^{d_G}}{\sqrt{|\eAut(\chop(G))|}} \\
    &= n^{1/4}\frac{\left(n^{p/4} \lambda\right)^{d_G}}{\sqrt{|\eAut(\chop(G))|}} \\
    &\leq n^{1/4}\frac{\left(c_p D^{-\frac{p - 2}{4}}\right)^{d_G}}{\sqrt{|\eAut(\chop(G))|}},
    \intertext{}
    |R_{G, H}| &\leq \sqrt{\frac{|\eAut(\chop(G))|}{|\eAut(\chop(H))|}} c_{G, H} n^{b_{H \setminus G} / 2} \lambda^{d_{H \setminus G}} \\
    &= \sqrt{\frac{|\eAut(\chop(G))|}{|\eAut(\chop(H))|}} c_{G, H} \left( n^{p/4} \lambda\right)^{d_{H \setminus G}} \\
    &\leq \sqrt{\frac{|\eAut(\chop(G))|}{|\eAut(\chop(H))|}} c_{G, H} \left( c_p D^{-\frac{p - 2}{4}}\right)^{d_{H \setminus G}}.
\end{align*}
Thus inductively we have, with $C(|\conn(H)|)$ denoting the number of distinct chains of subsets of the connected components of $H$,
\begin{align*}
    |\gamma_H|
    &\leq |\beta_H| + \sum_{G < H} |R_{G, H}| \cdot |\gamma_G| \\
    &\lesssim n^{1/4} \frac{C(|\conn(H)|)}{\sqrt{|\eAut(\chop(H))|}} \left( c_p D^{-\frac{p - 2}{4}}\right)^{d_H}
    \intertext{and bounding the number of chains using Proposition~\ref{prop:chains},}
    &\leq n^{1/4} \frac{3^{|\conn(H)|} (|\conn(H)|)!}{{\sqrt{|\eAut(\chop(H))|}}} \left( c_p D^{-\frac{p - 2}{4}}\right)^{d_H}.
\end{align*}

Note that every closed connected component of any $H \in \sG_{d, p\to}$ must have size at least 2 since $p$ is odd, and only the open connected component can have size 1.
Therefore, the number of connected components, if there are $d$ vertices, is at most $1 + \frac{d - 1}{2} = \frac{d + 1}{2}$.
Also, if $c_1, \dots, c_m$ are the frequencies with which various closed connected components happen in $G$, then, since permutations of connected components are automorphisms and the ``chop'' operation does not affect the closed components of a graph,
\begin{equation}
    |\eAut(\chop(H))| \geq \prod_{i = 1}^m c_i!.
\end{equation}
More specifically, suppose the total number of connected components of size 2 in $H$ is $\ell$.
There are $\frac{p + 1}{2}$ non-isomorphic connected $p$-regular graphs on two vertices (two vertices with an odd number of edges between them and a suitable number of loops on the ends).
Suppose the numbers of these components are $c_1, \dots, c_{(p + 1)/2}$.
We then have by a convexity argument
\begin{equation}
    |\eAut(\chop(H))| \geq \prod_{i = 1}^{(p + 1)/2} c_i! \geq \left(\frac{2\ell}{p + 1}\right)!^{\frac{p + 1}{2}} \geq \left(\frac{2}{e(p + 1)}\ell\right)^{\ell}.
\end{equation}

We may organize the calculation of the norm of $\gamma$ according to the number of connected components $k$, and the sizes $a_1$ of the open component and $a_2, \dots, a_k$ of the closed components:
\begin{align*}
    \|\gamma\|^2
    &\lesssim n^{1/2} \sum_{d = 1}^D \left( c_p^2 D^{-\frac{p - 2}{2}}\right)^{d} \sum_{k = 1}^{\frac{d + 1}{2}} 9^k k!^2 \\
    &\hspace{1.5cm}\sum_{\substack{a_1 \geq 1 \\ a_2, \dots, a_k \geq 2 \\ a_1 + a_2 + \cdots + a_k = d}} \left(\frac{2}{e(p + 1)}|\{i: a_i = 2\}|\right)^{-|\{i: a_i = 2\}|} |\sG_{a_1, p\to}| \prod_{i = 2}^k |\sG_{a_i, p}|
    \intertext{and in bounding the sizes of the sets of graphs by Propositions~\ref{prop:count-multigraphs} and \ref{prop:count-open-multigraphs}, let us replace $c_p^2$ with a larger $c_p^{\prime}$ that absorbs constants depending only on $p$, so that we obtain}
    &\leq n^{1/2} \sum_{d = 1}^D \left( c_p^{\prime} D^{-\frac{p - 2}{2}}\right)^{d} \left(9|\sG_{d, p\to}| + \sum_{k = 2}^{\frac{d + 1}{2}} 9^k k!^2 \sum_{\substack{a_1 \geq 1 \\ a_2, \dots, a_k \geq 2 \\ a_1 + a_2 + \cdots + a_k = d}} |\{i: a_i = 2\}|^{-|\{i: a_i = 2\}|} \prod_{i = 1}^k a_i^{\frac{p - 2}{2}a_i}\right)
    \intertext{Here, again by a convexity argument, the choice of $a_i$ that maximizes the expression in the inner sum will have the form $a_1 = 1$, $a_2 = \cdots = a_{s + 1} = 2$, $a_{s + 2} = \cdots = a_{k - 1} = 3$, $a_k = d - 3k + s + 5$. The value of a term with these choices is at most $3^{\frac{p - 2}{2}d} s^{-s}d^{\frac{p - 2}{2}(d - 3k + s + 5)}$. Since $d^{\frac{p - 2}{2}} \geq s$, to maximize this bound we take $s$ as large as possible, i.e., $s = k - 2$. Thus we may bound, again replacing $c_p^{\prime}$ with $c_p^{\prime\prime}$ and absorbing a constant depending only on $p$,}
    &\leq n^{1/2} \sum_{d = 1}^D \left( c_p^{\prime\prime} D^{-\frac{p - 2}{2}}\right)^{d} \left(d^{\frac{p - 2}{2}d} + \sum_{k = 2}^{\frac{d + 1}{2}} k!^2 (k - 2)^{-(k - 2)}d^{\frac{p - 2}{2}(d - 2k + 3)}\sum_{\substack{a_1 \geq 1 \\ a_2, \dots, a_k \geq 2 \\ a_1 + a_2 + \cdots + a_k = d}} 1\right)
    \intertext{Here, the remaining sum is at most the number of integer partitions of $d$, which per Proposition~\ref{prop:integer-partitions} is of order $\exp(O(\sqrt{d}))$. We have $k! / (k - 2)^{k - 2} \lesssim k^2 \leq d^2$, and $k! d^{-(p - 2)k} \leq k^k d^{-k} \leq 1$. Putting everything together, we find}
    &\leq n^{1/2} \sum_{d = 1}^D \left( c_p^{\prime\prime} e^{O(1 / \sqrt{d})}\right)^{d} \left(1 + d^{\frac{3}{2}(p - 2) + 3}\right).
\end{align*}
Finally, for $c_p^{\prime\prime}$ sufficiently small, the remaining summation is bounded by a series of the form $\sum_{d \geq 1} (1 - \eps)^d d^K$ for some constant $K$, which is a convergent series.
Thus we find $\Corr_{\leq D}(\PP)^2 \leq 2\|\gamma\|^2 = O(n^{1/2})$, completing the proof.
\end{proof}

\section{Application 2: Distinguishing Wigner from Wishart Tensors}
\label{sec:clt}

We first state the result we will prove formally (combining the upper and lower bound claims from the Introduction in Theorems~\ref{thm:wishart-informal} and~\ref{thm:wishart-informal-upper}).
\begin{theorem}[Low-degree analysis for Wigner vs.\ Wishart detection]
    \label{thm:clt}
    Suppose that $\mu_n$ are probability measures on $\Sym^p(\RR^n)$ satisfying the following properties for $A \sim \mu_n$:
    \begin{enumerate}
    \item For all $i \in [n]^p$ having a repeated entry, $A_{i_1, \dots, i_p} = 0$ almost surely.
    \item There is a constant $C > 0$ such that, for all $i \in [n]^p$, $|A_{i_1, \dots, i_p}| \leq C$ almost surely.
    \item $\|A\|_F^2 = n^p$ almost surely.
    \end{enumerate}
    Let $Z_1, \dots, Z_r \sim \Gin(n, 1/n)$ be i.i.d.\ and $A_1, \dots, A_r \sim \mu_n$ be i.i.d., and write $\PP = \PP_{n, r}$ for the law of $r^{-1/2}\sum_{j = 1}^r Z_j \cdot A_j$.
    There is a constant $a_{p, C} > 0$ such that the following holds.
    Suppose that $D = D(n) \leq \sqrt{n / 2p^2}$ is given and $r = r(n)$ satisfies
    \begin{equation}
        r \geq a_{p, C} \cdot \left\{\begin{array}{ll} n^p & \text{if } p \text{ is odd}, \\ n^{\frac{3}{2}p} & \text{if } p \text{ is even}\end{array}\right\}.
    \end{equation}
    Then, $\Adv_{\leq D}(\QQ = \Wig, \PP_{n, r(n)}) = O(1)$.

    Further, suppose that $\mu_n$ is supported on a single tensor $A$ having all entries with no repeated indices in their position equal to $c = c(p, n) > 0$.
    Let $D = 3$ if $p$ is even and $D = 4$ if $p$ is odd, and suppose that $r = r(n)$ satisfies
    \begin{equation}
        r  \ll \left\{\begin{array}{ll} n^p & \text{if } p \text{ is odd}, \\ n^{\frac{3}{2}p} & \text{if } p \text{ is even}\end{array}\right\}.
    \end{equation}
    Then, $\Adv_{\leq D}(\QQ = \Wig, \PP_{n, r(n)}) = \omega(1)$.
\end{theorem}

We establish some preliminary facts about the Ginibre ensemble.
\begin{proposition}
    \label{prop:Gin-invariant}
    $\Gin(n, \sigma^2)$ is orthogonally invariant: if $Z \sim \Gin(n, \sigma^2)$ and $Q \in \sO(n)$, then $QZ$ and $ZQ$ both again have the law $\Gin(n, \sigma^2)$.
\end{proposition}
\begin{corollary}
    \label{cor:ZT-invariant}
    For any random tensor $T$, the law of $Z \cdot T$ when $Z \sim \Gin(n, 1/n)$ independently of $T$ is orthogonally invariant.
\end{corollary}

A version of the Weingarten formula for $\EE_{Q \sim \Haar(n)} Q^{\otimes \ell}$ (see Appendix~\ref{app:weingarten}) also holds for tensor powers of $Z \sim \Gin(n, \sigma^2)$.
In fact, the formula is simpler: it only contains the leading order terms of the Weingarten formula.
\begin{proposition}
    \label{prop:Gin-wick}
    $\EE_{Z \sim \Gin(n, \sigma^2)} Z^{\otimes \ell} = \sigma^{\ell} \sum_{\mu} w(\mu) \otimes w(\mu)$.
\end{proposition}
\begin{proof}
    The result follows directly from an application of the Isserlis--Wick formula.
\end{proof}
\noindent
This simple fact implies remarkable simplifications for the cumulants of a tensor formed as $Z \cdot T$.
\begin{corollary}
    \label{cor:cumulants-after-Gin}
    Let $G$ be a $p$-regular multigraph on $d$ vertices.
    Then,
    \begin{equation}
        \Ex_{Z \sim \Gin(n, 1/n)} \kappa_G(Z \cdot T) = \frac{n^{\underline{pd/2}}}{n^{pd/2}} m_G(T).
    \end{equation}
    Moreover, let $G^{(0)} \colonequals G \setminus \Frob(G)$ be $G$ with all Frobenii removed.
    Then,
    \begin{equation}
        \Ex_{Z \sim \Gin(n, 1/n)} \kappa_G^c(Z \cdot T) = \frac{n^{\underline{pd/2}}}{n^{pd/2}} m_{G^{(0)}}(T) (\|T\|_F^2 - n^p)^{|\Frob(G)|}.
    \end{equation}
\end{corollary}
\begin{proof}
    Write $b \colonequals pd/2$ as usual.
    For the first claim, we observe by Corollary~\ref{cor:ZT-invariant} that
    \begin{align*}
        \Ex_{Z \sim \Gin(n, 1/n)} \kappa_G(Z \cdot T)
        &= \Ex_{Z \sim \Gin(n, 1/n)} m_G^!(Z \cdot T) \\
        &= \Ex_{Z \sim \Gin(n, 1/n)} \sum_{i_1, \dots, i_{b} \text{ distinct}} \langle e_{i_1}^{\otimes 2} \otimes \cdots \otimes e_{i_b}^{\otimes 2}, T^{\otimes d} Z^{\otimes pd} \rangle \\
        &= \frac{1}{n^b} \sum_{i_1, \dots, i_{b} \text{ distinct}} \sum_{\mu} \langle e_{i_1}^{\otimes 2} \otimes \cdots \otimes e_{i_b}^{\otimes 2}, T^{\otimes d} w(\mu) \otimes w(\mu) \rangle \tag{Proposition~\ref{prop:Gin-wick}}
        \intertext{and here only the matching that corresponds to the one of equal indices in the first term in the inner product contributes, whereby}
        &=  \frac{1}{n^b} \sum_{i_1, \dots, i_{b} \text{ distinct}} m_G(T) \\
        &= \frac{n^{\underline{b}}}{n^b} m_G(T),
    \end{align*}
    as claimed.
    For the second result, we first expand by definition over subsets of Frobenii, and then use the first result on each term:
    \begin{align*}
        \Ex_{Z \sim \Gin(n, 1/n)} \kappa_G^c(Z \cdot T)
        &= \sum_{S \subseteq \Frob(G)} (-1)^{|S|} (n - b + p|S|)^{\underline{p|S|}} \Ex_{Z \sim \Gin(n, 1/n)} \kappa_{G \setminus S}(Z \cdot T) \\
        &= \sum_{S \subseteq \Frob(G)} (-1)^{|S|} (n - b + p|S|)^{\underline{p|S|}} \frac{n^{\underline{b - p|S|}}}{n^{b - p|S|}} m_{G \setminus S}(T)
        \intertext{and by multiplicativity of the ordinary graph moments,}
        &= m_{G^{(0)}}(T) \sum_{S \subseteq \Frob(G)} (-1)^{|S|} (n - b + p|S|)^{\underline{p|S|}} \frac{n^{\underline{b - p|S|}}}{n^{b - p|S|}} (\|T\|_F^2)^{|\Frob(G)| - |S|}
        \intertext{and writing $f \colonequals |\Frob(G)|$ and introducing $s \colonequals |S|$, we have}
        &= m_{G^{(0)}}(T) \sum_{s = 0}^f \binom{f}{s}(-1)^{s} (n - b + ps)^{\underline{ps}} \frac{n^{\underline{b - ps}}}{n^{b - ps}} (\|T\|_F^2)^{f - s} \\
        &= m_{G^{(0)}}(T) \sum_{s = 0}^f \binom{f}{s}(-1)^{s} \frac{(n - b + ps)!}{(n - b)!} \frac{n!}{(n - b + ps)!} \frac{1}{n^{b - ps}} (\|T\|_F^2)^{f - s} \\
        &= \frac{n^{\underline{b}}}{n^b} m_{G^{(0)}}(T) \sum_{s = 0}^f \binom{f}{s}(-n^p)^{s} (\|T\|_F^2)^{f - s} \\
        &= \frac{n^{\underline{b}}}{n^b} m_{G^{(0)}}(T) (\|T\|_F^2 - n^p)^f
    \end{align*}
    by the binomial theorem, completing the proof.
\end{proof}

\begin{proof}[Proof of Theorem~\ref{thm:clt}]
    We first express the expectations of the centered cumulants of $\PP$ in terms of those of $\Gin(n, 1/n) \cdot \mu$, and evaluate the latter using Corollary~\ref{cor:cumulants-after-Gin}:
    \begin{align*}
    \Ex_{T \sim \PP} \kappa_G^c(T)
    &= \Ex_{\substack{Z_i \sim \Gin(n, 1/ n) \\ A_i \sim \mu}} \kappa_G^c \left(\frac{1}{\sqrt{r}} \sum_{i = 1}^r Z_i \cdot A_i \right) \\
    &= \sum_{G_1 \sqcup \cdots \sqcup G_r = G} \frac{n^{\underline{b}}}{n^{\underline{b_1}} \cdots n^{\underline{b_r}}} \Ex_{\substack{Z_i \sim \Gin(n, 1/ n) \\ A_i \sim \mu}} \prod_{i = 1}^r \kappa_{G_i}^c\left(\frac{1}{\sqrt{r}}Z_i \cdot A_i; -\frac{1}{r}\One_{\Frob}\right) \\
    &= r^{-d/2} \sum_{G_1 \sqcup \cdots \sqcup G_r = G} \frac{n^{\underline{b}}}{n^{\underline{b_1}} \cdots n^{\underline{b_r}}} \prod_{i = 1}^r \Ex_{\substack{Z \sim \Gin(n, 1/ n) \\ A \sim \mu}} \kappa_{G_i}^c(Z \cdot A) \\
    &= \frac{n^{\underline{b}}}{n^b} r^{-d/2} \sum_{G_1 \sqcup \cdots \sqcup G_r = G}  \prod_{i = 1}^r \Ex_{A \sim \mu} m_{G_i^{(0)}}(A) (\|A\|_F^2 - n^p)^{|\Frob(G_i)|}.
    \end{align*}
    We make a few observations.

    First, if $G$ contains any self-loops, then some $G_i^{(0)}$ contains a self-loop in every term of the sum.
    For this $i$, by our assumption that $A$ is zero on entries with repeated indices, we have $m_{G_i^{(0)}}(A) = 0$ almost surely.
    Thus the entire sum is zero, so $\Ex_{T \sim \PP} \kappa_G^c(T) = 0$ whenever $G$ has self-loops.

    Second, and similarly, if $G$ contains any Frobenii, then some exponent of $\|A\|_F^2 - n^p$ is positive in each term in the sum.
    By our assumption we have $\|A\|_F^2 - n^p = 0$ almost surely when $A \sim \mu$, so $\Ex_{T \sim \PP} \kappa_G^c(T) = 0$ whenever $G$ has any Frobenii as well.

    Lastly, let us consider the case where $G$ has neither self-loops nor Frobenii.
    We may bound the ordinary graph moments naively by using our assumption that the entries of $A \sim \mu$ are almost surely uniformly bounded by a constant $C > 0$.
    This implies
    \begin{equation}
        |m_{G}(A)| \leq C^{|V(G)|} n^{|E(G)|},
    \end{equation}
    since $m_G(A)$ is a sum of $n^{|E(G)|}$ terms, each of which is a product of $|V(G)|$ factors of size at most $C$.
    Using this, we find
    \begin{align*}
       |\Ex_{T \sim \PP} \kappa_G^c(T)|
       &\leq C^d n^{\underline{b}} r^{-d/2} \sum_{G_1 \sqcup \cdots \sqcup G_r = G} 1 \\
       &\leq C^d n^{b} r^{-d/2 + |\conn(G)|}
    \end{align*}
    where we have used that the number of partitions $G_1 \sqcup \cdots \sqcup G_r = G$ is just the number of assignments of each connected component of $G$ to one of $r$ bins, or $r^{|\conn(G)|}$.

    Since $G$ has no self-loops or Frobenii, it has no connected components on two vertices.
    Note that, depending on the parity of $p$, the smallest possible size of a connected component in $G$ will differ: when $p$ is even then it is 3, but when $p$ is odd then it is 4.
    This will ultimately lead to the different thresholds depending on the parity of $p$ in our result.
    Let us give this number a name:
    \begin{equation}
        \xi = \xi(p) \colonequals \left\{\begin{array}{ll} 3 & \text{if } p \text{ is even}, \\ 4 & \text{if } p \text{ is odd}\end{array}\right\}.
    \end{equation}

    Substituting the above bound into Corollary~\ref{cor:invariant-advantage-bound}, we find
    \begin{align*}
        \Adv_{\leq D}(\Wig, \PP)
        &\lesssim \sum_{d = 0}^D (C^2 r^{-1}n^{p/2})^d \sum_{\substack{G \in \sG_{d, p} \\ |V(K)| \geq \xi \text{ for all } K \in \conn(G)}} r^{2|\conn(G)|} \\
        &\leq \sum_{d = 0}^D (C^2 r^{-1}n^{p/2})^d \sum_{\ell = 1}^{d / \xi} \sum_{\substack{\xi \leq a_1 \leq \cdots \leq a_{\ell} \\ a_1 + \cdots + a_{\ell} = d}} r^{2\ell} \prod_{i = 1}^{\ell} |\sG_{a_i, p}|
        \intertext{and using Proposition~\ref{prop:count-multigraphs}, we have}
        &\leq \sum_{d = 0}^D (C^2 e^{p + 1} p^{-p/2} r^{-1}n^{p/2})^d \sum_{\ell = 1}^{d / \xi}  \sum_{\substack{\xi \leq a_1 \leq \cdots \leq a_{\ell} \\ a_1 + \cdots + a_{\ell} = d}} r^{2\ell} \prod_{i = 1}^{\ell} a_i^{\frac{p - 2}{2}a_i}
        \intertext{and a convexity argument shows that the inner product over $a_i$ is maximized when $a_1 = \cdots = a_{\ell - 1} = \xi$ and $a_{\ell} = d - (\ell - 1)\xi$. We may therefore continue bounding}
        &\leq \sum_{d = 0}^D (C^2 e^{p + 1} p^{-p/2} \xi^{\frac{p - 2}{2}} r^{-1}n^{p/2})^d \sum_{\ell = 1}^{d / \xi}  \sum_{\substack{\xi \leq a_1 \leq \cdots \leq a_{\ell} \\ a_1 + \cdots + a_{\ell} = d}} r^{2\ell} d^{d - (\ell - 1)\xi}
        \intertext{Now, since $d \leq D \leq n^{1/2}$ by assumption, while $r \geq n^3$ (provided we choose $a_{p, C} \geq 1)$, we have $r^2 \geq n^6$ while $d^{\xi} \leq d^4 \leq n^2$, so $r^2 \geq d^{\xi}$ and the largest term of the remaining sum is the one where $\ell = d / \xi$. Using Proposition~\ref{prop:integer-partitions} to bound the number of terms in the sum, we find}
        &\leq \sum_{d = 0}^D (C^2 e^{p + 1} p^{-p/2} \xi^{\frac{p - 2}{2}} r^{-1}n^{p/2})^d \cdot \exp(O(\sqrt{d})) r^{2d / \xi} d^{\xi} \\
        &\leq \sum_{d = 0}^{\infty} d^{\xi} \left(C^2 e^{p + 1} p^{-p/2} \xi^{\frac{p - 2}{2}} \exp(O(1 / \sqrt{d})) \cdot r^{-(1 - \frac{2}{\xi})}n^{p/2}\right)^d.
    \end{align*}

    Thus we find that, provided that $r \geq a_p n^{\frac{p}{2} \cdot \frac{\xi}{\xi - 2}}$, the base of the exponent in the above series will be strictly smaller than 1, so the sum will converge and and the proof will be complete.
    When $p$ is even, then $\xi = 3$ and the exponent of $n$ in the condition above is $3p/2$, while when $p$ is odd, then $\xi = 4$ and the exponent is $p$, so this gives precisely the stated result in both cases.

    Finally, for the upper bound, note that for this choice of deterministic $A$ we have $m_G(A) \asymp n^{|E(G)|}$ for each $G$, since each term in the sum in $m_G(A)$ is equal and close to 1.
    We thus have for any $G$ on $d$ vertices that
    \begin{equation}
        \Ex_{T \sim \PP} \kappa_G^c(T) \asymp n^{\underline{pd / 2}} \, r^{-d/2 + |\conn(G)|}
    \end{equation}
    by the same calculations as above.
    Let us take $d = \xi \in \{3, 4\}$ and $G$ any connected $p$-regular graph on $\xi$ vertices (say, a triangle with every edge repeated $p / 2$ times when $p$ is even, or a complete graph on 4 vertices with each edge of one perfect matching repeated $p - 2$ times and the other edges occurring once when $p$ is odd).
    We have $|\conn(G)| = 1$.
    By Corollary~\ref{cor:invariant-advantage-bound}, we may lower bound the advantage by the contribution of just the term corresponding to this graph $G$, which, since it is of constant size, has $|\eAut(G)| = O(1)$, so
    \begin{align*}
        \Adv_{\leq D}(\Wig, \PP)
        &\gtrsim n^{-p\xi/2} \left(\Ex_{T \sim \PP} \kappa_G^c(T)\right)^2 \\
        &\gtrsim n^{-p\xi / 2} (n^{\underline{p\xi / 2}})^2 r^{-\xi + 2} \\
        \intertext{and using Proposition~\ref{prop:falling-factorial} on the falling factorials gives}
        &\gtrsim r^{-\xi + 2} n^{p \xi / 2},
    \end{align*}
    which diverges by our assumption both when $p$ is even and when $p$ is odd.
\end{proof}

\begin{remark}
    It appears difficult to substantially relax the assumption that $\|A\|_F^2 = n^p$ exactly.
    Indeed, considering terms corresponding to $G$ consisting \emph{only} of Frobenii, if we had a bound $|\|A\|_F^2 - n^p| \leq K$ (or, speaking more roughly, if the typical scale of this difference is $K$) then we see that our bounds would yield $|\EE \kappa_G^c(T)| \leq C^d n^{pd/2} K^{d/2}$, and this would be the best bound we can achieve regardless of the value of $r$.
    Thus even if $K$ is a constant this would not give an $O(1)$ bound on the advantage.
    It would be more natural to assume that $\|A\|_F^2 - n^p$ is centered and $O(n^p)$-subgaussian, say, in our result, but (at least with a naive bounding strategy as we have taken) such an assumption would therefore not suffice to control the above terms when $D = \omega(1)$.
    The same issue arises with relaxing the assumption that the diagonal entries of $A \sim \mu$ are exactly zero.
\end{remark}

\section*{Acknowledgments}
\addcontentsline{toc}{section}{Acknowledgments}

Some of this work was carried out at the Simons Institute for the Theory of Computing and the Banff International Research Station for Mathematical Innovation and Discovery. We are grateful to Denis Bernard, Guy Bresler, Josh Grochow, Sam Hopkins, Nick Read, Tselil Schramm, Guilhem Semerjian, Piotr \'{S}niady, and Dan Spielman for helpful discussions. C.M.\ is especially grateful to Alex Russell for his hospitality during a sabbatical long ago when they learned about free probability and explored diagrammatic methods.

\addcontentsline{toc}{section}{References}
\bibliographystyle{alpha}
\bibliography{main}

\newcommand{\etalchar}[1]{$^{#1}$}
\begin{thebibliography}{COGHK{\etalchar{+}}22}

\bibitem[ABA{\v{C}}13]{ABAC-2013-RandomMatricesComplexity}
Antonio Auffinger, G{\'e}rard Ben~Arous, and Ji{\v{r}}{\'\i} {\v{C}}ern{\`y}.
\newblock Random matrices and complexity of spin glasses.
\newblock {\em Communications on Pure and Applied Mathematics}, 66(2):165--201,
  2013.

\bibitem[ABP73]{Atiyah73}
M.~Atiyah, R.~Bott, and V.K. Patodi.
\newblock On the heat equation and the index theorem.
\newblock {\em Inventiones Math.}, 19:279--330, 1973.

\bibitem[ADGM17]{homotopy}
Anima Anandkumar, Yuan Deng, Rong Ge, and Hossein Mobahi.
\newblock Homotopy analysis for tensor {PCA}.
\newblock In {\em Conference on Learning Theory}, pages 79--104. PMLR, 2017.

\bibitem[AGVP23]{AGVP-2023-FiniteFreeCumulants}
Octavio Arizmendi, Jorge Garza-Vargas, and Daniel Perales.
\newblock Finite free cumulants: multiplicative convolutions, genus expansion
  and infinitesimal distributions.
\newblock {\em Transactions of the American Mathematical Society},
  376(06):4383--4420, 2023.

\bibitem[AHH12]{AHH-2012-RandomTensors}
Andris Ambainis, Aram~W Harrow, and Matthew~B Hastings.
\newblock Random tensor theory: extending random matrix theory to mixtures of
  random product states.
\newblock {\em Communications in Mathematical Physics}, 310(1):25--74, 2012.

\bibitem[AMP16]{AMP-2016-GraphMatrices}
Kwangjun Ahn, Dhruv Medarametla, and Aaron Potechin.
\newblock Graph matrices: norm bounds and applications.
\newblock {\em arXiv preprint arXiv:1604.03423}, 2016.

\bibitem[AP18]{AP-2018-CumulantsFiniteFreeConvolution}
Octavio Arizmendi and Daniel Perales.
\newblock Cumulants for finite free convolution.
\newblock {\em Journal of Combinatorial Theory, Series A}, 155:244--266, 2018.

\bibitem[Ban10]{banica}
Teodor Banica.
\newblock The orthogonal {Weingarten} formula in compact form.
\newblock {\em Letters in Mathematical Physics}, 91(2):105--118, 2010.

\bibitem[BB20]{BB-2020-ReducibilityStatCompGaps}
Matthew Brennan and Guy Bresler.
\newblock Reducibility and statistical-computational gaps from secret leakage.
\newblock In {\em 33rd Annual Conference on Learning Theory (COLT 2020)}, pages
  648--847. PMLR, 2020.

\bibitem[BB24]{BB-2024-FourierRandomGeometricGraphs}
Kiril Bangachev and Guy Bresler.
\newblock On the {Fourier} coefficients of high-dimensional random geometric
  graphs.
\newblock {\em arXiv preprint arXiv:2402.12589}, 2024.

\bibitem[BBH21]{BBH-2021-DeFinettiWishartMatrices}
Matthew Brennan, Guy Bresler, and Brice Huang.
\newblock De {Finetti}-style results for {Wishart} matrices: Combinatorial
  structure and phase transitions.
\newblock {\em arXiv preprint arXiv:2103.14011}, 2021.

\bibitem[BBN20]{BBN-2020-PhaseTransitionsLatentGeometry}
Matthew Brennan, Guy Bresler, and Dheeraj Nagaraj.
\newblock Phase transitions for detecting latent geometry in random graphs.
\newblock {\em Probability Theory and Related Fields}, 178(3):1215--1289, 2020.

\bibitem[BCRT20]{iron-landscapes}
Giulio Biroli, Chiara Cammarota, and Federico Ricci-Tersenghi.
\newblock How to iron out rough landscapes and get optimal performances:
  averaged gradient descent and its application to tensor {PCA}.
\newblock {\em Journal of Physics A: Mathematical and Theoretical},
  53(17):174003, 2020.

\bibitem[BDER16]{BDER-2016-WignerWishartDetection}
S{\'e}bastien Bubeck, Jian Ding, Ronen Eldan, and Mikl{\'o}s~Z R{\'a}cz.
\newblock Testing for high-dimensional geometry in random graphs.
\newblock {\em Random Structures \& Algorithms}, 49(3):503--532, 2016.

\bibitem[BEAH{\etalchar{+}}22]{fp}
Afonso~S Bandeira, Ahmed El~Alaoui, Samuel Hopkins, Tselil Schramm, Alexander~S
  Wein, and Ilias Zadik.
\newblock The {Franz-Parisi} criterion and computational trade-offs in high
  dimensional statistics.
\newblock {\em Advances in Neural Information Processing Systems},
  35:33831--33844, 2022.

\bibitem[BGG{\etalchar{+}}16]{BhattiproluGGLT16}
Vijay V. S.~P. Bhattiprolu, Mrinalkanti Ghosh, Venkatesan Guruswami, Euiwoong
  Lee, and Madhur Tulsiani.
\newblock Multiplicative approximations for polynomial optimization over the
  unit sphere.
\newblock {\em Electron. Colloquium Comput. Complex.}, {TR16-185}, 2016.

\bibitem[BGL17]{BhattiproluGL17}
Vijay Bhattiprolu, Venkatesan Guruswami, and Euiwoong Lee.
\newblock Sum-of-squares certificates for maxima of random tensors on the
  sphere.
\newblock In {\em Approximation, Randomization, and Combinatorial Optimization.
  Algorithms and Techniques, {APPROX/RANDOM} 2017}, volume~81 of {\em LIPIcs},
  pages 31:1--31:20. Schloss Dagstuhl - Leibniz-Zentrum f{\"{u}}r Informatik,
  2017.

\bibitem[BGS13]{BGS-2013-UniversalityPSpinGlass}
Valentin Bonzom, Razvan Gurau, and Matteo Smerlak.
\newblock Universality in $p$-spin glasses with correlated disorder.
\newblock {\em Journal of Statistical Mechanics: Theory and Experiment},
  2013(02):L02003, 2013.

\bibitem[BHK{\etalchar{+}}19]{sos-clique}
Boaz Barak, Samuel Hopkins, Jonathan Kelner, Pravesh~K Kothari, Ankur Moitra,
  and Aaron Potechin.
\newblock A nearly tight sum-of-squares lower bound for the planted clique
  problem.
\newblock {\em SIAM Journal on Computing}, 48(2):687--735, 2019.

\bibitem[BM18]{BM-2018-ReconfigurationConnectivityConstraints}
Nicolas Bousquet and Arnaud Mary.
\newblock Reconfiguration of graphs with connectivity constraints.
\newblock In {\em International Workshop on Approximation and Online
  Algorithms}, pages 295--309. Springer, 2018.

\bibitem[Bol82]{Bollobas-1982-AsymptoticUnlabelledRegularGraphs}
B{\'e}la Bollob{\'a}s.
\newblock The asymptotic number of unlabelled regular graphs.
\newblock {\em Journal of the London Mathematical Society}, 2(2):201--206,
  1982.

\bibitem[Bon24]{bonnin2024universality}
Remi Bonnin.
\newblock Universality of the {W}igner-{G}urau limit for random tensors.
\newblock {\em arXiv preprint arxiv:2404.14144}, 2024.

\bibitem[Bra37]{brauer-37}
Richard Brauer.
\newblock On algebras which are connected with the semisimple continuous
  groups.
\newblock {\em Annals of Mathematics}, 38(4):857--872, 1937.

\bibitem[CDN14]{CDN-2014-AllRealEigenvaluesSymmetricTensors}
Chun-Feng Cui, Yu-Hong Dai, and Jiawang Nie.
\newblock All real eigenvalues of symmetric tensors.
\newblock {\em SIAM Journal on Matrix Analysis and Applications},
  35(4):1582--1601, 2014.

\bibitem[CGL23a]{CGL-2023-TensorHCIZIntegral2}
Beno{\^\i}t Collins, Razvan Gurau, and Luca Lionni.
\newblock The tensor {Harish}-{Chandra}--{Itzykson}--{Zuber} integral {II}:
  Detecting entanglement in large quantum systems.
\newblock {\em Communications in Mathematical Physics}, pages 1--48, 2023.

\bibitem[CGL23b]{CGL-2023-TensorHCIZIntegral1}
Beno{\^\i}t Collins, Razvan~G Gurau, and Luca Lionni.
\newblock The tensor {Harish}-{Chandra}--{Itzykson}--{Zuber} integral {I}:
  {Weingarten} calculus and a generalization of monotone {Hurwitz} numbers.
\newblock {\em Journal of the European Mathematical Society}, 2023.

\bibitem[CMSS06]{collins-mingo-sniady-speicher}
Benoit Collins, James~A. Mingo, Piotr Sniady, and Roland Speicher.
\newblock Second order freeness and fluctuations of random matrices, {III}.
  {Higher} order freeness and free cumulants.
\newblock {\em Documenta Mathematica}, 12, 07 2006.

\bibitem[COGHK{\etalchar{+}}22]{grp-test}
Amin Coja-Oghlan, Oliver Gebhard, Max Hahn-Klimroth, Alexander~S Wein, and
  Ilias Zadik.
\newblock Statistical and computational phase transitions in group testing.
\newblock In {\em Conference on Learning Theory}, pages 4764--4781. PMLR, 2022.

\bibitem[Col03]{Collins-2003-WeingartenHCIZ}
Beno{\^\i}t Collins.
\newblock Moments and cumulants of polynomial random variables on unitary
  groups, the {Itzykson-Zuber} integral, and free probability.
\newblock {\em International Mathematics Research Notices}, 2003(17):953--982,
  2003.

\bibitem[C{\'S}06]{collins-sniady}
Beno{\^\i}t Collins and Piotr {\'S}niady.
\newblock Integration with respect to the haar measure on unitary, orthogonal
  and symplectic group.
\newblock {\em Communications in Mathematical Physics}, 264(3):773--795, 2006.

\bibitem[DdL{\etalchar{+}}22]{DOLTS-2022-Fast3TensorDecomposition}
Jingqiu Ding, Tommaso d'Orsi, Chih-Hung Liu, Stefan Tiegel, and David Steurer.
\newblock Fast algorithm for overcomplete order-3 tensor decomposition.
\newblock {\em arXiv preprint arXiv:2202.06442}, 2022.

\bibitem[DDL23]{graph-matching}
Jian Ding, Hang Du, and Zhangsong Li.
\newblock Low-degree hardness of detection for correlated {Erd\H{o}s-R\'{e}nyi}
  graphs.
\newblock {\em arXiv preprint arXiv:2311.15931}, 2023.

\bibitem[DHS20]{DHS-2020-SpikedMatrixHeavyTailed}
Jingqiu Ding, Samuel~B Hopkins, and David Steurer.
\newblock Estimating rank-one spikes from heavy-tailed noise via self-avoiding
  walks.
\newblock {\em arXiv preprint arXiv:2008.13735}, 2020.

\bibitem[DMW23]{subhypergraph}
Abhishek Dhawan, Cheng Mao, and Alexander~S Wein.
\newblock Detection of dense subhypergraphs by low-degree polynomials.
\newblock {\em arXiv preprint arXiv:2304.08135}, 2023.

\bibitem[Evn21]{Evnin-2021-MelonicDominanceLargestEigenvalueTensor}
Oleg Evnin.
\newblock Melonic dominance and the largest eigenvalue of a large random
  tensor.
\newblock {\em Letters in Mathematical Physics}, 111(3):66, 2021.

\bibitem[Ger31]{gershgorin}
S.~Ger{\v s}gorin.
\newblock {\"U}ber die abgrenzung der eigenwerte einer matrix.
\newblock {\em Bulletin de l'Acad\'emie des Sciences de l'URSS. Classe des
  sciences math\'ematiques et na}, pages 749--754, 1931.

\bibitem[GJJ{\etalchar{+}}20]{GJJPR-2020-SK}
Mrinalkanti Ghosh, Fernando~Granha Jeronimo, Chris Jones, Aaron Potechin, and
  Goutham Rajendran.
\newblock Sum-of-squares lower bounds for {Sherrington-Kirkpatrick} via planted
  affine planes.
\newblock In {\em 61st Annual Symposium on Foundations of Computer Science
  (FOCS 2020)}, pages 954--965. IEEE, 2020.

\bibitem[GJW24]{ld-opt}
David Gamarnik, Aukosh Jagannath, and Alexander~S Wein.
\newblock Hardness of random optimization problems for {Boolean} circuits,
  low-degree polynomials, and {Langevin} dynamics.
\newblock {\em SIAM Journal on Computing}, 53(1):1--46, 2024.

\bibitem[Gur14]{Gurau-2014-UniversalityRandomTensors}
Razvan Gurau.
\newblock Universality for random tensors.
\newblock In {\em Annales de l'IHP Probabilit{\'e}s et statistiques},
  volume~50, pages 1474--1525, 2014.

\bibitem[Gur17]{Gurau-2017-RandomTensors}
R{\u{a}}zvan~Gheorghe Gur{\u{a}}u.
\newblock {\em Random tensors}.
\newblock Oxford University Press, 2017.

\bibitem[Gur20]{Gurau-2020-WignerSemicircleLawTensors}
Razvan Gurau.
\newblock On the generalization of the {Wigner} semicircle law to real
  symmetric tensors.
\newblock {\em arXiv preprint arXiv:2004.02660}, 2020.

\bibitem[GW98]{GoodmanWallach}
Roe Goodman and Nolan~R. Wallach.
\newblock {\em Representations and Invariants of the Classical Groups}.
\newblock Cambridge University Press, 1998.

\bibitem[Hak63]{Hakimi-1963-RealizabilityDegreesGraph2}
Seifollah~Louis Hakimi.
\newblock On realizability of a set of integers as degrees of the vertices of a
  linear graph {II}. {Uniqueness}.
\newblock {\em Journal of the Society for Industrial and Applied Mathematics},
  11(1):135--147, 1963.

\bibitem[Has20]{hastings}
Matthew~B Hastings.
\newblock Classical and quantum algorithms for tensor principal component
  analysis.
\newblock {\em Quantum}, 4:237, 2020.

\bibitem[HKP{\etalchar{+}}17]{sos-detect}
Samuel~B Hopkins, Pravesh~K Kothari, Aaron Potechin, Prasad Raghavendra, Tselil
  Schramm, and David Steurer.
\newblock The power of sum-of-squares for detecting hidden structures.
\newblock In {\em 58th Annual Symposium on Foundations of Computer Science
  (FOCS)}, pages 720--731. IEEE, 2017.

\bibitem[Hop18]{hopkins-thesis}
Samuel Hopkins.
\newblock {\em Statistical inference and the sum of squares method}.
\newblock PhD thesis, Cornell University, 2018.

\bibitem[HR18]{HR-1918-AsymptoticFormulaeCombinatorics}
Godfrey~H Hardy and Srinivasa Ramanujan.
\newblock Asymptotic formula{\ae} in combinatory analysis.
\newblock {\em Proceedings of the London Mathematical Society}, 2(1):75--115,
  1918.

\bibitem[HS17]{HS-bayesian}
Samuel~B Hopkins and David Steurer.
\newblock Efficient {Bayesian} estimation from few samples: community detection
  and related problems.
\newblock In {\em 58th Annual Symposium on Foundations of Computer Science
  (FOCS)}, pages 379--390. IEEE, 2017.

\bibitem[HSS15]{pmlr-v40-Hopkins15}
Samuel~B. Hopkins, Jonathan Shi, and David Steurer.
\newblock Tensor principal component analysis via sum-of-square proofs.
\newblock In Peter Grünwald, Elad Hazan, and Satyen Kale, editors, {\em Proc.
  28th Conference on Learning Theory}, volume~40 of {\em Proceedings of Machine
  Learning Research}, pages 956--1006, 2015.

\bibitem[HSSS16]{HSSS}
Samuel~B. Hopkins, Tselil Schramm, Jonathan Shi, and David Steurer.
\newblock Fast spectral algorithms from sum-of-squares proofs: tensor
  decomposition and planted sparse vectors.
\newblock In {\em Proceedings of the Forty-Eighth Annual ACM Symposium on
  Theory of Computing}, pages 178--–191, 2016.

\bibitem[Iss18]{isserlis}
L.~Isserlis.
\newblock On a formula for the product-moment coefficient of any order of a
  normal frequency distribution in any number of variables.
\newblock {\em Biometrika}, 12:134--139, 1918.

\bibitem[JLM20]{JLM-2020-StatisticalTensorPCA}
Aukosh Jagannath, Patrick Lopatto, and Leo Miolane.
\newblock Statistical thresholds for tensor {PCA}.
\newblock {\em Annals of Applied Probability}, 30(4):1910--1933, 2020.

\bibitem[JP21]{JP-2021-BasesInnerProductPolynomials}
Chris Jones and Aaron Potechin.
\newblock Almost-orthogonal bases for inner product polynomials.
\newblock {\em arXiv preprint arXiv:2107.00216}, 2021.

\bibitem[JPRX23]{CPRX-2023-SOSDensestSubgraph}
Chris Jones, Aaron Potechin, Goutham Rajendran, and Jeff Xu.
\newblock Sum-of-squares lower bounds for densest $k$-subgraph.
\newblock In {\em Proceedings of the 55th Annual ACM Symposium on Theory of
  Computing}, pages 84--95, 2023.

\bibitem[KVWX23]{coloring-clique}
Pravesh Kothari, Santosh~S Vempala, Alexander~S Wein, and Jeff Xu.
\newblock Is planted coloring easier than planted clique?
\newblock In {\em The Thirty Sixth Annual Conference on Learning Theory}, pages
  5343--5372. PMLR, 2023.

\bibitem[KWB22]{KWB-2022-LowDegreeNotes}
Dmitriy Kunisky, Alexander~S Wein, and Afonso~S Bandeira.
\newblock Notes on computational hardness of hypothesis testing: Predictions
  using the low-degree likelihood ratio.
\newblock In Paula Cerejeiras and Michael Reissig, editors, {\em Mathematical
  Analysis, its Applications and Computation}, pages 1--50, Cham, 2022.
  Springer International Publishing.

\bibitem[LMSY22]{LMSY-2022-TestingRandomGeometric}
Siqi Liu, Sidhanth Mohanty, Tselil Schramm, and Elizabeth Yang.
\newblock Testing thresholds for high-dimensional sparse random geometric
  graphs.
\newblock In {\em Proceedings of the 54th Annual ACM SIGACT Symposium on Theory
  of Computing}, pages 672--677, 2022.

\bibitem[Mas14]{Massoulie-2014-CommunityDetectionWeakRamanujan}
Laurent Massouli{\'e}.
\newblock Community detection thresholds and the weak ramanujan property.
\newblock In {\em Proceedings of the forty-sixth annual ACM symposium on Theory
  of computing}, pages 694--703, 2014.

\bibitem[MFC{\etalchar{+}}19]{Maillard_2019}
Antoine Maillard, Laura Foini, Alejandro~Lage Castellanos, Florent Krzakala,
  Marc M{\'e}zard, and Lenka Zdeborov{\'a}.
\newblock High-temperature expansions and message passing algorithms.
\newblock {\em Journal of Statistical Mechanics: Theory and Experiment},
  2019(11):113301, 2019.

\bibitem[Mik20]{Mikulincer-2020-CLTWishartTensors}
Dan Mikulincer.
\newblock A {CLT} in {Stein}'s distance for generalized {Wishart} matrices and
  higher order tensors.
\newblock {\em arXiv preprint arXiv:2002.10846}, 2020.

\bibitem[MR95]{molloy-reed}
Michael Molloy and Bruce Reed.
\newblock A critical point for random graphs with a given degree sequence.
\newblock {\em Random Structures \& Algorithms}, 6(2-3):161--180, 1995.

\bibitem[MR11]{MR-2011-GraphIntegralCircuit}
Cristopher Moore and Alexander Russell.
\newblock A graph integral formulation of the circuit partition polynomial.
\newblock {\em Combinatorics, Probability \& Computing}, 20(6):911, 2011.

\bibitem[MR14]{richard-montanari}
Andrea Montanari and Emile Richard.
\newblock A statistical model for tensor {PCA}.
\newblock In {\em Proceedings of the 27th International Conference on Neural
  Information Processing Systems - Volume 2}, NIPS'14, page 2897–2905. MIT
  Press, 2014.

\bibitem[MS17]{mingo-speicher}
James~A. Mingo and Roland Speicher.
\newblock {\em Free Probability and Random Matrices}, volume~35 of {\em Fields
  Institute Monographs}.
\newblock Springer, 2017.

\bibitem[MW19]{MW-2019-SpectralTensorNetworks}
Ankur Moitra and Alexander~S Wein.
\newblock Spectral methods from tensor networks.
\newblock In {\em 51st Annual ACM SIGACT Symposium on Theory of Computing (STOC
  2019)}, pages 926--937, 2019.

\bibitem[NS06]{NS-2006-LecturesCombinatoricsFreeProbability}
Alexandru Nica and Roland Speicher.
\newblock {\em Lectures on the combinatorics of free probability}, volume~13.
\newblock Cambridge University Press, 2006.

\bibitem[N{\'S}11]{Novak2011WhatIF}
Jonathan Novak and Piotr {\'S}niady.
\newblock What is\ldots a free cumulant?
\newblock {\em Notices of the American Mathematical Society}, 58:300--301,
  2011.

\bibitem[OR20]{OR-2020-TensorPCATraceInvariants}
Mohamed Ouerfelli and Vincent Rivasseau.
\newblock A new framework for tensor {PCA} based on trace invariants.
\newblock 2020.

\bibitem[OTR22]{Ouerfelli-Tamaazousti-Rivasseau-2022}
Mohamed Ouerfelli, Mohamed Tamaazousti, and Vincent Rivasseau.
\newblock Random tensor theory for tensor decomposition.
\newblock {\em Proceedings of the AAAI Conference on Artificial Intelligence},
  36(7):7913--7921, 2022.

\bibitem[Oue22]{Ouerfelli-2022-TensorPCA}
Mohamed Ouerfelli.
\newblock {\em New perspectives and tools for Tensor Principal Component
  Analysis and beyond}.
\newblock PhD thesis, Universit{\'e} Paris-Saclay, 2022.

\bibitem[PR22]{PR-2022-SOSPCA}
Aaron Potechin and Goutham Rajendran.
\newblock Sub-exponential time sum-of-squares lower bounds for principal
  components analysis.
\newblock {\em Advances in Neural Information Processing Systems},
  35:35724--35740, 2022.

\bibitem[Pro76]{Procesi-1976-InvariantTheoryMatrices}
Claudio Procesi.
\newblock The invariant theory of $n \times n$ matrices.
\newblock {\em Advances in mathematics}, 19(3):306--381, 1976.

\bibitem[QCC18]{QCC-2018-TensorEigenvalues}
Liqun Qi, Haibin Chen, and Yannan Chen.
\newblock {\em Tensor eigenvalues and their applications}, volume~39.
\newblock Springer, 2018.

\bibitem[Qi05]{Qi-2005-EigenvaluesRealSupersymmetricTensor}
Liqun Qi.
\newblock Eigenvalues of a real supersymmetric tensor.
\newblock {\em Journal of Symbolic Computation}, 40(6):1302--1324, 2005.

\bibitem[Qi07]{Qi-2007-EigenvaluesInvariantsTensors}
Liqun Qi.
\newblock Eigenvalues and invariants of tensors.
\newblock {\em Journal of Mathematical Analysis and Applications},
  325(2):1363--1377, 2007.

\bibitem[RRS17a]{RSS}
Prasad Raghavendra, Satish Rao, and Tselil Schramm.
\newblock Strongly refuting random {CSPs} below the spectral threshold.
\newblock In {\em Proceedings of the 49th Annual ACM SIGACT Symposium on Theory
  of Computing}, pages 121--131, 2017.

\bibitem[RRS17b]{raghavendra-rao-schramm}
Prasad Raghavendra, Satish Rao, and Tselil Schramm.
\newblock Strongly refuting random {CSP}s below the spectral threshold.
\newblock In {\em Proceedings of the 49th Annual ACM SIGACT Symposium on Theory
  of Computing}, STOC 2017, page 121–131, 2017.

\bibitem[RSWY23]{count-communities}
Cynthia Rush, Fiona Skerman, Alexander~S Wein, and Dana Yang.
\newblock Is it easier to count communities than find them?
\newblock In {\em 14th Innovations in Theoretical Computer Science Conference
  (ITCS 2023)}. Schloss-Dagstuhl-Leibniz Zentrum f{\"u}r Informatik, 2023.

\bibitem[Sem24]{semerjian2024matrix}
Guilhem Semerjian.
\newblock Matrix denoising: Bayes-optimal estimators via low-degree
  polynomials.
\newblock {\em arXiv preprint arxiv:2402.16719}, 2024.

\bibitem[Sub17]{Subag-2017-ComplexitySecondMoment}
Eliran Subag.
\newblock The complexity of spherical $p$-spin models — a second moment
  approach.
\newblock {\em The Annals of Probability}, 45(5):3385--3450, 2017.

\bibitem[SW22]{SW-2020-LowDegreeEstimation}
Tselil Schramm and Alexander~S Wein.
\newblock Computational barriers to estimation from low-degree polynomials.
\newblock {\em The Annals of Statistics}, 50(3):1833--1858, 2022.

\bibitem[WAM19]{wein-elalaoui-moore}
Alexander~S. Wein, Ahmed~El Alaoui, and Cristopher Moore.
\newblock The {Kikuchi} hierarchy and tensor {PCA}.
\newblock In {\em 60th {IEEE} Annual Symposium on Foundations of Computer
  Science, {FOCS} 2019}, pages 1446--1468. {IEEE} Computer Society, 2019.

\bibitem[Wei23]{Wein-2022-LowDegreeTensorDecomposition}
Alexander~S Wein.
\newblock Average-case complexity of tensor decomposition for low-degree
  polynomials.
\newblock In {\em Proceedings of the 55th Annual ACM Symposium on Theory of
  Computing (STOC 2023)}, pages 1685--1698, 2023.

\bibitem[Wen88]{wenzl}
Hans Wenzl.
\newblock On the structure of {B}rauer's centralizer algebras.
\newblock {\em Annals of Mathematics}, 128(1):173--193, 1988.

\bibitem[Wey46]{Weyl46}
Hermann Weyl.
\newblock {\em The Classical Groups: Their Invariants and Representations}.
\newblock 1946.

\bibitem[Wic50]{wick}
G.~C. Wick.
\newblock The evaluation of the collision matrix.
\newblock {\em Physical Review}, 80:268--272, 1950.

\bibitem[Wil99]{Will-1999-SwitchingDistanceGraphs}
Todd~G Will.
\newblock Switching distance between graphs with the same degrees.
\newblock {\em SIAM Journal on Discrete Mathematics}, 12(3):298--306, 1999.

\bibitem[WZ24]{WZ-2024-PowerIterationTensorPCA}
Yuchen Wu and Kangjie Zhou.
\newblock Sharp analysis of power iteration for tensor {PCA}.
\newblock {\em arXiv preprint arXiv:2401.01047}, 2024.

\bibitem[ZJ09]{ZinnJustin-2009-JucysMurphyElementsWeingarten}
Paul Zinn-Justin.
\newblock {Jucys}-{Murphy} elements and {Weingarten} matrices.
\newblock {\em arXiv preprint arXiv:0907.2719}, 2009.

\end{thebibliography}

\clearpage

\appendix

\section{Characterizations of Invariants}
\label{app:invariants}

Our goal in this Appendix is to give diagrammatic proofs of Theorems~\ref{thm:tensor-invariant}, \ref{thm:tensor-invariant-multi}, and \ref{thm:tensor-equivariant}.
We first discuss the first two Theorems on invariant polynomials.
Recall that we say that a function $f:\Sym^p(\RR^n) \to \RR$ is \emph{invariant} if, for all $T \in \Sym^p(\RR^n)$ and all $Q \in \sO(n)$, $f(Q \cdot T) = f(T)$.
Similarly, a function $f: \Sym^{p_1}(\RR^n) \times \Sym^{p_2}(\RR^n) \times \ldots \times \Sym^{p_m}(\RR^n) \to \RR$ is invariant if, for any symmetric tensors $T_1,\ldots,T_m$ of arity $p_1,\ldots,p_m$ respectively, $f(T_1,\ldots,T_m) = f(Q \cdot T_1,\ldots,Q \cdot T_m)$ for all $Q \in \sO(n)$.

A priori, restricting $T$ to $\Sym^p (\RR^n)$ could weaken the definition of invariance. There are indeed polynomials that are invariant on symmetric tensors but not on general tensors: for instance, for $2 \times 2$ matrices, the function $m_{11}^2 + 2 m_{12}^2 + m_{22}^2$ is the Frobenius norm $\|M\|_F^2$ and hence invariant, but only if $M$ is symmetric. However, the following proposition shows that we can always replace such a polynomial with one that is invariant on general tensors by symmetrizing over $\sO(n)$. In what follows we will identify invariant polynomials with this symmetrized version.

\begin{proposition}
\label{prop:sym-invariant-extension}
Suppose $f$ is an invariant polynomial. Then there is a polynomial $g$ of the same degree which coincides with $f$ on $\Sym^p(\RR^n)$ and which is invariant on all tensors, i.e., $g(Q \cdot T)=g(T)$ for all $T \in (\RR^n)^{\otimes p}$.
\end{proposition}

\begin{proof}
Let $g(T) = \Exp_Q f(Q \cdot T)$, where the expectation is over the Haar measure on $\sO(n)$.
\end{proof}

Next, we show Theorem~\ref{thm:tensor-invariant}, which says that any invariant polynomial is a linear combination of graph moments. This generalizes a similar fact for matrices: the invariant polynomials of a matrix are precisely the symmetric polynomials of its eigenvalues. These in turn can be written in terms of spectral moments, or equivalently traces of $T$'s matrix powers.

\begin{theorem}
\label{thm:matrix-invariant}
    Let $R$ be the ring of polynomials $f: \Sym^2(\RR^n) \to \RR$ in the entries of a matrix so that, for any orthogonal matrix $Q$, $f(T) = f(Q^\top T Q)$.
    Then $R$ is generated by the polynomials $m_\ell(T) = \tr(T^\ell)$ for $\ell \geq 0$. In particular, any invariant homogeneous polynomial $f(T)$ of degree $d$ is a linear combination of graph moments $\sum_i \alpha_i m_{G_i}(T)$ where each $G_i$ is a disjoint union of cycles with a total of $d$ vertices.
\end{theorem}

\begin{example} For instance, $f(T) = (\tr T)^2 \tr (T^2) \tr (T^3)$, which is a polynomial of degree $7$, is $m_G(T)$ where $G$ is the 2-regular multigraph \raisebox{-7pt}{\includegraphics[height=18pt]{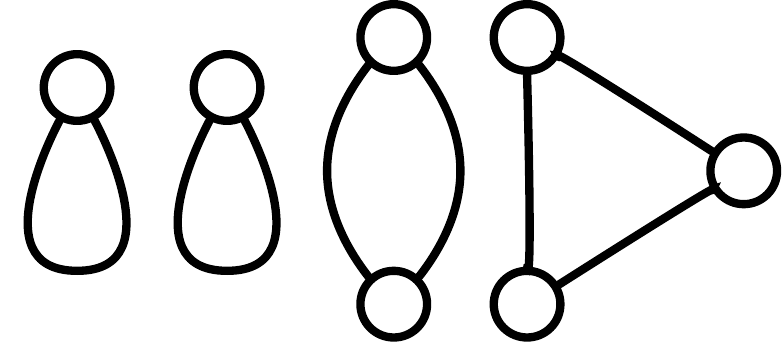}}.
\end{example}

\begin{proof}[Proof of Theorem~\ref{thm:tensor-invariant}]
We illustrate the proof in Figure~\ref{fig:invariant}. First note that any polynomial $f$ of degree $d$ in a $p$-ary tensor $T$ can be written as an inner product between $T^{\otimes d}$ and a vector of coefficients $C$, which we can also view as a $dp$-ary tensor.

Now, if $f$ is invariant, it remains the same if we place a copy of any orthogonal matrix $Q^\top$ on each of the $dp$ edges of $T^{\otimes d}$. But this is equivalent to applying $Q^{\otimes dp}$ to $C$: that is,
\[
f(Q^\top \cdot T)
= \langle Q^\top \cdot T^{\otimes d}, C \rangle
= \langle T^{\otimes d}, Q \cdot C \rangle
= \langle T^{\otimes d}, Q^{\otimes dp} C \rangle
= \langle T^{\otimes d}, C \rangle
= f(T) \, .
\]
Since $f(T) = \langle T^{\otimes d}, Q^{\otimes dp} C \rangle$ for any $Q \in \sO(n)$, we can symmetrize $C$ by taking the expectation over the Haar measure, obtaining
\begin{equation}
\label{eq:symmetrized}
f(T) = \langle T^{\otimes d}, \Pi_{dp} C \rangle \, ,
\end{equation}
where
\begin{equation}
\label{eq:pi-ell}
\Pi_\ell = \Exp_{Q \in \sO(n)} Q^{\otimes \ell} \, ,
\end{equation}
the projection operator that we discussed in Section~\ref{sec:prelim:weingarten}.
Specifically, $\Pi_{\ell}$ projects onto the \emph{trivial subspace} under the action of $\sO(n)$, i.e., the set of vectors $w \in (\RR^n)^{\otimes \ell}$ such that $Q^\top \cdot w = Q^{\otimes \ell} w = w$ for all $Q \in \sO(n)$.
As we present there, it follows from representation theory that $\Pi_{\ell}$ may be written $\Pi_{\ell} = \sum_{\mu, \nu} \Wg_{\mu, \nu} w(\mu) \otimes w(\nu)$ for some $\Wg_{\mu, \nu} \in \RR$, $\mu, \nu$ perfect matchings of $[\ell]$, and $w(\mu)$ the indicator vector of pairs of indices being equal under the matching $\mu$.

Since the image of $\Pi_\ell$ is spanned by the $w(\mu)$, the symmetrized coefficients $\Pi_{dp} C$ in~\eqref{eq:symmetrized} are a linear combination $\sum_\mu \alpha_\mu w(\mu)$. But for each $\mu$, $\langle T^{\otimes d}, w(\mu) \rangle$ is exactly a graph moment $m_{G(\mu)}(T)$, where $G(\mu)$ is the multigraph $G(\mu)$ formed by matching the $dp$ half-edges of its vertices according to $\mu$. (The reader may be familiar with the configuration model of random graphs~\cite{molloy-reed} where $\mu$ is uniformly random.) Thus
\[
f(T) = \sum_\mu \alpha_\mu m_{G(\mu)}(T) \, .
\]
as illustrated in Figure~\ref{fig:invariant}.
This completes the proof for invariant polynomials $f(T)$ of a single $p$-ary tensor.
\end{proof}

\begin{proof}[Proof of Theorem~\ref{thm:tensor-invariant-multi}]
For invariant polynomials of multiple tensors $f(T_1,\ldots,T_m)$, if $p$ is homogeneous of degree $d_i$ in each $T_i$ and each $T_i$ has arity $k_i$, we can write it as an inner product
\[
f(T_1,\ldots,T_m)
= \left\langle \bigotimes_{i=1}^m T_i^{\otimes d_i} ,
C \right\rangle \, ,
\]
where the tensor of coefficients $C$ has total arity $\ell = \sum_i d_i k_i$. The argument then goes through as before; see Figure~\ref{fig:mixed-example-matching} for an example.
\end{proof}

\begin{figure}
    \centering
\includegraphics[width=\columnwidth]{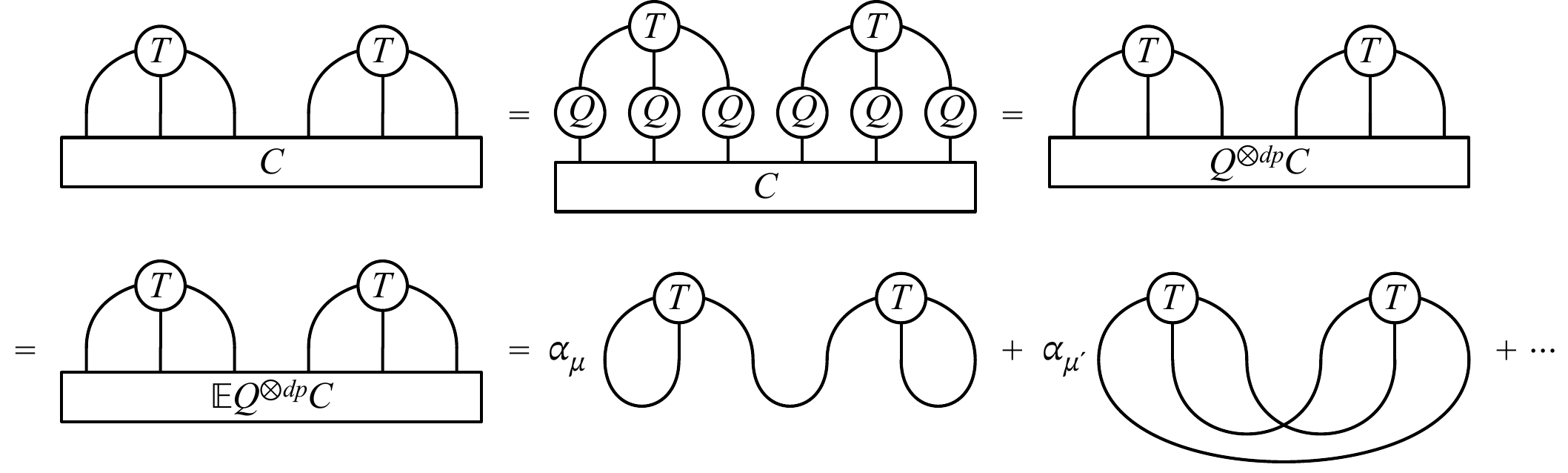}
\caption{The proof of Theorem~\ref{thm:tensor-invariant}. Any homogeneous polynomial of degree $d$ can be written as the inner product of $T^{\otimes d}$ with a $dp$-ary tensor of coefficients $C$. By hypothesis the function is unchanged if we symmetrize these coefficients by conjugating them with a Haar-random orthogonal matrix $Q$. But their image under the projection operator $\Exp Q^{\otimes dp}$ is a linear combination of matching vectors, each of which induces a multigraph connecting the copies of $T$.}
\label{fig:invariant}
\end{figure}

\begin{figure}
    \centering
\includegraphics[width=1.8in]{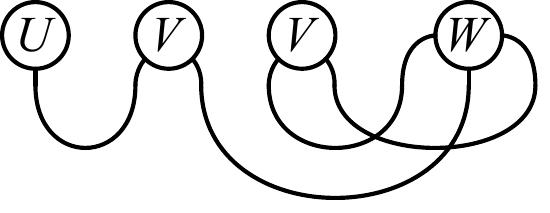}
    \caption{Generalizing the argument of Theorem~\ref{thm:tensor-invariant} to mixed moments. The inner product of $U \otimes V \otimes V \otimes W$ with the matching vector shown gives the multigraph on the right in Figure~\ref{fig:mixed-example} and the mixed moment~\eqref{eq:mixed-example}.}
    \label{fig:mixed-example-matching}
\end{figure}

\noindent
Again by similar arguments, we may also prove Theorem~\ref{thm:tensor-equivariant} on equivariant polynomials and open graph moments.

\begin{proof}[Proof of Theorem~\ref{thm:tensor-equivariant}]
Let $W$ be an $\ell$-ary tensor of indeterminates. Then if $f(T)$ is $\ell$-ary and equivariant, the inner product $P = \langle f(T), W \rangle$ is an invariant function of $T$ and $W$. Similarly, if $f(T_1,\ldots,T_m)$ is equivariant, then $F = \langle f(T_1,\ldots,T_m), W \rangle$ is an invariant function of $\{ T_1,\ldots,T_m,W \}$.

If $f(T)$ is a homogeneous polynomial, then $F(T,W)$ is a homogeneous polynomial of degree $1$ in $W$, i.e., which is multilinear in $W$'s entries. By Theorem~\ref{thm:tensor-invariant}, $F$ is a linear combination of mixed moments $m_G(T,W)$ where $W$ corresponds to a single vertex $w$ of degree $\ell$ and $G$'s other vertices have degree $p$. Then each entry of $f(T)$ is a partial derivative of $F(T,W)$ by the corresponding entry of $W$,
\[
f(T)_{i(E')} =
\frac{\partial F(T,W)}
{\partial W_{i(E')}}  \, .
\]
Taking this partial derivative removes $W$'s vertex from $G$, leaving a $p$-regular open multigraph $G \setminus \{w\}$ with $|F|=\ell$ open edges. This completes the proof for $f(T)$, and mutatis mutandis for $f(T_1,\ldots,T_m)$.
\end{proof}

\section{Properties of Wigner Tensors}
\label{app:wigner}

\subsection{Basic Properties}

We give some further properties of $T \sim \Wig(p, n, \sigma^2)$.
As a result of symmetrizing over all permutations $\pi$, the entries of $T$ have different variances depending on the pattern of repetitions in their indices. Specifically, their variance is equal to the size of their stabilizer subgroup, i.e., the number of permutations in $S_k$ which leave them unchanged.

\begin{proposition}
\label{prop:wig-variance}
Let $T \sim \Wig(p,n,\sigma^2)$. For a given sequence of indices $i=(i_1,\ldots,i_p) \in [n]^p$ and $j \in [n]$, let $c_j(i)$ be the number of times $j$ appears in $i$. Then
\[
T_i \sim \sN\left( 0, \sigma^2 \prod_{j=1}^n c_j(i)! \right) \, .
\]
\end{proposition}

\begin{proof}
Write $m = \prod_{j=1}^n c_j(i)!$. Then $m$ is the number of permutations $\pi$ such that $\pi(i)=i$, and the number of distinct $\pi(i)$ contributing to $T_i$ in~\eqref{eq:wig-def} is $p!/m$. For each of these the same $G_{\pi(i)}$ appears $m$ times in~\eqref{eq:wig-def}, multiplying its variance by $m^2$. Since the $G_{\pi(i)}$ are independent, we then have $\Var T_i = (\sigma^2 / p!) \times (p! / m) \times m^2 = \sigma^2 m$.
\end{proof}

This ensemble is orthogonally invariant, and for $p=2$ we recover the Gaussian orthogonal ensemble of random matrix theory.

\begin{proposition}
    For $T \sim \Wig(p, n, \sigma^2)$, for any $Q \in \sO(n)$, $Q \cdot T$ has the same law as $T$. In particular, $T \sim \Wig(2, n, 1)$ has the law of the standard $n \times n$ Gaussian orthogonal ensemble, with $T_{ij} = T_{ji} \sim \sN(0, 1 + \delta_{ij})$.
\end{proposition}

\subsection{Hardness of Computing Moments}

In this section, we give the proof of Theorem~\ref{thm:wigner-moments-hard} on the hardness of computing the expected graph moments of a Wigner tensor.
In the course of developing the tools for the proof, we will find various useful reinterpretations of the quantities involved in the moments.

\begin{definition}
    Let $V$ be a finite set, $p: V \to \NN$ a \emph{degree sequence}, and $\sG(V, p)$ the set of graphs (with loops and parallel edges allowed) on $V$ such that each $v \in V$ has degree $p(v)$.
    $G, H \in \sG(V, p)$ are related by a \emph{switching} if there is a pair of distinct edges $\{v, w\}, \{v^{\prime}, w^{\prime}\} \in E(G)$ such that $H$ is formed by replacing both of these edges with the edges $\{v, v^{\prime}\}, \{w, w^{\prime}\}$.
    We view $\sG(V, p)$ as a graph, where $G \sim H$ if $G$ and $H$ are related by a switching.

    The \emph{switching distance} between $G$ and $H$, denoted $d_{\switch}(G, H)$, is the distance between $G$ and $H$ in $\sG(V, p)$, or equivalently the minimum number of switchings required to reach $H$ from $G$.
\end{definition}
\noindent
We note that the degree sequence is preserved by the switching operation, so we must restrict our attention to a fixed degree sequence for the switching distance to be defined.
Even so, it is not obvious that the switching distance is finite, but this is indeed the case.
Also, we emphasize that, unlike with the sets $\sG_{d, p}$ in the main text, we are working with \emph{labelled} graphs here.

\begin{proposition}[Lemma 1 of \cite{Hakimi-1963-RealizabilityDegreesGraph2}]
    Whenever $\sG(V, p)$ is non-empty, then it is connected.
    Thus, $d_{\switch}(G, H) < \infty$ for any $G, H \in \sG(V, p)$.
\end{proposition}

\begin{remark}[Loops and parallel edges]
    We will rely on the results of the works \cite{Hakimi-1963-RealizabilityDegreesGraph2,Will-1999-SwitchingDistanceGraphs,BM-2018-ReconfigurationConnectivityConstraints}, which variously work with not $\sG(V, p)$ but the induced subgraph on the vertex set consisting either of simple graphs or loopless multigraphs (with parallel edges allowed).
    However, implicit in their results are the following facts:
    \begin{enumerate}
        \item The induced subgraphs of $\sG(V, p)$ on both simple graphs and on loopless multigraphs are connected.
        \item There is a path between simple graphs (respectively, loopless multigraphs) $G, H \in \sG(V, p)$ of minimum length passing through only simple graphs (respectively, loopless multigraphs); that is, the distance between $G$ and $H$ in the induced subgraph on simple graphs (respectively, loopless multigraphs) equals the distance in $\sG(V, p)$.
    \end{enumerate}
    Thus we will rephrase their results over $\sG(V, p)$ without further comment.
\end{remark}

Moreover, we may always reduce switching distance computations from a sequence of possibly large degrees to the case $p \equiv 1$, which corresponds to sets of \emph{perfect matchings}.

\begin{definition}
    For $G \in \sG(V, p)$, let $V^{\prime}$ be a set of size $\sum_{v \in V} p(v)$, which we view as consisting of $v^{\prime}_{a,b}$ for $a \in V$ and $b \in [p(a)]$.
    We say that a perfect matching $\mu \in \sG(V^{\prime}, 1)$ \emph{realizes} $G$ if $G$ is obtained from identifying $v^{\prime}_{a,1}, \dots, v^{\prime}_{a, p(a)}$ into a single vertex for each $a$.
\end{definition}
\noindent
In other words, $\mu$ realizes $G$ if $G$ would be obtained from $\mu$ under the configuration model; this is in somewhat different language the same definition used in the main text in constructing the graph Weingarten function.

\begin{proposition}
    \label{prop:switch-matching}
    $d_{\switch}(G, H) = \min\{d_{\switch}(\mu_G, \mu_H): \mu_G \text{ realizes } G, \mu_H \text{ realizes } H\}$.
\end{proposition}

\begin{definition}
    For $G$ and $H$ defined on the same vertex set, we write $G \triangle H$ for the graph containing the symmetric difference of the edge sets $G$ and $H$, where if $G$ has $a$ edges between $i$ and $j$ and $H$ has $b$ edges between $i$ and $j$, then $G \triangle H$ has $|a - b|$ edges between $i$ and $j$.
    We view this graph as being colored, where, in the above setting, if $a > b$ then the edges between $i$ and $j$ are blue, and otherwise they are red.

    A \emph{symmetric circuit} in such a colored graph is a closed walk of even length alternating between red and blue edges.
    A \emph{symmetric circuit partition} is a partition of the edges into symmetric circuits.
    We denote the number of circuits in the largest symmetric circuit partition by $\circuit(G, H)$.
\end{definition}

The following useful fact shows that the (minimum) switching distance is essentially equivalent to the maximum circuit partition.

\begin{proposition}[Theorem 2.5 of \cite{Will-1999-SwitchingDistanceGraphs}, Theorem 22 of \cite{BM-2018-ReconfigurationConnectivityConstraints}]
    \label{prop:dswitch-circ}
    $d_{\switch}(G, H) = \frac{1}{2}|E(G \triangle H)| - \circuit(G, H)$.
\end{proposition}
\noindent
The proofs cited above treat simple and loopless graphs, respectively, but the same argument extends straightforwardly to graphs with loops as well.
One way to see this is to use Proposition~\ref{prop:switch-matching} to rephrase the computation of $d_{\switch}(G, H)$ as a similar computation of matchings realizing $G$ and $H$, which are simple graphs and to which the results of \cite{Will-1999-SwitchingDistanceGraphs} apply.

\begin{corollary}
    \label{cor:switching-distance-bound}
    $d_{\switch}(G, H) \leq |E(G)| = |E(H)|$.
\end{corollary}
\begin{proof}
    The result follows since $|E(G \triangle H)| = 2|E(G - H)| \leq 2|E(G)|$.
\end{proof}

We now move towards relating the switching distance to the quantities involved in the Wigner moments; in particular, we will relate it to the exponent $c_{\max}(G)$, which we recall satisfied:
\begin{equation}
    \Ex_{W \sim \Wig} m_G(W) \sim n^{c_{\max}(G)}.
\end{equation}

\begin{definition}
    For $p$ a constant and $|V|$ even, denote by $\sF = \sF(V, p) \subset \sG(V, p)$ the subset of multigraphs that consist of a disjoint union of Frobenii of degree $p$.
\end{definition}

Our first main result is that the exponent $c_{\max}(G)$ may be computed through the minimum switching distance to a set of Frobenii.
\begin{lemma}
    \label{prop:cmax-dswitch}
    For any $p$-regular $G$, $c_{\max}(G) = |E(G)| - d_{\switch}(G, \sF)$.
\end{lemma}
\noindent
We note that, while $\sF$ and $\sG$ refer to sets of labelled graphs, the quantities being computed here do not depend on the labelling, since $\sF$ is invariant under permutations of the vertex labels.
\begin{proof}
    Let $v = |V(G)|$, identify $V(G)$ with $[v]$, and let $\mu$ be a perfect matching realizing $G$.
    View the vertex set of $\mu$ as $[v] \times [p]$, so that $(v, 1), \dots, (v, p)$ are the ``expanded'' vertex set corresponding to $v \in V$.
    To any even edge coloring of $G$ we may associate a perfect matching $\kappa$ of $[v]$ and perfect matchings $\eta_{\{i, j\}}$ of $[p]$ for each $\{i, j\} \in \kappa$, so that vertices matched in $\kappa$ have the same edge colors in their neighborhoods, and $\eta_{\{i, j\}}$ is a matching between half-edges of the same color incident with $i$ and with $j$.
    We may view $\mu^{\prime} \colonequals \bigsqcup_{\{i, j\} \in \kappa} \eta_{\{i, j\}}$ as a matching on the same vertex set as $\mu$.
    The number of colors in the given edge coloring of $G$ is at most the number of cycles into which $\mu \sqcup \mu^{\prime}$ decomposes; conversely, there is an even edge coloring of $G$ with precisely this number of edge colors, formed by assigning a different color to the edges of $G$ corresponding to the edges of $\mu$ lying in each cycle.

    But, the possible $\mu^{\prime}$ described above are precisely the matchings that realize a graph of ($p$-regular) Frobenii on $[v]$.
    Thus $c_{\max}(G)$ is equivalently the maximum number of cycles into which $\mu \sqcup \mu^{\prime}$ decomposes for any $\mu^{\prime}$ realizing a graph of Frobenii on $[v]$; moreover, the same holds for any $\mu$ realizing $G$.
    Let us write $\cyc(\mu \sqcup \mu^{\prime})$ for this quantity.
    We have shown so far that
    \begin{equation}
        \label{eq:cmax-cyc}
        c_{\max}(G) = \max\{\cyc(\mu \sqcup \mu^{\prime}): \mu \text{ realizes } G, \mu^{\prime} \text{ realizes some } F \in \sF\}.
    \end{equation}

    Now, separating the 2-cycles in $\mu \sqcup \mu^{\prime}$, which correspond to edges shared between $\mu$ and $\mu^{\prime}$, from longer cycles and using Proposition~\ref{prop:dswitch-circ}, we have
    \begin{align*}
         \cyc(\mu \sqcup \mu^{\prime})
         &= \cyc(\mu \triangle \mu^{\prime}) + |\mu \cap \mu^{\prime}| \\
         &= \circuit(M, M^{\prime}) + \left( |E(\mu)| - \frac{1}{2}|E(\mu \triangle \mu^{\prime})|\right) \\
         &= |E(G)| - \left(\frac{1}{2}|E(\mu \triangle \mu^{\prime})| - \circuit(\mu, \mu^{\prime})\right) \\
         &= |E(G)| - d_{\switch}(\mu, \mu^{\prime}).
    \end{align*}
    Finally, substituting, we find
    \begin{align*}
        c_{\max}(G)
        &= |E(G)| - \min\{d_{\switch}(\mu, \mu^{\prime}): \mu \text{ realizes } G, \mu^{\prime} \text{ realizes some } F \in \sF\} \\
        &= |E(G)| - d_{\switch}(G, \sF),
    \end{align*}
    completing the proof.
\end{proof}

We also learn some interesting structural facts about maximum even edge colorings from this proof, using the interpretation of $c_{\max}(G)$ in \eqref{eq:cmax-cyc}.
\begin{proposition}
    \label{prop:max-coloring-unique-within-nbd}
    In any maximum even edge coloring, for any vertex, all edges adjacent to that vertex have distinct colors.
\end{proposition}
\begin{proof}
    Suppose otherwise.
    By \eqref{eq:cmax-cyc}, we then have a matching $\mu$ of $[d] \times [p]$ realizing $G$ and another matching $\mu^{\prime}$ realizing a disjoint union of Frobenii.
    Suppose without loss of generality that there are two edges incident with vertex 1 and having the same color, i.e., two edges in $\mu$, touching vertices among $(1, 1), \dots, (1, p)$, and belonging to the same cycle in $\mu \sqcup \mu^{\prime}$.
    Suppose again without loss of generality that these latter vertices are $(1, 1)$ and $(1, 2)$.
    Under $\mu^{\prime}$, all vertices with first coordinate 1 are matched to vertices with some other first coordinate $i$.
    So, suppose $(1, 1)$ is matched with $(i, j)$ and $(1, 2)$ with $(i, k)$.
    Define $\mu^{\prime\prime}$ by instead matching $(1, 1)$ with $(i, k)$ and $(1, 2)$ with $(i, j)$.
    Then, the cycle containing these edges in $\mu \sqcup \mu^{\prime}$ is broken into two cycles, while all other cycles are unchanged.
    Thus, the original coloring must not have had the maximum number of colors, and we reach a contradiction.
\end{proof}

\begin{proposition}
    \label{prop:max-coloring-unique-across-nbd}
    In any maximum even edge coloring, every colored neighborhood occurs exactly twice.
\end{proposition}
\begin{proof}
    A similar argument applies in this case as well.
    In the matching interpretation, if there is a colored neighborhood occurring four (or more) times in a coloring, then the same collection of $p$ cycles in $\mu \sqcup \mu^{\prime}$ must pass through the vertices $(i_a, 1), \dots, (i_a, p)$ for $a = 1, 2, 3, 4$ and some choice of $i_1, i_2, i_3, i_4 \in [d]$.
    Suppose without loss of generality that the $(i_1, j)$ are matched to the $(i_2, j)$ and the $(i_3, j)$ to the $(i_4, j)$ in $\mu^{\prime}$ (if the cycles only pass through four of the same collections of vertices such a pairing must occur; if there are more such collections then one may find such pairs by following the matching of $\mu^{\prime}$.
    Then, it is possible to instead match either the $(i_1, j)$ to the $(i_3, j)$ or the $(i_4, j)$ such that one of the $p$ cycles is broken into two (by matching its endpoints differently), and no two of the other cycles are merged (by matching their endpoints to one another arbitrarily without joining two distinct cycles).
    Again, the total number of cycles must increase by at least 1, contradicting the maximality of the original choice of $\mu^{\prime}$.
\end{proof}

As a corollary of this result we learn that $w_{\max}(G)$, the sum of the ``weights'' of even edge colorings with $c_{\max}(G)$ colors as defined in the main text, is actually just a counting problem of the number of non-isomorphic even edge colorings with this number of colors; the weights necessarily always equal 1.

Finally, to prove Theorem~\ref{thm:wigner-moments-hard}, we will use the connection between $c_{\max}$ and the switching distance that we have developed as well as the following result.

\begin{theorem}[Theorem 3.2 of \cite{Will-1999-SwitchingDistanceGraphs}]
    It is NP-hard to decide whether $d_{\switch}(G, H) \geq d$ given $d \geq 0$ and simple graphs $G, H$ on the same vertex set and having the same degree sequence.
\end{theorem}

\begin{remark}
    It may be tempting to try to use the cumulants we have defined to compute $m_H$ and its Wigner expectation, and thereby to compute or estimate $c_{\max}$ and so a switching distance.
    The issue with such a strategy appears to be that computing with the relationship between the cumulants and the graph moments requires forming the graph Weingarten function $\Wg_{G, H}$, which involves a summation over the matchings realizing $G$ and $H$.
    There are exponentially many of these matchings, and indeed enumerating them would allow one to compute the switching distance by brute force.
\end{remark}

\begin{proof}[Proof of Theorem~\ref{thm:wigner-moments-hard}]
    We will show that if it is possible to compute $c_{\max}(G)$, then it is also possible to compute switching distances between simple graphs.
    Namely, let $G, H$ be simple graphs on a shared vertex set $V$ and with a shared degree sequence $p: V \to \NN$.
    We will show that computing $d_{\switch}(G, H)$ may be encoded in the computation of $c_{\max}(J)$ for a certain graph $J = J(G, H)$.
    By Proposition~\ref{prop:cmax-dswitch}, this is equivalent to computing $d_{\switch}(J, \sF)$.
    Let $v = |V|$, identify $V$ with $[v]$, and set $e = \frac{1}{2}\sum_{v \in V}p(v) = |E(G)| = |E(H)|$.

    We form $J$ on the vertex set $\{1, \dots, v, 1^{\prime}, \dots, v^{\prime}\}$.
    For $i, j \in \{1, \dots, v\}$, we draw one edge between $i$ and $j$ in $J$ if $i \sim j$ in $G$.
    Likewise, for $i^{\prime}, j^{\prime} \in \{1^{\prime}, \dots, v^{\prime}\}$, we draw one edge between $i^{\prime}$ and $j^{\prime}$ in $J$ if $i \sim j$ in $H$.
    Finally, for each $i \in [v]$, we draw $3e - p(i)$ edges between $i$ and $i^{\prime}$ in $J$.
    Thus, $J$ is a disjoint union of one copy of $G$ and one copy of $H$, along with a ``very heavy matching'' with many repeated edges between corresponding vertices in $G$ and $H$ under their joint labelling.
    Note also that $J$ is $3e$-regular.
    The basic idea is that this heavy matching will force the nearest disjoint union of Frobenii to correspond to the matching of vertices with the same labels in $G$ and $H$.

    More formally, let $F_0$ be the graph of Frobenii on the heavy matching in $J$.
    We claim that $d_{\switch}(J, \sF) = d_{\switch}(J, F_0)$; that is, that $F_0$ is a minimizer of the switching distance of $J$ to any disjoint union of Frobenii.
    Indeed, we have $d_{\switch}(J, F_0) \leq d_{\switch}(G, H) + e \leq 2e$, because to reach $F_0$ we may first transform $G$ into (a copy of) $H$ by switchings, and then align pairs of corresponding edges in the two copies of $H$ with $F_0$ one at a time.
    (We also use the inequality $d_{\switch}(G, H) \leq e$ from Corollary~\ref{cor:switching-distance-bound}.)
    On the other hand, for any $F \in \sF \setminus \{F_0\}$, to reach $F$ from $J$ we must change at least the $3e - p(i)$ edges between $i$ and $i^{\prime}$ for some $i \in v$.
    Since $3e - p(i) \geq 3e - e \geq 2e$, we have $d_{\switch}(J, F) \geq 2e$, proving the claim.

    Finally, we claim that the inequality mentioned above is tight, that is, that
    \begin{equation}
        d_{\switch}(J, \sF) = d_{\switch}(J, F_0) = d_{\switch}(G, H) + e.
    \end{equation}
    After showing this the proof will be complete.

    By Proposition~\ref{prop:dswitch-circ}, we have $d_{\switch}(J, F_0) = \frac{1}{2}|E(J \triangle F_0)| - \circuit(J, F_0)$.
    Let us view the edges in $J \triangle F_0$ coming from $J$ as ``red'' and those coming from $F_0$ as ``blue.''
    Then, by our construction of $J$, we have that $J \triangle F_0$ consists of two disjoint red copies of $G$ and $H$, together with blue matchings between corresponding vertices in $G$ and $H$, where the edge between $i$ and $i^{\prime}$ is repeated $p(i)$ times.
    In particular, we have $|E(J \triangle F_0)| = 2e + \sum_{i = 1}^v p(i) = 4e$, so $d_{\switch}(J, F_0) = 2e - \circuit(J, F_0) = e + (e - \circuit(J, F_0))$.
    Thus it suffices to show that $d_{\switch}(G, H) = e - \circuit(J, F_0)$.
    Further, since $d_{\switch}(G, H) = |E(G - H)| - \circuit(G, H) = e - |E(G \cap H)| - \circuit(G, H)$, it also suffices to show $\circuit(J, F_0) = |E(G \cap H)| + \circuit(G, H)$.

    Given a symmetric circuit partition of $G \triangle H$, we may produce one of $J \triangle F_0$ by traversing one blue edge (between $G$ and $H$) between every edge of the given partition, and also adding circuits of length 4 including each pair of edges shared between $G$ and $H$.
    This shows $\circuit(J, F_0) \geq |E(G \cap H)| + \circuit(G, H)$.

    Conversely, we claim that, given a symmetric circuit partition of $J \triangle F_0$, there is another partition of at least the same size that contains circuits of length 4 including each pair of edges shared between $G$ and $H$.
    Suppose $\{i, j\}$ is such an edge.
    Then, edges $\{i, j\}$ and $\{i^{\prime}, j^{\prime}\}$ must belong to different circuits $C, C^{\prime}$ of the partition of $J \triangle F_0$.
    In $C$, $\{i, j\}$ must be surrounded by two blue edges $m_1, m_2$ of the heavy matching, and likewise in $C^{\prime}$, $\{i^{\prime}, j^{\prime}\}$ must be surrounded by two blue edges $m_1^{\prime}, m_2^{\prime}$.
    Then, we may form another symmetric circuit partition by replacing $C, C^{\prime}$ with $\{\{i, j\}, m_1, \{i^{\prime}, j^{\prime}\} m_2\}, (C \setminus \{\{i, j\}, m_1\}) \sqcup (C^{\prime} \setminus \{\{i^{\prime}, j^{\prime}\}, m_1^{\prime}\})$ which is of the same size.

    So, there is a maximum symmetric circuit partition of $J \triangle F_0$ that contains all circuits of length 4 including pairs of edges shared between $G$ and $H$.
    The remaining circuits in such a partition correspond to a symmetric circuit partition of $G \triangle H$.
    Thus we have the opposite inequality $\circuit(J, F_0) \leq |E(G \cap H)| + \circuit(G, H)$ as well, and the proof is complete.
\end{proof}

Despite this hardness result, it is still possible to give some tractable and general bounds on the exponent $c_{\max}(G)$.

\begin{proposition}
    For any simple $p$-regular $G$ on $d$ vertices,
    \begin{equation}
        c_{\max}(G)
        \leq \frac{(p + 1)d}{4}.
    \end{equation}
    Equality holds if and only if $G$ has a \emph{transitive perfect matching} $\mu$: a matching such that, whenever $\{v, v^{\prime}\}, \{w, w^{\prime}\} \in \mu$ and $v \sim w$, then $v^{\prime} \sim w^{\prime}$ as well.
\end{proposition}
\noindent
It is an interesting simplification of the problem of computing $c_{\max}(G)$ to check whether the condition for this bound being saturated holds or not.
When $p = 3$, one may check that a transitive perfect matching exists if and only if each connected component of $G$ is isomorphic to one of two graphs built from a union of two cycles: if we imagine forming a cubical complex from this graph by adding a 2-cell between the ``lateral'' edges connecting edges of the transitive perfect matching, then the resulting surface must either be a cylinder without its top and bottom or the same with a ``twist,'' i.e., a M\"{o}bius band.
Once $p \geq 4$, such a simple classification seems elusive, and it is unclear whether to expect the problem to be easy or hard.

\section{Conditioning of Weingarten Matrices}
\label{app:weingarten}

We have discussed in Section~\ref{sec:prelim:weingarten} the structure of the projection operator $\Pi = \Pi_\ell = \Exp_Q Q^{\otimes \ell}$.
As stated there, $\Pi_\ell$ projects onto the space spanned by the matching vectors $w(\mu)$ defined in~\eqref{eq:matching-vector} where $\mu$ is a perfect matching of $[\ell]$.
Thus the projection may be written
\begin{equation}
    \Pi_{\mu, \nu} = \sum_{\mu, \nu} (M^{-1})_{\mu,\nu} \, w(\mu) \otimes w(\nu),
\end{equation}
where $M^{-1}$ is the inverse of the Gram matrix
\[
M_{\mu, \nu} = \langle w(\mu), w(\nu) \rangle
\, ,
\]
or the Moore--Penrose pseudoinverse if the $w(\mu)$ are overcomplete and $M$ is not of full rank. For matching vectors, the Gram matrix is given by
\begin{equation}
\label{eq:matching-gram}
M_{\mu, \nu}
= \langle w_\mu, w_\nu \rangle
= n^{\textrm{\# of cycles in $\mu \sqcup \nu$}}
= n^{\ell/2 - \Delta(\mu,\nu)} \, .
\end{equation}
Here by $\mu \sqcup \nu$ we mean the 2-regular multigraph with $\ell$ vertices whose edges are the pairs in $\mu$ and $\nu$, and each of these cycles gives a factor of $n$. We define $\Delta(\mu,\nu)$ as the minimum number of swaps it takes to transform $\mu$ to $\nu$, where a swap changes two pairs in a matching from $\{(a,b),(c,d)\}$ to $\{(a,c),(b,d)\}$ or $\{(a,d),(b,c)\}$. The relation
\[
\textrm{\# of cycles in $\mu \sqcup \nu$}
= \ell/2 - \Delta(\mu,\nu)
\]
follows because $\mu \sqcup \nu$ has $\ell/2$ cycles if and only if $\mu=\nu$, and each swap on a shortest path from $\mu$ to $\nu$ merges two cycles into one.
See Figure~\ref{fig:matching-gram} for an example.

The inverse of $M$ is called the \emph{Weingarten function} and is denoted $\Wg$. Thus
\[
\Pi = \sum_{\mu,\nu} \Wg_{\mu,\nu} \,w(\mu) \otimes w(\nu) \, ,
\]
where the sum is over all perfect matchings $\mu,\nu$ of $[\ell]$.
By permutation symmetry, $\Wg_{\mu, \nu} = (M^{-1})_{\mu, \nu}$ is a function only of the cycle structure of $\mu \sqcup \nu$, i.e., the number of cycles of each length that this graph has.
For this reason $\Wg$ is often called a \emph{function}, but for our purposes we will view it as a matrix and seek to understand its spectrum, or equivalently the spectrum of $M$.

\begin{figure}
    \centering
    \includegraphics[width=3in]{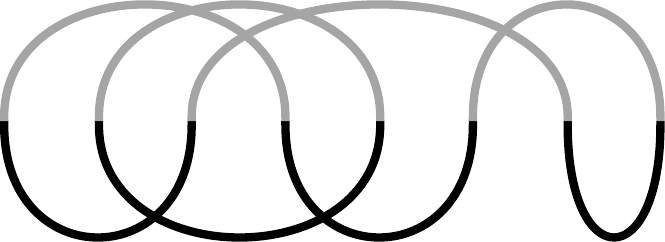}
    \caption{The inner product of two matching vectors $w(\mu), w(\nu)$ as defined in~\eqref{eq:matching-vector} is $n$ raised to the number of cycles in their disjoint union $\mu \sqcup \nu$. In this example, $\ell=8$ and there are two cycles in $\mu \sqcup \nu$, one of length $2$ and one of length $6$. Thus $\langle w(\mu), w(\nu) \rangle = n^2$ and $\mu$ and $\nu$ are $\Delta(\mu,\nu)=2$ swaps apart.}
    \label{fig:matching-gram}
\end{figure}

For sufficiently large $n$, $M$ is dominated by its diagonal, and a representation-theoretic argument~\cite{brauer-37,wenzl} shows that $M$ has full rank whenever $n \ge \ell/2$. However, we need the stronger property that its smallest eigenvalue is bounded above zero. Using a simple counting argument we will show that this holds, and moreover that the matrix is well-conditioned, whenever $n > (1 + \eps)\ell^2$.

\begin{proposition}
\label{prop:gram-bound}
Suppose that $n > \ell^2$. Then the Gram matrix $M$ defined in~\eqref{eq:matching-gram} has full rank, and all of its eigenvalues lie in the interval $[n^{\ell/2} (1 - \ell^2/n), n^{\ell/2} (1 + \ell^2/n)]$.
\end{proposition}

\begin{proof}
Since there are only $\binom{\ell}{2}$ possible swaps, for any matching $\mu$ there are at most $\binom{\ell}{2}^t \le (\ell^2/2)^t$ matchings $\nu$ such that $\Delta(\mu,\nu)=t$. Therefore, if we write
\[
M = n^{\ell/2} (\id + H )
\quad \text{where} \quad
H_{\mu \nu} =
\begin{cases}
0 & \mu=\nu \\
n^{-\Delta(\mu,\nu)} & \mu \ne \nu \, ,
\end{cases}
\]
then if $n > \ell^2$ the sum of any row $H_\mu$ of $H$ gives a bound on the operator norm of $H$,
\[
|H|
\le \sum_{\nu \ne \mu} H_{\mu \nu} \le \sum_{t=1}^\infty \left(\frac{\ell^2}{2n}\right)^{\!t}
\le \ell^2/n \, .
\]
The Ger\v{s}gorin circle theorem~\cite{gershgorin}
then implies that $M$ has full rank whenever $n > \ell^2$ and that its eigenvalues lie in the stated interval.
\end{proof}

Since the eigenvalues of $\Wg$ are the reciprocals of those of $M$, Proposition~\ref{prop:gram-bound} implies that $\Wg$ is well-conditioned.
\begin{corollary}
\label{cor:wg}
Suppose that $n \ge 2 \ell^2$. Then, all of the eigenvalues of $(\Wg_{\mu, \nu})$ lie in the interval $[\frac{1}{2}n^{-\ell/2}, 2 n^{-\ell/2}]$.
\end{corollary}

\begin{remark}
\label{rem:sniady}
Proposition~\ref{prop:gram-bound} and Corollary~\ref{cor:wg} are essentially tight, since when $n = o(\ell^2)$ then the largest eigenvalue of $\Wg$ is roughly $n^{-\ell/2} \,\e^{\ell^2/2n}$.
We outline the argument, which we learned of from a personal communication with Piotr \'Sniady.
Per Proposition 5 of \cite{ZinnJustin-2009-JucysMurphyElementsWeingarten} (originally due to \cite{Collins-2003-WeingartenHCIZ}), the distinct eigenvalues of $\Wg$ correspond to Young diagrams $\lambda$ of $\ell$ boxes with an even number of boxes in each row.\footnote{Actually, this correspondence could give another more explicit treatment of the conditioning results derived above, but we prefer to give a self-contained treatment using the simpler Ger\v{s}gorin circle theorem.}
We write $\lambda / 2$ for the same diagram where the number of boxes in each row is halved.
We view these diagrams as sets of $(i, j)$, the ``coordinates'' of the boxes of the diagram, with $i, j \geq 0$ so that the box in the top left corner has coordinates $(0, 0)$, the one below coordinates $(1, 0)$, the one to the right coordinates $(0, 1)$, and so forth.
Then, the eigenvalue of $\Wg$ associated to $\lambda$ is $1 / (\prod_{(i, j) \in \lambda / 2} (n + j - i))$.
The largest eigenvalue of $\Wg$ is then the one corresponding to $\lambda$ two columns of $\ell / 2$ boxes each, which gives
\begin{equation}
    \frac{1}{\prod_{i = 0}^{\ell / 2 - 1} (n - i)} = n^{-\ell / 2} \prod_{i = 0}^{\ell / 2 - 1} \frac{1}{1 - \frac{i}{n}} \approx n^{-\ell / 2} \exp\left(\frac{\ell^2}{2n}\right),
\end{equation}
as claimed.
\end{remark}

In the main text we mostly encounter not this Weingarten function over matchings, but the ``graph Weingarten function'' that we define indexed by $\sG_{d, p}$ (Definition~\ref{def:graph-weingarten}).
This has entries
    \begin{equation}
        \Wg_{G, H} = \sum_{\substack{\mu \text{ realizes } G \\ \nu \text{ realizes } H}} \Wg_{\mu, \nu}
    \end{equation}
for $G, H \in \sG_{d, p}$.
We also recall from Proposition~\ref{prop:number-realizing-matchings} that the number of summands over each axis is
\begin{equation}
    \#\{\mu \text{ realizes } G\} = \frac{p!^d d!}{|\eAut(G)|}.
\end{equation}
Using this, we state a corollary on the spectrum of this compressed version of $\Wg$ in the form that will be useful in the main text.
\begin{corollary}
    \label{cor:wg-graph}
    Suppose that $pd \leq \sqrt{n / 2}$ is even.
    Then the symmetric matrix indexed by $G, H \in \sG_{d, p}$ with entries
    \begin{equation}
        n^{pd / 2} \cdot \frac{\sqrt{|\eAut(G)| \cdot |\eAut(H)|}}{p!^d d!} \cdot \Wg_{G, H}
    \end{equation}
    has all its eigenvalues lying in the interval $[\frac{1}{2}, 2]$.
\end{corollary}
\begin{proof}
    Write $\ell \colonequals pd$.
    Define a matrix $\widetilde{J}$ indexed by $G \in \sG_{d, p}$ and $\mu$ perfect matchings of $[pd]$ with
    \begin{equation}
        \widetilde{J}_{G, \mu} \colonequals \sqrt{\frac{|\eAut(G)|}{p!^d d!}} \One\{\mu \text{ realizes } G\} = \frac{1}{\sqrt{\#\{\mu \text{ realizes } G\}}} \One\{\mu \text{ realizes } G\}.
    \end{equation}
    By construction, we have $\widetilde{J} \, \widetilde{J}^{\top} = \id$.
    On the other hand, calling $X$ the matrix in the claim, we have
    \begin{equation}
        X = \widetilde{J} (n^{\ell / 2} \Wg) \widetilde{J}^{\top}.
    \end{equation}
    By Corollary~\ref{cor:wg}, we then have
    \begin{equation}
        X \preceq 2 \widetilde{J}\widetilde{J}^{\top} = 2\id,
    \end{equation}
    and the lower bound follows similarly.
\end{proof}

By an identical proof and using Proposition~\ref{prop:number-realizing-matchings-chop}, we also have the following variation for the Weingarten matrix appearing for open multigraphs.
\begin{corollary}
    \label{cor:wg-graph-chop}
    Suppose that $pd - 1 \leq \sqrt{n / 2}$ is even.
    Then the symmetric matrix indexed by $G, H \in \sG_{d, p \to}$ with entries
    \begin{equation}
        n^{(pd - 1) / 2} \cdot \frac{\sqrt{|\eAut(\chop(G))| \cdot |\eAut(\chop(H))|}}{p!^{d - 1}(p - 1)! (d - 1)!} \cdot \Wg_{\chop(G), \chop(H)}
    \end{equation}
    has all its eigenvalues lying in the interval $[\frac{1}{2}, 2]$.
\end{corollary}

As an aside, let us mention that computing $\Wg$ explicitly is a notoriously difficult problem. We can get a sense of how its entries scale by writing it as a geometric series,
\begin{equation}
\label{eq:wg-geosum}
\Wg = M^{-1}
= n^{-\ell/2} (\id+H)^{-1}
= n^{-\ell/2} \sum_{t=0}^\infty (-1)^t H^t \, ,
\end{equation}
and expand this into a sum over paths through the space of perfect matchings. For a given pair of matchings $\mu, \nu$, let $\{ \mu \sim \nu \}$ denote the set of paths $\sigma$ where $\mu = \sigma_0 \ne  \sigma_1 \ne \cdots \ne \sigma_t = \nu$ for some $t \ge 0$, and write $|\sigma|=t$. Define $\Delta(\sigma)$ as $\sigma$'s total length in swap distance,
\[
\Delta(\sigma) = \sum_{i=1}^t \Delta(\sigma_i,\sigma_{i-1}) \, ,
\]
and define $g(\sigma)$ as $\sigma$'s \emph{geodesic defect}, i.e., the difference between its length and the shortest-path distance between its endpoints,
\[
g(\sigma) = \Delta(\sigma) - \Delta(\mu,\nu) \ge 0 \, .
\]
Then (following \cite{collins-sniady,banica}) we have
\begin{align}
\Wg_{\mu \nu}
&= n^{-\ell/2}
\sum_{\sigma \in \{\mu \sim \nu\}} (-1)^{|\sigma|} \,n^{-\Delta(\sigma)} \\
&= n^{-\ell/2-\Delta(\mu,\nu)}
\sum_{\sigma \in \{\mu \sim \nu\}}(-1)^{|\sigma|} \,n^{-g(\sigma)} \, .
\end{align}
Thus for graphs of constant size we have $\Wg_{\mu \nu} = O(n^{-\ell/2-\Delta(\mu,\nu)})$.
The prefactor for a given $\Wg_{\mu \nu}$, i.e., the sum of signed geodesics from $\mu$ to $\nu$, might grow rapidly with $\ell$.

For instance, for $D=4$ where there are $3$ perfect matchings, the rows of $M$ and $\Wg$ where $\mu=\{(1,2),(3,4)\}$ are
\[
\begin{array}{c|ccc}
\nu & \{(1,2),(3,4)\} & \{(1,3),(2,4)\} & \{(1,4),(2,3)\} \\
M_{\mu,\nu} &
n^2 & n & n \\
\Wg_{\mu,\nu} & \frac{n+1}{n(n-1)(n+2)}
& \frac{-1}{n(n-1)(n+2)}
& \frac{-1}{n(n-1)(n+2)}
\end{array}
\]

\section{Combinatorial Bounds}

We gather some auxiliary combinatorial results that are used in the main text.

\begin{proposition}
    \label{prop:falling-factorial}
    Suppose $k \leq n / 2$.
    Then, $n^k \exp(-k^2 / n) \leq n^{\underline{k}} \leq n^k$.
\end{proposition}
\begin{proof}
    The upper bound is immediate.
    For the lower bound, we have:
    \begin{align*}
        n^{\underline{k}}
        &= n^k \cdot 1 \cdot \left(1 - \frac{1}{n}\right) \cdot \left(1 - \frac{k - 1}{n}\right) \\
        &= n^k \exp\left(\sum_{i = 0}^{k - 1} \log\left(1 - \frac{i}{n}\right)\right)
        \intertext{and here, noting that $\log(1 - x) \geq -2x$ for all $0 \leq x \leq \frac{1}{2}$, we have}
        &\geq n^k \exp\left(\frac{2}{n}\sum_{i = 0}^{k - 1} i\right) \\
        &\geq n^k \exp\left(-\frac{k^2}{n}\right),
    \end{align*}
    as claimed.
\end{proof}

\begin{proposition}
    \label{prop:chains}
    A \emph{chain} of subsets of $[n]$ is a sequence of strict inclusions $\emptyset \subsetneq A_1 \subsetneq A_2 \subsetneq \cdots \subsetneq A_k \subsetneq [n]$.
    The number of chains of subsets of $[n]$ is at most $3^n n!$.
\end{proposition}
\begin{proof}
    Call $C(n)$ the number of chains on $[n]$.
    The chains are in bijection with the ordered set partitions of $[n]$, where a chain as defined in the statement corresponds to the partition $(A_1, A_2 \setminus A_1, \cdots, [n] \setminus A_k)$.
    Every ordered set partition consists recursively of a choice of a non-empty set and an ordered set partition on the remaining elements.
    Thus, $C(0) = 1$ and for $n \geq 1$ we have
    \begin{equation}
        C(n) = \sum_{i = 1}^n \binom{n}{i} C(n - i) = n C(n - 1) + \sum_{i = 2}^n \binom{n}{i} C(n - i).
    \end{equation}
    We proceed by strong induction.
    Clearly the claim holds for $n = 0$.
    Suppose that the claim holds for all $C(m)$ with $0 \leq m \leq n - 1$.
    Then,
    \begin{align*}
        C(n)
        &\leq 3^{n - 1}n! + \sum_{i = 2}^n \binom{n}{i} 3^{n - i}(n - i)! 3^{n - i} \\
        &= 3^n n!\left(\frac{1}{3} + \sum_{i = 2}^n \frac{3^{-i}}{i!}\right) \\
        &\leq 3^n n!\left(\frac{1}{3} + \exp\left(\frac{1}{3}\right) - 1\right) \\
        &< 3^n n!,
    \end{align*}
    completing the proof.
\end{proof}

\begin{proposition}[Section 2 of \cite{HR-1918-AsymptoticFormulaeCombinatorics}]
    \label{prop:integer-partitions}
    An \emph{integer partition} of $d \geq 1$ is a sequence of $1 \leq a_1 \leq \cdots \leq a_{\ell}$ such that $a_1 + \cdots + a_{\ell} = d$.
    There is an absolute constant $C > 0$ so that the number of integer partitions of $d$ is at most $\frac{1}{d}\exp(C\sqrt{d})$ for all $d \geq 1$.
\end{proposition}

\end{document}